\documentclass[11pt]{article}
\usepackage{hyperref}
 \usepackage{color}
\usepackage{amsmath}
\usepackage{amssymb}

\usepackage{amsthm}

\usepackage{epstopdf}
\usepackage{graphicx}
\usepackage{mathtools}
\usepackage[usenames,dvipsnames, table]{xcolor}

\usepackage[hang]{caption}

\usepackage{oldgerm}  
\usepackage{pdflscape}
\usepackage{afterpage}
\usepackage{capt-of}
\usepackage{float}

\def\bee{\begin{enumerate}}\def\eee{\end{enumerate}}
\def\bei{\begin{itemize}}\def\eei{\end{itemize}}
\newcommand{\nco}{\newcommand}
\def\CC{{\cal C}}
\def\R{\mathbb{R}}\def\C{\mathbb{C}}
\nco{\red}{\color{red}}
\nco{\blue}{\color{blue}}
\nco{\cyan}{\color{cyan}}
\nco{\brown}{\color{Magenta}}
\def\Blue#1{\blue #1\normalcolor}
\def\Red#1{\red #1\normalcolor}

\nco{\magenta}{\normalcolor}

\nco{\violet}{\color{violet}}
\nco{\redend}{\normalcolor}
\nco{\magentaend}{\normalcolor}
\def\inv#1{\frac{1}{#1}}

\def\card#1{|#1|}
\def\tr{{\rm tr}\,}
\def\ie{{\it i.e. }}
\def\ommit#1{{}}
\def\({\left(}
\def\){\right)}
\def\oh{\frac{1}{2}}
\def\ie{{\it i.e.,\/}\ }
\def\ie{{\rm i.e.,\/}\ }

\def\Diff{\mathrm{Diff}}
\def\dual{{swap\ }}
\def\duality{{swapping\ }}
\def\colour{colour}

\def\be{\begin{equation}}\def\ee{\end{equation}}
\def\bea{\begin{eqnarray}}\def\eea{\end{eqnarray}}

\def\pamatrix#1{\begin{pmatrix}#1\end{pmatrix}}
\def\smat#1{\mbox{\footnotesize{\mbox{$\begin{pmatrix}#1\end{pmatrix}$}}}}

\nco{\rnc}{\renewcommand}
 \rnc{\title}[1]{{\Large\bf\mbox{}\\\medskip#1\bigskip\medskip\\}}
 \rnc{\author}[1]{{\large #1\smallskip\\}}
 \nco{\address}[1]{{\em #1\medskip\\}}

\def\g{{\tt g}}\def\Gg{\gamma}

\newtheorem{mytheo}{Theorem}
\newtheorem{myprop}{Proposition}

\begin{document}
%\markboth{Robert Coquereaux and Jean-Bernard Zuber} 
%{Maps, Immersions and Permutations}

%%%%%%%%%%%%%%%%%%%%% Publisher's Area please ignore %%%%%%%%%%%%%%
%\catchline{}{}{}{}{}
%%%%%%%%%%%%%%%%%%%%%%%%%%%%%%%%%%%%%%%%%%%%%%%%%%%%%%%%%%%%%%%

\begin{titlepage}
\begin{center}
\title{Maps, immersions and permutations}
\medskip
\author{Robert Coquereaux} 
\address{Aix Marseille Universit\'e, Universit\'e de Toulon, CNRS, CPT, UMR 7332, 13288 Marseille, France\\
Centre de Physique Th\'eorique (CPT)}
\centerline{and}
\medskip
\author{Jean-Bernard Zuber}
\address{
 Sorbonne Universit\'es, UPMC Univ Paris 06, UMR 7589, LPTHE, F-75005, 
Paris, France\\
\& CNRS, UMR 7589, LPTHE, F-75005, Paris, France
 }
\bigskip\medskip

\today

\begin{abstract}
\noindent {We consider the problem of  counting 
and of listing  topologically inequivalent ``planar" 
{4-valent} maps with a single component and a given number $n$ of vertices. 
This enables us to count and to tabulate immersions of a circle in a sphere (spherical curves),
{extending results by Arnold and followers}. 
Different options where the circle and/or the sphere are/is oriented are considered in turn,
following Arnold's classification of the different types of symmetries. We also consider 
the case of {\it bi{\colour}able} and bi{\colour}ed maps {or} immersions, where {\it faces} are bicoloured. 
Our method  extends to immersions of a circle in a higher genus 
Riemann surface.  There the bi{\colour}ability is no longer automatic and has to be assumed.
We thus have two separate countings in non zero genus, that of bicolourable maps 
and that of general  maps.
\\
We use a classical method of encoding 
maps in terms of permutations,  on which the constraints of ``one-componentness" and of 
a given genus may be applied. Depending on the orientation issue
and on the bicolourability assumption,
permutations  for a map with $n$ vertices live in $S_{4n}$ or in
$S_{2n}$. 

In a nutshell, our method reduces to the counting (or listing) of orbits of certain subset of $S_{4n}$ (resp. $S_{2n}$)
under the action of the centralizer of a certain element of $S_{4n}$ (resp. $S_{2n}$). This is achieved either by
appealing to a formula by Frobenius or by a direct enumeration of these orbits.\\
Applications to knot theory are briefly mentioned.} 
\end{abstract}
\end{center}

\vspace*{40mm}

\vspace*{05mm}

%\vspace*{40mm}
\end{titlepage}

%%%%%%%%%%%%%%%%%%%%%%%%%%%%%%%%%%%%%%%%%%%%%%
%%%%%%%%%%%%%%%%%%%%%%%%%%%%%%%%%%%%%%%%%%%%%%
\section{Introduction}

In the present paper we are interested in the problem of enumerating the (equivalence classes of) curves with $n$ double points that one can draw on the sphere  or on an orientable surface of genus $g$. Such curves may be regarded as the images of immersions of
a circle in that  surface.
In the present paper we shall be mainly dealing with immersions in a compact surface.
See Fig.\ \ref{imm2} for an illustration of the difference between immersions in the plane or in the sphere. 
The problem is tightly connected with the census of {\it knots} and {\it virtual knots} 
but we shall content ourselves with brief comments about knots. Recall that for non zero genus, (virtual) 
knot diagrams or drawings of
curves exhibit {\it virtual crossings} in addition to their regular crossings: the former may be regarded as artifacts due to
the projection on the plane of the figure, see \cite{Kauffman} and the example of Fig.\ \ref{virtual2}.\\
Since the curves  (or the knot diagrams),  like the surfaces themselves, can be oriented, the discussion and the results will naturally split into four cases (OO, UO, OU, UU) that we define now: 
a genus $g$ curve is the image of the circle under an immersion $S^1 \to \Sigma$,  the latter being an orientable surface of genus $g$, 
and,   if the curve is not simple (if it crosses itself),  the multiple points of the immersion should be double points
 with distinct tangents (in other words we consider generic closed curves). 
When $g = 0$, this is called a generic spherical curve.  Both the circle and the surface can be oriented. If $S^1$ is not oriented, one may consider the sets UU and UO of $\Diff (\Sigma)$-equivalent and  $\Diff^{+}(\Sigma)$-equivalent unoriented curves. If $S^1$ is oriented, one considers the sets OU and OO of $\Diff(\Sigma)$-equivalent and  $\Diff^{+}(\Sigma)$-equivalent oriented curves.  $\Diff^{+} (\Sigma)$ denotes the group of  orientation-preserving diffeomorphisms of the oriented surface $\Sigma$.
 For spherical curves,   these four types of immersions have been considered by previous authors,  \cite{Arnold, Arnold2, OEIS}.  Correspondingly, in knot theory
one may consider knots up to mirror symmetry, and oriented or unoriented. \\  
 Following Carter \cite{Carter}, who coined this adjective in the UU case, one says that two immersed curves are OO, UO, OU or UU  geotopic if they are equivalent in the previous sense.  Now it is clear that the operation of adding handles to a surface in which a circle is immersed defines immersed curves in higher genus surfaces. 
It is therefore natural to consider the following definition \cite{Carter}: two immersed curves are stably geotopic if and only if there is a collection of handles that can be added to either surface, or both, in such a way that the curves become geotopic on the resulting surfaces. 
In this paper we assume that the studied immersions are cellular, in the sense that the complement of each associated immersed curve is homeomorphic to a collection of open disks,
so that the classifications obtained in this paper when $g>0$ for the different kinds of immersions should always be understood up to stable geotopy (although this will not be in general repeated in the text).  In other words the genus given in our tables for an immersed curve with a given number of crossings is such that the chosen curve cannot be immersed in a surface of smaller genus. One could then use surgeries to obtain a classification for all generic immersed curves (see \cite{Bourgoin}).
\\
On top of  the  question of orientation, we introduce the issue of {\it bicolourability} of the curve. 
 By definition a curve is bicolourable if one can assign opposite colors to adjacent faces.
While in genus 0, any self intersecting curve may be 
bi{\colour}ed (with adjacent faces of opposite {\colour}s), it is no longer true for higher genus,
see Fig.\ \ref{virtual2} for an example. Moreover when a curve is bi{\colour}ed, the two possible {\colour}ings may be or not (topologically) equivalent, see Fig.\ \ref{imm2-swp}. 
This bi{\colour}ing is quite natural in the context of knot theory, where it amounts to considering
the curve as an {\it alternating} knot, see Fig.\ \ref{sh-faces} below; 
there are two ways of doing that, which may or not  lead to equivalent knots. 
We shall thus append a suffix $c$, $b$ or no suffix at all to the symbols OO, OU, etc: 
OOc will refer to (inequivalent) bi{\colour}ings of  immersions of an oriented circle in an oriented surface, OOb to bicolourable (but not
bi{\colour}ed) immersions, and OO alone to general, bicolourable or not, immersions. Likewise
for immersions of type UO, OU and UU. This results in $3\times 4=12$ different types of immersions,
and the reader who is eager to see numbers may jump to Tables \ref{TableXYZ1}, \ref{TableXYZ2} to
see their cardinals tabulated up to 10 crossings for all genera. The reader who prefers figures to
numbers is directed to Fig.\ \ref{ImmUU8} and \ref{ImmUU9v3a}-\ref{ImmUU9v3b} for a complete list of indecomposable irreducible spherical
curves of UU type, with  respectively $n=8$ and 9 crossings.

 \begin{figure}[bthp]
 \centering{\includegraphics[width=16pc]{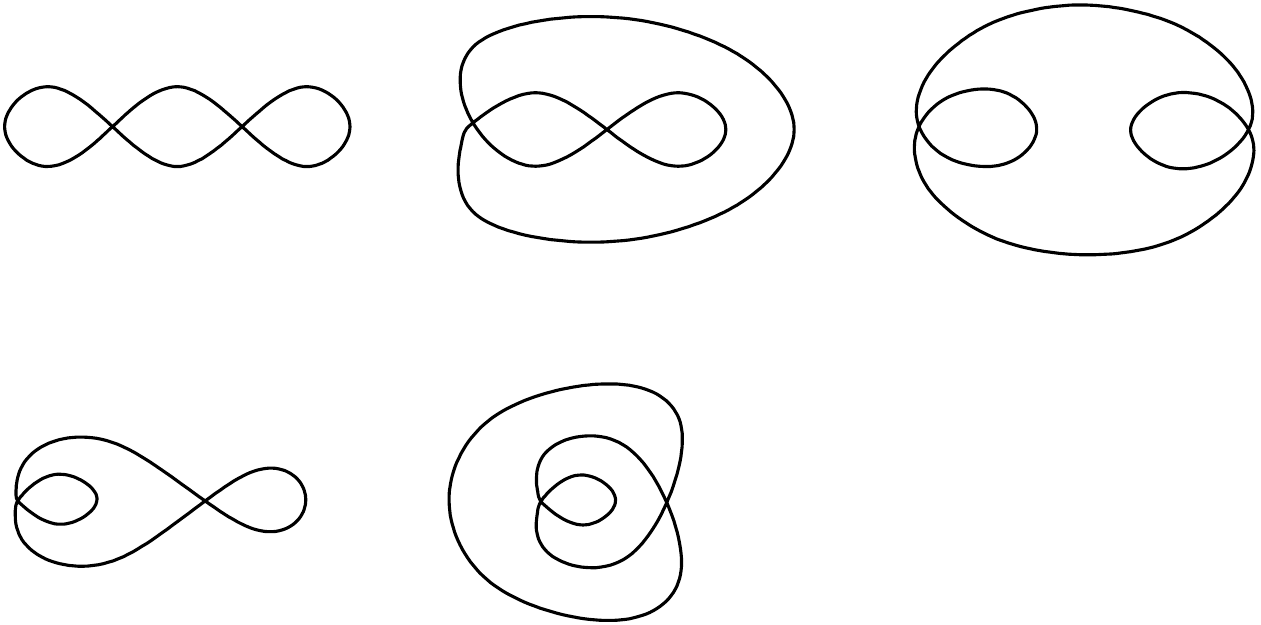} 
\caption{\label{imm2} Immersions of an unoriented circle with two double points. The five immersions in the plane give rise to two 
distinct immersions on the sphere, for instance the two lying on the left.}}
\end{figure}

 \begin{figure}[htbp]
 \centering{\includegraphics[width=7pc]{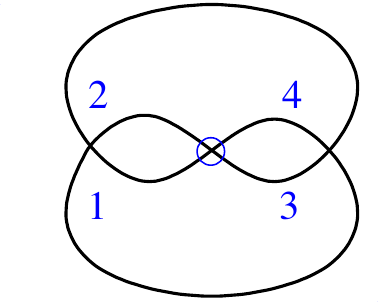}\qquad \qquad\raisebox{10pt}{\includegraphics[width=7pc]{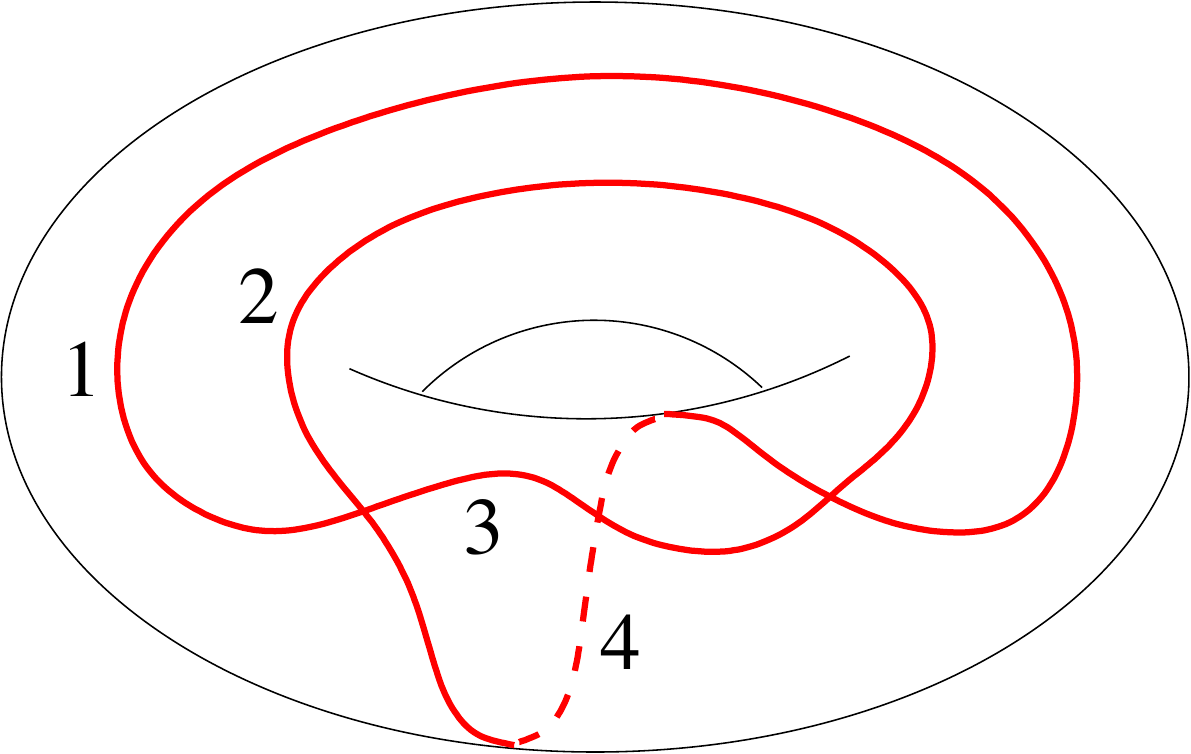}}
\caption{\label{virtual2} {The diagram on the left describes a genus 1  immersion
and is {\it not} bi{\colour}able. The little blue circle encircles a {\it virtual} crossing. On the right, the same immersed
in a torus.
}}}
\end{figure}

 \begin{figure}[htbp]
 \centering{\raisebox{5pt}{\includegraphics[width=5.4pc]{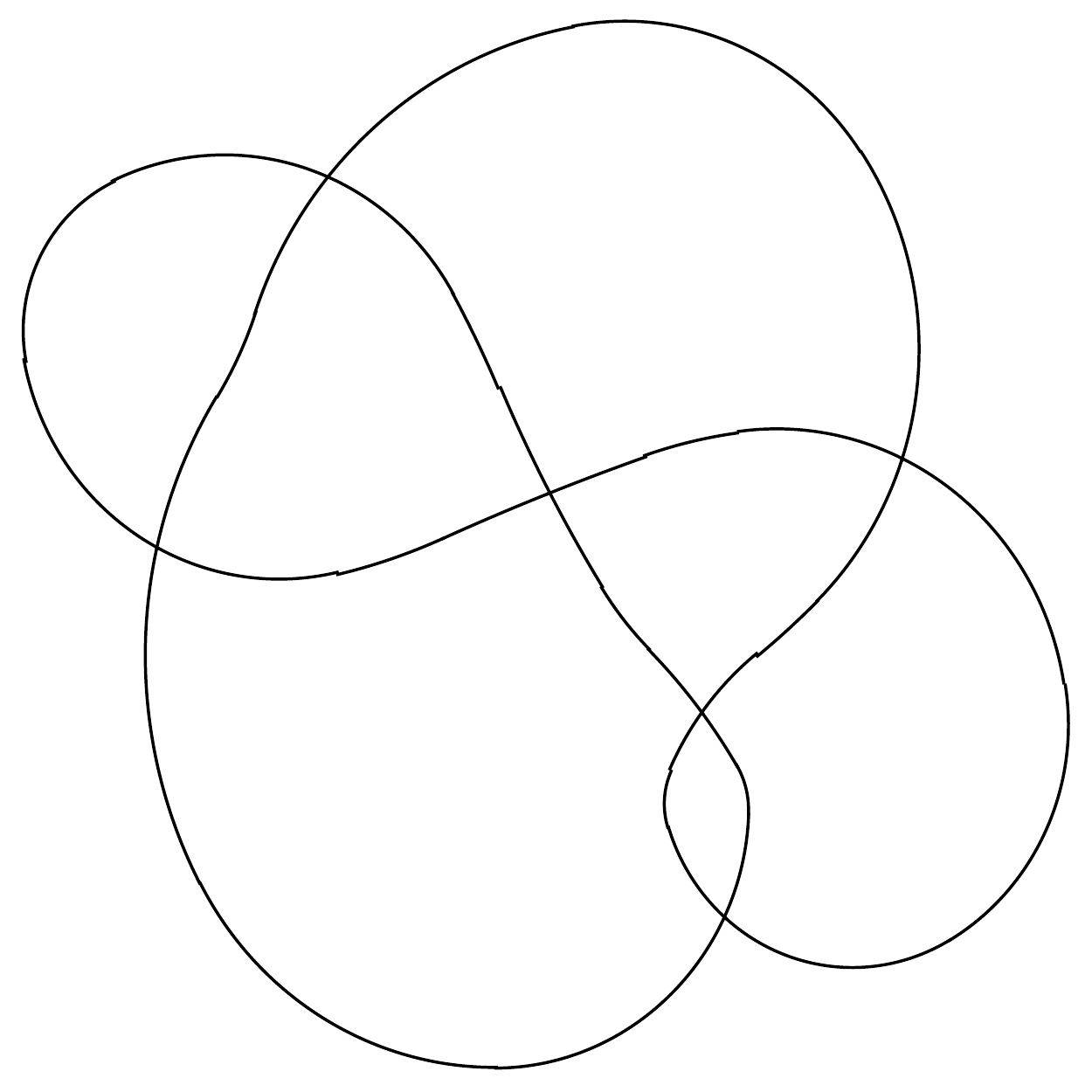}}\quad 
 \rotatebox{70}{\includegraphics[width=5pc]{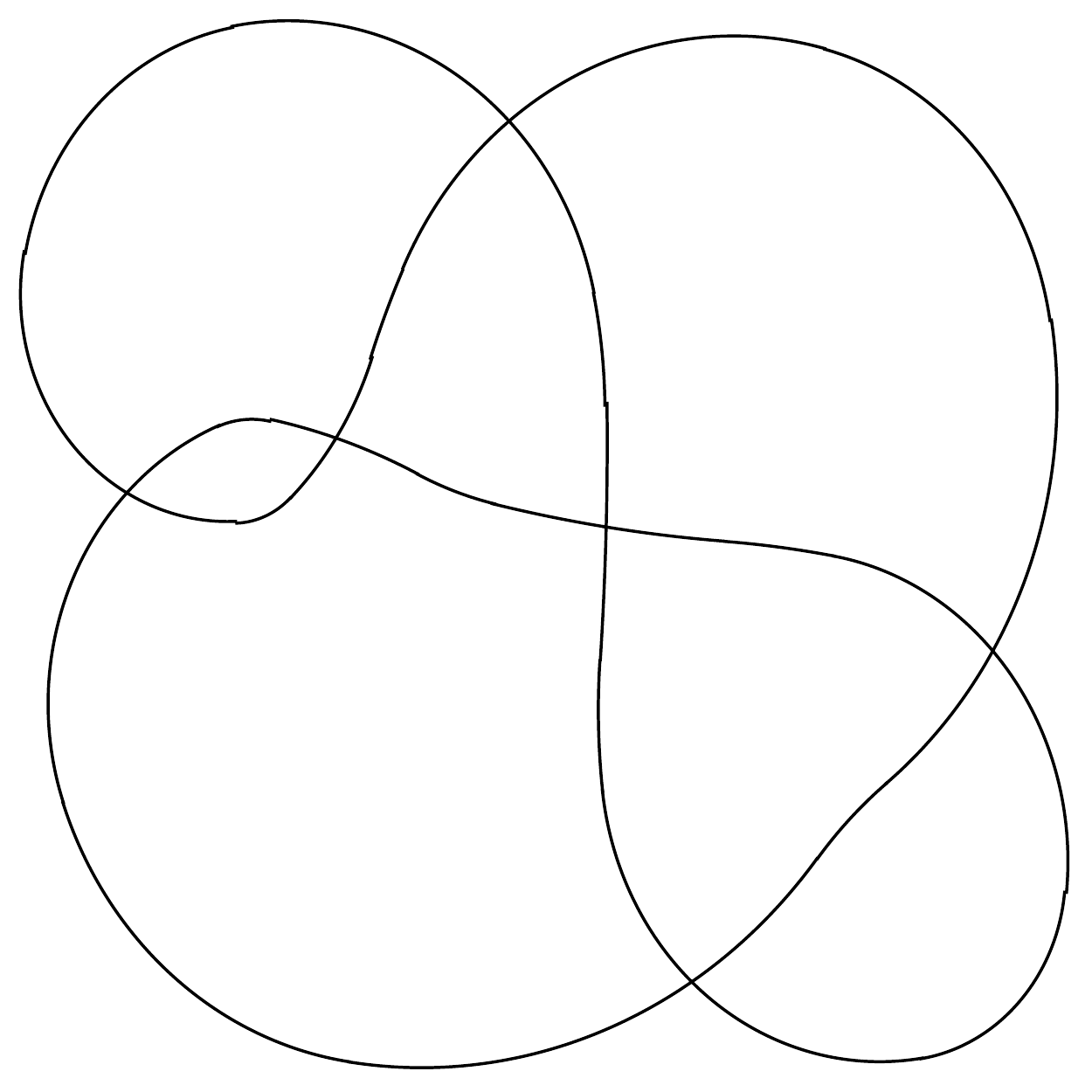}}
\caption{\label{imm6} Two immersions of an unoriented circle with $n=6$ double points. Distinct on an oriented sphere, but equivalent on an unoriented sphere.}}
\end{figure}

 \begin{figure}[htbp]
 \centering{\includegraphics[width=19pc]{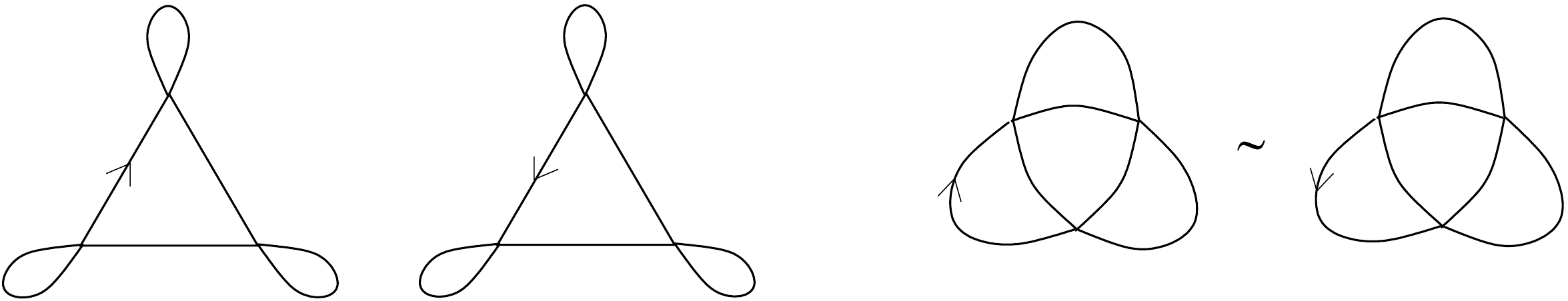}
\caption{\label{immrev} Immersions of an oriented circle.  Left : an $n=3$ immersion  not  equivalent to its reverse; in contrast, the trefoil
{\it is}  equivalent to its reverse.}}
\end{figure}

 \begin{figure}[bhtp]
 \centering{\includegraphics[width=18pc]{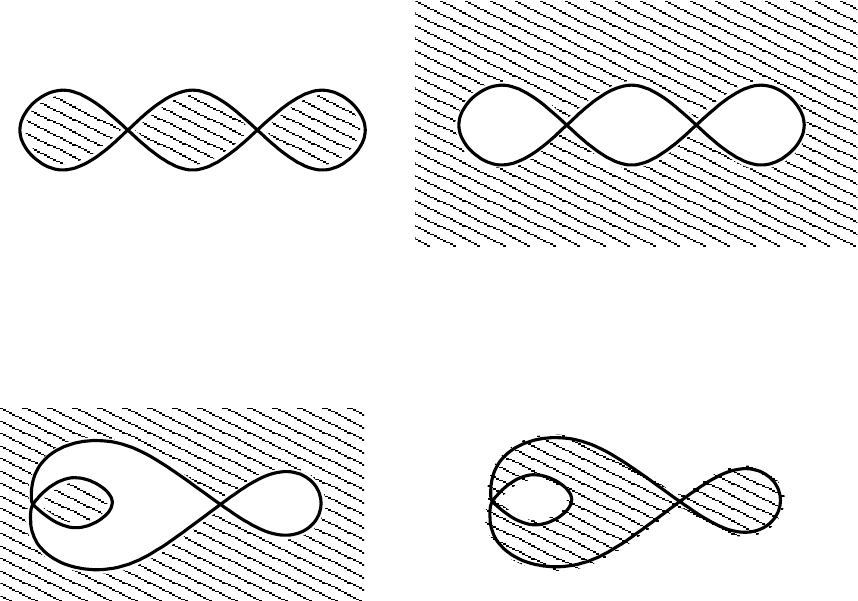} 
\caption{\label{imm2-swp} {Swapping {\colour}s : the  two diagrams on the top
are not equivalent, while the two diagrams on the bottom are (on the sphere, of course).
The first two contribute 2 to $\card{{\rm UUc}}$ and 1 to $\card{{\rm UUb}}$, the last two contribute 1 to both. 
}}}
\end{figure}

 For $g >0$ immersions they are few explicit results made available in the literature, see however\footnote{ The latter reference (that we discovered at a later stage of our work) contains, for the UU case, a table of isomorphism classes of immersions, with given genus and a given number of crossings (up to five), obtained using Gauss diagrams and a method described in \cite{Carter}.} \cite{ChmutovDuzhin} and \cite{CairnsElton}.

We shall regard curves with simple crossings (images of immersions) as {\it 4-valent maps}.
  Our immersions, being cellular, indeed define maps\footnote{They also define cellular embeddings of particular graphs called ``simple assembly graphs without endpoints'' in \cite{BDJSV}, see also \cite{Arredondo}.}: recall that a map is a graph embedded in a surface with its 2-cells (aka faces) homeomorphic to open disks. 
 The fact that faces do not contain handles will be used repeatedly in this paper, in particular when using the Euler formula to determine the genus of embedded curves. 
We should insist on the fact that in this work we consider circle immersions/maps, bi{\colour}able 
or not:  all the curves that we consider  have a single connected  component  (in the language of knot 
theory, we are interested in {\it knots}, not in {\it links}).\\
Matrix integrals in the large size limit which are quite effective for the counting of maps of
a given genus fail to distinguish maps with different numbers of components. 
We thus use an alternative method regarding maps as {\it combinatorial maps}, \ie maps described by 
pairs of permutations, following an old idea  by Walsh and Lehman \cite{WalshLehman}, or some variants.
The constraints of ``one-componentness" and of fixed genus may be easily enforced in that description. 
Depending on the orientation issue and on the bicolourability assumption, permutations for a map with $n$ vertices live in $S_{4n}$ or in $S_{2n}$.
 \\ 
This method, however, yields  labelled maps. To obtain unlabelled maps and immersions,
a quotient by a relabelling group has to be performed. This is achieved by considering
{\it orbits} of the combinatorial maps under the action of some subgroup of the permutation 
group. \\
The set-up of the paper is as follows. In Sect. 2, we present the simplest version
of the previous idea, where the two permutations encoding general immersions live in 
$S_{4n}$. The rapid growth of $(4n)!$  limits its practical use beyond $n=6$.
In Sect. 3, we consider bicolourable maps and introduce a better 
coding by pairs of permutations of $S_{2n}$.  Orbits of these pairs under the action of 
the hyperoctahedral group yield immersions of type OOc. {Sect. 4}  is devoted to
a study of the various types of bicoloured or bicolourable immersions that may be derived
from the OO type. We derive some relations between the numbers of these different types
(Theorem \ref{Theo3c}). In Sect. 5, we remove the assumption of bicolourability and 
encode the general maps and immersions by another choice for the  pair of permutations
of $S_{2n}$. Sect. 6 and 7 gather results, comments on the asymptotia and 
on the application to knot theory and our conclusions. Appendix  A gives some details on the algorithms
used for counting orbits, {\ Appendix B contains several tables of interest that will be described later},  and  Appendix  C reviews the connection between maps and Feynman diagrams of
matrix or scalar integrals.

A notational comment: in the following, we make use of two notations for the cardinal of a set $X$, either $|X|$ or $\#X$.

   %%%%%%%%%%%%%%%%%%%%%%%%%%%%%%%%%%%%%%%%%%%%%%%%%%%%%%%%%%%%%
\section{UO immersions,  first method using permutations of $S_{4n}$}
\label{Xpmethod}

\subsection{The subset $X = [2^{2n}]$ of $S_{4n}$ and its orbits (``X method")} 
\label{X}
In the present section, we obtain the number of circle immersions of type UO,  with $n$ crossings, by counting the number of orbits
of solutions for a particular set of equations written in the group $S_{4n}$, under the action of a particular subgroup.  
We shall actually recover part of these results later, with other methods, which are faster (see Sect.\ \ref{Ypmethod} and \ref{Zpmethod}), but the technique presented here has an interest of its own.

\paragraph{Method: description of a curve by a permutation belonging to a particular conjugacy class $X$ of $S_{4n}$.\\}
In a first stage, we consider a labelling of {\it half-edges} of the maps. 
For a 4-valent map with $n$ vertices, there are $4n$ such half-edges, and we consider the symmetric group
$S_{4n}$ acting on these labels. 
 We choose 
  $\sigma\in [4^n]$ \ \footnote{The notation $[\cdot]$ refers to the conjugacy classes of the permutation group} to describe the clockwise linking pattern of  half-edges at the vertices, and  consider all possible pairings of half-edges 
(propagators in physicists' parlance) encoded in  permutations $\tau\in [2^{2n}]$. 
Note that this method of labelling half-edges is not original,
it has been used by Walsh and Lehman \cite{WalshLehman} and rediscovered later by Drouffe, as
quoted in \cite{BIZ}. 

\paragraph{Example of encoding}
See below in Fig.\ \ref{newfiga} the map encoded by 
\begin{center}$\sigma=(1,2,3,4)(5,6,7,8)(9,10,11,12)(13,14,15,16)$ 
\\$\tau = (1, 13)(2, 5)(3, 6)(4, 16)(7, 8)(9, 12)(10, 15)(11, 14)$\end{center}
in cycle notation.

\normalcolor\paragraph{Orbits of $X= [2^{2n}]$ for the adjoint action of the centralizer of an element of $[4^n]$}

\begin{mytheo}
\label{theo1}
Call  $\sigma = (1, 2, 3, 4)(5,6,7,8) \ldots (4n-3, 4n-2, 4n-1, 4n) \in [4^n] \subset S_{4n}$,  using cycle notation, and  $\CC_\sigma = C(S_{4n}, \sigma)$, the centralizer of $\sigma$ in $S_{4n}$.
Let $X =  [2^{2n}] $ denote the conjugacy class of $S_{4n}$ whose elements are products of $2n$ transpositions. 
Then we have :\\
Circle immersions of type UO, \ie immersions of the unoriented circle in an orientable and oriented surface of genus $g$, are in bijection with the orbits of  $\CC_\sigma$ acting by conjugation  on the set of permutations $\tau$ that belong to $X$ and solve the simultaneous equations:
 \bea
    \nonumber
   &\sigma^2 \tau \in [(2n)^2] & \qquad\qquad\qquad\,{\rm(I)}\qquad\qquad{\rm one-componentness}\\
   \nonumber 
  &c(\sigma \tau) =  n+2 - 2g &\qquad\qquad\qquad{\rm (II)}_{{g}}\qquad\qquad{\rm genus \ condition} 
     \eea
 where $c(x)$ is the function giving the number of cycles (including singletons) of the permutation $x$. 
\end{mytheo}

\begin{proof}
That labelled maps are in one-to-one correspondance with pairs $(\sigma,\tau)$ has
been known for long \cite{WalshLehman}.  The sequence of  labels as one goes 
 across the crossings is described by the permutation $\sigma^2 \tau$, and imposing
 condition (I) ensures that  the curve has a single component (hence also that the graph is connected).
Condition (II)$_g$ follows from Euler relation, if one realizes that the number of faces 
of the map is just the number of cycles $c(\sigma\tau)$, (another observation made
by many previous authors\dots). 
  A change of labels by $\gamma\in S_{4n}$ acts on  $\sigma$ and $\tau$ by
 conjugation: $(\sigma,\tau)\to (\sigma^\gamma, \tau^\gamma)$, with 
 $\alpha^\gamma:= \gamma \alpha\gamma^{-1}$. 
  The form of $\sigma$ as well as 
  conditions on permutations $\tau$ of the type (I), (II)$_{{g}}$, (I)$\cap$(II)$_{{g}}$,
    are invariant under the action of any $\gamma$ in the centralizer 
   $\CC_\sigma$ of $\sigma$,  \ie 
    \be\nonumber \tau \ {\rm satisfies\ (I)\ and/or\ (II)}_{{g}}\ ,\ \gamma\in \CC_\sigma \quad  
   \Longrightarrow\quad     \tau^\gamma\ {\rm satisfies\ it\ too\,. }\ee
   \end{proof}
   Heuristically,  $\CC_\sigma$ is the group of reparametrizations (relabellings) of the edges 
   of the diagram that leave the pattern of edges around each vertex unchanged. Quotienting by that 
   group, \ie considering its orbits  for the adjoint action  thus enables one to go from labelled  maps to unlabelled, topologically distinct maps.
 Finally note 
 that the definition of $\sigma$ as describing the, say,  {\it clockwise} linking pattern at 
vertices has singled out an orientation of the surface, while no information about the
orientation of the circuit described by $\sigma^2\tau$ is provided: 
the maps are naturally asssociated 
with immersions of an {\it unoriented} circle in an {\it oriented} surface, hence of type UO in 
our nomenclature. 

  We shall use the following notations for the relevant subsets of $X=[2^{2n}]$: 
     \begin{equation*}
     \begin{split}
     X'=&\{ \tau\in [2^{2n}] \ |\   \sigma^2\,\tau \in [(2n)^2] \} \\
     X'_g=&\{ \tau\in [2^{2n}] \ |\  \sigma^2\,\tau \in [(2n)^2]\ \ \text{and} \ \sigma\,\tau\ {\rm has \ } n+2-2g\ {\rm cycles} \}\,.
     \end{split}
       \end{equation*}
 In particular we denote $X^{\prime \prime} = X^\prime_0$, corresponding to {\it planar} 
 (in fact spherical) maps.
   The family of sets $X^\prime_g$ is a partition of $X^\prime$, and the set of orbits of the latter (identified with circle immersions), for the adjoint action of the subgroup $\CC_\sigma$,  is also partitioned into orbits corresponding to the various circle immersions of genus~$g$.

Remarks.\\ (i) In  the group $S_{4n}$,  the function $c$ is related to the ($n$-independent)  length function $\ell$ by $\ell(x) = 4n - c(x)$. 
Therefore equation (II)$_{{g}}$ also reads $\ell(\sigma\tau) = 3n -2+2g$.
\\
 (ii) In the wording of the theorem we made a convenient choice for $\sigma$;  this is actually irrelevant since the choice amounts to
  labelling the half-edges in a specific way.\\
  (iii) An arbitrary curve has $c(\sigma\tau)=n+2-2g\ge 1$,  therefore the possible values of  $g$ are such that  $2g \leq n+1$.

\paragraph{Examples.}
Let us choose $n=4$, then\\ $\sigma = (1,2,3,4)(5,6,7,8)(9,10,11,12)(13,14,15,16)$ in cycle notation, equivalently\\ $\sigma =  [2, 3, 4, 1, 6, 7, 8, 5, 10, 11, 12, 9, 14, 15, 16, 13 ]$ in list notation.

First example: $\tau = (1, 13)(2, 5)(3, 6)(4, 16)(7, 8)(9, 12)(10, 15)(11, 14)$. \\ 
One checks that $\sigma^2 \tau = (1, 15, 12, 11, 16, 2, 7, 6)(3, 8, 5, 4, 14, 9, 10, 13) \in [8^2]$, 
 so $\tau$ obeys condition (I) and therefore encodes an immersion (possibly non-spherical). One then evaluates $\sigma \tau = (1, 14, 12, 10, 16)(2, 6, 4, 13)(3, 7, 5)(11, 15)$; its number of cycles is $6$ (there are two unwritten\footnote{Since $n$ is fixed, it is unnecessary to write explicitly the 1-cycles when using the cycle notation, but one should remember that the function $c(x)$ should count the total number of cycles !} singletons: (8) and (9), so that the genus is $0$, and the immersion is actually spherical. The immersion encoded by permutation $\tau$ is given in Fig.\ \ref{newfiga}. Notice that $1$-cycles give rise (or come from) {\it kinks}, also known as {\it simple loops}.  
 
   \begin{figure}[htbp]
 \centering
\includegraphics[height = 5pc]{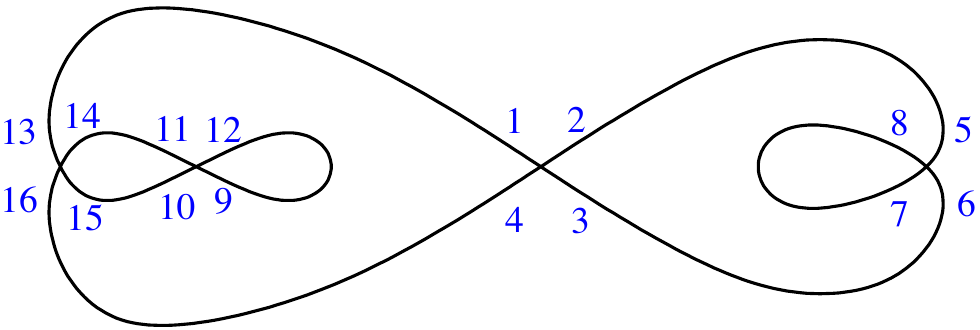}  
\caption{\label{newfiga} The diagram encoded by \\ $\tau = (1, 13)(2, 5)(3, 6)(4, 16)(7, 8)(9, 12)(10, 15)(11, 14)$.}
\end{figure}

Second example: $\tau = (1, 8)(2, 3)(4, 16)(5, 13)(6, 12)(7, 14)(9, 15)(10, 11)$. \\ One checks that  $\sigma^2 \tau = (1, 6, 10, 9, 13, 7, 16, 2)(3, 4, 14, 5, 15, 11, 12, 8) \in [8^2]$, but this time $\sigma \tau = (1, 5, 14, 8, 2, 4, 13, 6, 9, 16)(7, 15, 10, 12)$ which has two 1-cycles (3) and (11),  and therefore a number of cycles equal to $4$, so the permutation $\tau$ describes an immersion in a surface of genus 1. The encoding is made explicit in Fig.\ \ref{newfigb}. Notice that this circle immersion has four real crossings, as expected, but also one virtual one.
  \begin{figure}[htbp]
 \centering
\includegraphics[height = 8pc]{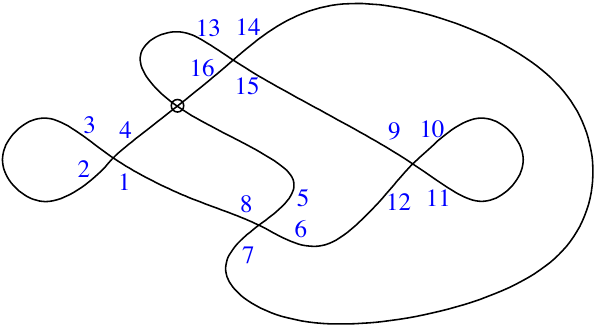}  
\caption{\label{newfigb} The diagram encoded by\\ $\tau = (1, 8)(2, 3)(4, 16)(5, 13)(6, 12)(7, 14)(9, 15)(10, 11)$. \\ The virtual crossing is indicated 
by an open circle.}
\end{figure}

%%%%%%%%%%%%%%%%%%%%%%%%

\subsubsection{Counting orbits}  
  One would like to count the orbits for the $\CC_\sigma$ action on the sets $X$, $X^\prime$, $X^\prime_{g}$ and in particular on $X^{\prime\prime}=X^\prime_{0}$.
      How to find {\it a priori} the number and the {\it lengths} of the $\CC_\sigma$ orbits ?
    \paragraph{Burnside's lemma} 
    asserts that the number of orbits in, say, $X^\prime$ is related to the total number $\sum_\kappa \card{X'^\kappa}$ 
    of fixed points in the action of $\kappa\in \CC_\sigma$  acting in $X^\prime$, 
    \ie the number of pairs $(\kappa,\xi)$ such that $\kappa \xi=\xi \kappa$, 
     by  
    \be \card{X' /\CC_\sigma}=\#\,\CC_\sigma\mathrm{-orbits\ in\ } X'= \frac{\sum_\kappa \card{X'^\kappa}}{\card{\CC_\sigma}}\,.\ee
   This implies, however, the computation of $\card{X'} \times \card{\CC_\sigma}$ pairs of products $(\kappa\xi,\xi\kappa)$,
    which  becomes prohibitively large  for $n\ge 6$. 
    
       \def\orb{{\rm orb\,}} \def\Cl{{\rm Cl}}
   \paragraph{Orbits, double classes  and a formula by Frobenius.}
   Let us first state a simple but useful theorem (that belongs to the folklore)
\begin{mytheo}
\label{Coqtheo} Let $G$ be a finite group and  $H$ be a subgroup of $G$.
Take $x \in G$ and call $\Cl(x)$ its conjugacy class.
Then the  orbits for the adjoint action of $H$ on $\Cl(x)$ are in one-to-one correspondence with 
double cosets $H  \backslash G /  K$ where $K=C(G,x)$  is the centralizer of $x$ in $G$. 
\end{mytheo}   
\begin{proof}
Let $x\in G$. Then $y,y'\in \Cl(x)$ belongs to the same $H$-orbit iff $\exists h\in H$:
$y'=h y h^{-1}$, but $y=gxg^{-1}$ and $y'=g'xg'^{-1}$, hence  $g'xg'^{-1}=h gxg^{-1} h^{-1}$
or $g'^{-1}h g x= x g'^{-1}h g $, 
 from which it follows that $k:= g'^{-1}h g\in K:=C(G,x)$ and $g' =h g k^{-1} \in H g K$. 
 \end{proof}

The counting of $H$-orbits in  $\Cl(x)$ thus amounts to the counting of these double cosets. 
     Frobenius \cite{Frobenius} has given a formula for the number of double cosets $H\backslash G/K$. 
     In essence his method consists
     in computing in two different ways the number of solutions of equation $h g k=g$ 
      with $g\in G$, $h\in H$ and $k\in K$, with the result that
     \be\label{frobenius} 
     \card{H\backslash G/K} \,=\, \frac{\card{G}}{\card{H}\,\card{K} }\sum_\mu \frac{\card{H_\mu}\,\card{K_\mu}}{\card{G_\mu}}\ee
   where the sum runs over conjugacy classes $G_\mu$ of $G$, $H_\mu=H\cap G_\mu$ and $K_\mu=K\cap G_\mu$. 

   We are going to make repeated use of this connection between orbits and double classes and of Frobenius'  formula.
   In the problem at hand, $G=S_{4n}$,    there is a one-to-one correspondence between the orbits of 
   $\tau\in \Cl=X=[2^{2n}]$ under the action of $H=\CC_\sigma=C(S_{4n}, \sigma)$ and double cosets of $S_{4n}$ of the form 
     $\CC_\sigma \backslash S_{4n}/\CC_\tau$ with  $\CC_\tau = C(S_{4n}, \tau)$.
In this particular case of permutations, Frobenius'  formula is easy to apply without calculating 
     explicitly the intersections: classes $G_\mu$ are indexed by partitions, and $H_\mu$, $K_\mu$ 
     can be simply deduced from the knowledge of the cyclic structure of their own conjugacy classes.
  Note that this method allows one to  count the   $\CC_\sigma$-orbits in $X$, but not those in $X^\prime$ or in $X^{\prime\prime}$.

 \paragraph{Poor man's method.}
In the latter case, the same
 group theoretical considerations are not applicable, and we resort to a direct construction 
   of these orbits on a computer. This  ``brute force"  method has the merit of giving not only the
   number and lengths of orbits, but also an explicit representative of each of them, thus providing a {\it catalogue} 
   of the corresponding maps or immersions.
  In practice, this   method may be used up to $n=6$. For higher values, we resort 
  to  alternative methods presented in the next sections. 
       
\subsubsection{Orbits lengths}

The length of orbits, for the $\CC_\sigma$ action on any of the sets $X$, $X^\prime$ or $X^\prime_g$, is not constant.
For instance, when $n=4$ there are $121$ orbits for the $\CC_\sigma$ action on $X^\prime$ ($21$ of genus $0$, $64$ of genus $1$, and $36$ of genus $2$), but they are not all of the same length: $92$ are of maximal size (namely $6144$, the order of $\CC_\sigma$), $23$ are of size $\card{\CC_\sigma}/2$, and $6$ are of size $\card{\CC_\sigma}/4$). This corresponds to the fact that the centralizer $C(\CC_\sigma,\tau)$, in $\CC_\sigma$, of a permutation $\tau$ describing some specific immersion is not necessarily trivial, and  $\#\mathrm{Orb}(\tau) = \card{\CC_\sigma}/\card{C(\CC_\sigma, \tau)}$.
Existence of ``symmetries'', for a specific immersion, is measured, or actually defined, by $C(\CC_\sigma, \tau)$. The order 
$\omega$ of this group will not be given in our tables, but it is easy to obtain for every particular case.
For large $n$, almost all orbits have trivial stabilizers \cite{RW}, and an  estimate for the total number of immersions, including all values of $g$, is asymptotically given by $\card{X^\prime}/ \card{\CC_\sigma}$, equal to $\tfrac{(4n-2)!!}{4^n n!}$, when $n>2$ (see below, Appendix C and Table \ref{TableX}).

\subsection{Results}
   \paragraph{The group $\CC_\sigma$.}
      Given $\sigma\in [4^n]$, \ie a product of $n$ cyclic permutations on $n$ disjoint sets of 4 objects, 
      its centralizer $\CC_\sigma$  is made of cycles operating on the $n$ same sets of 4 objects, times
      any permutation of these $n$ cycles.
      Whence the order
      $$ \card{\CC_\sigma}= 4^n \, n!\,, $$
      i.e. $\card{\CC_\sigma}={4, 32, 384, 6144, 122\,880, 2\,949\,120\cdots}\quad {\rm for}\ n=1,2,3,4,5,6,$ see Table \ref{TableY}. 
      \\ 
      \normalcolor
  \paragraph{The set $X=[2^{2n}]$ and its $\CC_\sigma$-orbits.}   
  \normalcolor  {A standard result is that $\card{ X} = (4n-1) !!$}.

      How many orbits are there when $\CC_\sigma$ acts by conjugation on the class  $X=[2^{2n}]$~? 
     By use of Frobenius'  formula for double cosets (\ref{frobenius}) 
     we find that for $n=1,2,\cdots,9$, there are
         \be 
     \# \CC_\sigma{\rm-orbits\ in\ } X= 2, 10, 54, 491, {6430 } , 119\,475 , 2\,775\,582 , 76\,733\,201, 2\,439\,149\,685  \,.\ee
     One is of length 1 : the orbit of $\sigma^2$. 
     
      \paragraph{The set $X^\prime$ and its $\CC_\sigma$-orbits.}   We prove in Appendices  
      C.2 and C.3, using 
      a simple integral calculation or a purely combinatorial argument,
           that $\card{ X^\prime} = (4n-2)!!$.
       Acting on that $X^\prime$, $\CC_\sigma$ has a number of orbits given by 
     \be  \# \CC_\sigma{\rm-orbits\ in\ } X'= 1,3,13,121,1538, 28\,010, \cdots 
     \ee

Taking $n=4$ for example,   using the orbit stabilizer theorem and denoting as above by $\omega$ the order of the 
centralizer $C(\CC_\sigma,\tau)$, one finds that
there are $92$ orbits in $X^\prime$ with $\omega=1$, $23$ orbits with $\omega=2$ and 6 orbits with $\omega=4$, a total of $121$.
One checks that  $92\card{\CC_\sigma} + 23  \card{\CC_\sigma}/2 + 6  \card{\CC_\sigma}/4 = 645120 = \card{X^\prime}$, as it should.
  Moreover $\#( \CC_\sigma {\backslash} S_{4n}/ \CC_\tau) = 491$ corresponding to the $\CC_\sigma$-orbits of $X$,  but only $121$ correspond to orbits of $X^\prime$.

%%%%%%%%%%

    \paragraph{The number of $\CC_\sigma$-orbits in $X^{\prime\prime}$.}
    Among the $\CC_\sigma$ orbits in $X^{\prime}$ we pick those that are such that $\sigma \tau$ has $n+2$ cycles (condition (II)$_0$ for genus 0).
     We find a number of relevant orbits equal to    $1,2,6,21,99,588,\cdots$.
  As discussed above, those are the numbers of immersions with $n$ double points
      of an unoriented circle in the oriented sphere,
     see Fig.\ \ref{immersions}, \cite{Arnold,Valette}  and OEIS sequence A008987.
     \be\label{immersions1} \#\mathrm{UO\ spherical\ immersions} =1,2,6,21,99,588,3829,\cdots \ee
     In fact the number 3829 (for $n=7$) and  further terms will be obtained below through a different method.

     The following Table \ref{TableX} summarizes, for $n=1,\dots,6$, most of the results obtained using the above technique.
The number of immersions in surfaces of specific genus $g >0$ can be obtained in the same way (only the spherical ones appear in Table \ref{TableX}) but the corresponding values are gathered in Table \ref{TableXYZ1} because we shall later recover and extend the results obtained in the present section.
   The last line of Table \ref{TableX} refers to a quantity (free energy) defined in Appendix C.

\begin{table}[ht]
\caption{Orbits of $X$ subsets. $X = [2^{2n}]$. The numbers in blue give the asymptotic estimate of the number of orbits.
The numbers of spherical UO immersions are given by  the line $\#\ \CC_\sigma$-orbits in $X^{\prime\prime}$.  
 The total numbers of UO immersions (all genera) are given by the line $\#\ \CC_\sigma$-orbits in $X^\prime$.
 Last entry  $F_n^{(0,1)}$ of the Table is defined in Appendix C.}
\centering
\hglue-4mm
    \begin{tabular}{ | l || c|c|c|c|c|c|c|} 
\hline
     $n$ & 1& 2 & 3 & 4 & 5 & 6 \\
      \hline \hline
      $\card{\CC_\sigma}= 4^n n!$ &   4 & 32 &  384 & 6144 & 122\,880 & 2\,949\,120 \\
      \hline 
      $\card{X}=%\card{[2^{2n}]}=
      (4n-1)!!$ & 3& 105& 10\,395& 2\,027\,025& 654\,729\,075& 316\,234\,143\,225 \\
           \hline
    $\#$ $\CC_\sigma$-orbits in $X$ & 2& 10& 54& 491&{6430 }  &119\,475        \\ 
        \hline
      $\card{X'}=(4n-2)!!$ \hfill (I)&  2& 48& 3840& 645\,120& 185\,794,560& 81\,749\,606\,400
        \\
      \hline
  \rowcolor{Apricot}       $\#$ $\CC_\sigma$-orbits in $X^\prime$ &   1&3&13&121& 1538 &  28\,010  
       \\
      \hline
      \Blue{$\card{X'}/\card{\CC_\sigma}$}&$\Blue{\frac{1}{2}}$ & $\Blue{\frac{ 3}{2}}$ & \Blue{10} & \Blue{105} & \Blue{1512} & \Blue{27\,720}
    \\
            \hline
$\card{X''}$  \quad  \hfill (I)$\cap$(II)$_0$ &2& 32& 1344& 99\,840& 11\,034\,624 &1\,646\,100\,480 
\\
      \hline  
\rowcolor{Apricot} $\#$  $\CC_\sigma$-orbits   in $X^{\prime\prime}$& 1 & 2 & 6 & 21 & 99 & {588} 
 \\
 \hline
      \Blue{$\card{X''}/\card{\CC_\sigma}$}&$\Blue{\frac{1}{2}}$ & $\Blue{1}$ & \Blue{3.5} & \Blue{16.25} & \Blue{89.8} & \Blue{558.17}
 \\
\hline
 $F_n^{(0,1)}=
 \frac{\card{X''}}{4^n n!}$ & $\frac{1}{2}$& 1 &$\frac{7}{2}$ &$\frac{65}{4}$ & $\frac{449}{5}$& 
 $\frac{3349}{6}$ \\
 \hline  
          \end{tabular}
     \label{TableX}
      \end{table}

   \vglue10mm \begin{figure}[htbp]
 \centering
\includegraphics[width=30pc,  height = 10pc]{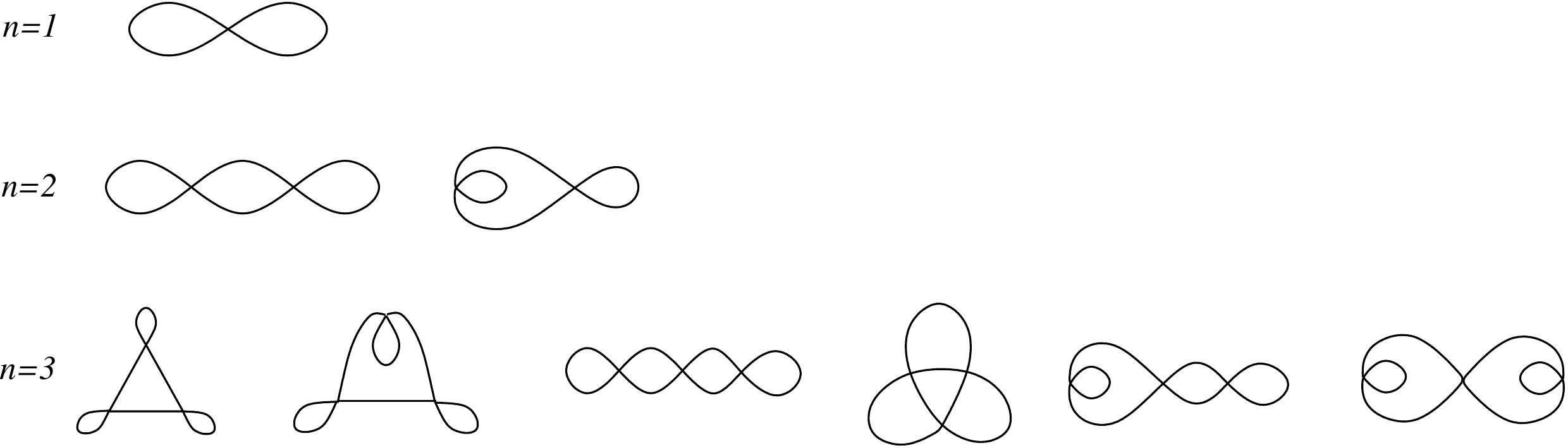}
\caption{\label{immersions} Immersions of an unoriented circle in the oriented sphere with $n$ double points, for $n=1,2,3$}
\end{figure}
     
%%%%%%%%%%%%%%%%%%%%%%%%%%%%%%%%%%%%%%%%%%%%%%%%%%%
\penalty -5000
\section{Bi{\colour}able and bi{\colour}ed maps of types UO and OO using~$S_{2n}$.} 
\label{Ypmethod}
\penalty 10000

In a nutshell: here we shall get UOc, the 
bicoloured immersions of type UO, then, forgetting the colour assignment, we shall get UOb, the 
bicolourable immersions of type UO, which turns out to be identical to 
the immersions UO for genus 0 maps (spherical curves). This will be explained below.  
     
\subsection{The set $Y = S_{2n}$, its orbits, and immersions of type UOc (``Y method")}
\label{Y}
{In the present section, we shall study the orbits
of solutions for a particular set of equations written in a set $Y$ {defined as} the symmetric group $S_{2n}$ itself, under the action of a particular subgroup that turns out to be its hyperoctahedral subgroup. }

\paragraph{Method: description of a bi{\colour}ed map
by a pair of permutations of $Y=S_{2n}$.}
It is a well known fact that {\it planar} maps with vertices of even valency may have 
their faces bi{\colour}ed. This applies of course to our 4-valent planar maps.
For non planar (\ie of genus $g>0$) maps, this is no longer 
guaranteed, (as already discussed in the Introduction and examplified in Fig.\ \ref{virtual2})
and we have to assume that the map is bicolourable, see below. 
     We then turn to a more efficient encoding of such {\it bicoloured} maps by permutations\,    \cite{ZJZ04}. 
      
     For a bi{\colour}ed map with $n$ vertices and $2n$ labelled edges, we deal 
     with permutations of $S_{2n}$, instead of $S_{4n}$ as above. A map 
      is encoded into a {\it pair} of permutations $\sigma, \tau \in S_{2n}$: $\sigma$ describes the 
     sequence of edges as white faces are traveled clockwise, while $\tau$ describes the counterclockwise sequence of
     edges on shaded faces, see Fig.\ \ref{sh-faces}a.   When considering the map as (the plane projection of)
     an {\it alternating} knot, one uses the convention for overcrossings/undercrossings shown on  Fig.\ \ref{sh-faces}b. 
     Define $\rho =\sigma^{-1}\tau$ and $\tilde\rho=\sigma \tau^{-1}$;
           it is clear that $\rho$ describes the pairings of edges at overcrossings,
     and $\tilde\rho$ at undercrossings, and they are
     both a product of $n$ disjoint transpositions, $\rho, \tilde\rho \in [2^{n}]$. 
     The chain of edges as one follows a thread of the knot/link is thus 
     described by $\rho \tilde\rho=\sigma^{-1}\tau\sigma\tau^{-1}$, and  white, resp. shaded, 
     faces correspond to cycles of $\sigma$, resp.~$\tau$. 
    \normalcolor
     
     \begin{figure}[htbp]
 \centering
\includegraphics[height = 8pc]{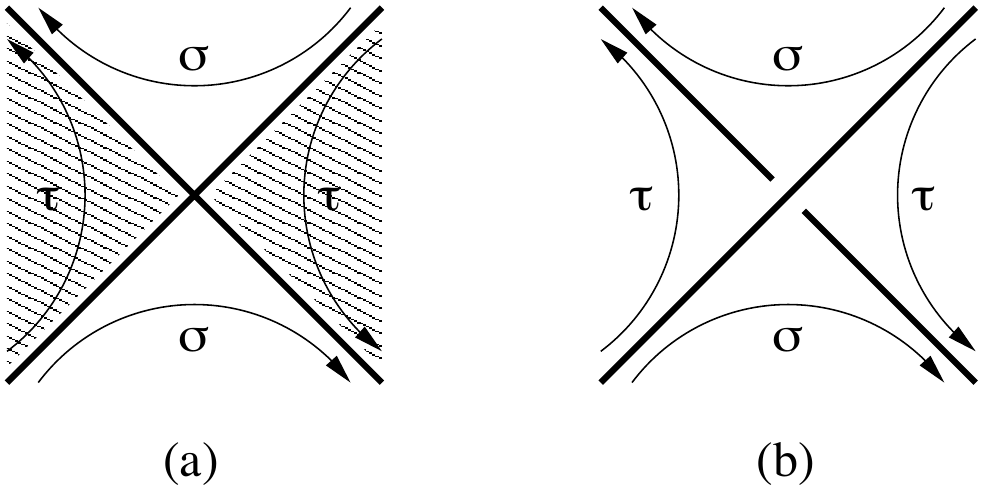}  
\caption{\label{sh-faces} White and shaded faces $\leftrightarrow$ over/under-crossings}
\end{figure}

Just as in Sect.\ \ref{X},   the two conditions of one-componentness and 
genus $g$ amount to imposing
\begin{equation*}  
\begin{split}
      \rho\,\tilde\rho=\sigma^{-1}\tau\sigma\tau^{-1} \ {\rm has \ } 2\ {\rm equal\ cycles},\ i.e.\  \rho\,\tilde\rho \in [n^2]\,.
        \quad\,&{\rm (I')}
       \quad\ {\rm one-componentness} \\
     c(\sigma)+c(\tau)=n+2-2g  \qquad\qquad &{\rm(II')}_g
        \quad\ \ {\rm genus } \ g 
\end{split}
\end{equation*}  

We want to count all $\sigma$ and $\tau$ subject to the above conditions.
Actually, it is convenient to fix $\rho$ in $[2^n]$, defining it for example by $\rho=\rho_0$,  $\rho_0(2i-1):=2i\,, \ \rho_0(2i)=2i-1\,,\  i=1,\cdots n$ (it is only a relabelling of the edges).
This choice being made, a map is then described (up to the conjugate action of the centralizer of $\rho$, see below) by a single permutation  $\sigma$, 
 since $\tau=\sigma \rho$. Notice that $\tilde\rho=\sigma\rho\sigma^{-1}$. 
With this choice for $\rho$,  the two conditions 
 (I${}^\prime$) and (II${}^\prime$)${}_g$ can be written in terms of equations for $\sigma$ (see theorem \ref{theo2}, below). 

As in Sect.\ \ref{Xpmethod}, it is natural to define the subsets $Y'$ and $Y'_g$ of $Y$, made of those
 permutations $\sigma$
that  respectively obey the conditions (I${}^\prime$) and (I${}^\prime$)$\cap$(II${}^\prime$)${}_g$.
The sets  $Y'_g$ constitute a partition of $Y^\prime$.
 
Ultimately, in order to count the number of  curves, one decomposes  the previous subsets $Y^\prime_g$, in particular $Y^{\prime \prime}=Y^\prime_0$  for spherical curves,  
into orbits for the conjugate action of $\CC_\rho$, the centralizer of $\rho_0$ in $S_{2n}$.

Finally, we observe that  the convention that $\sigma$ describes the clockwise
sequence of labels on white faces (and $\tau$ the counterclockwise one on shaded faces)
assumes that the sphere or the higher genus surface is oriented, while nothing specifies the
orientation of the curve. Our orbits,  in this section, are thus of type UO.

\paragraph{{Bi{\colour}ed versus bi{\colour}able curves.}}
One could think that the orbits of $Y^\prime_g$ should determine the various UO 
circle immersions of genus $g$. This is not so for two reasons, already mentioned in the Introduction.
First, the  curves obtained in this way correspond to    a bi{\colour}ing of a curve. For lack of a better name
we call  ``bi{\colour}ed immersions'' the bi{\colour}ed curves associated with the orbits of $Y^\prime_g$ (recall that in the language of 
knot theory, they describe alternating knots), and denote their set  by UOc. 
 Depending on whether the two alternative {\colour}ings (\ie the two choices of alternating over- and under-crossings)
 are or not topologically equivalent, they will contribute differently to the 
 counting of ordinary, un{\colour}ed immersions; this will be spelled out in Sect.\ \ref{bicoloredtobicolorable}. 
 Secondly, for genus $g>0$, not all curves are bi{\colour}able, see Fig.\ \ref{virtual2} for an example. 
 We shall call
``{\sl bi{\colour}able} curves'' or ``bi{\colour}able immersions''  (not to be confused with the {\sl bi{\colour}ed} 
ones previously described) the curves obtained by this technique, {\sl after} 
erasing the {\colour}s, and denote their set by UOb. Finally we recall that UO refers to immersions studied in the 
previous section, with no assumption of bi{\colour}ability. \\
In the Tables \ref{TableXYZ1} and \ref{TableXYZ2}, the reader can find  the cardinals of these various sets of
immersions, and check that $|$UO$|=|$UOb$|$ in genus 0, while  for $g>0$, $|$UO$|>|$UOb$|$,   
as expected since  bi{\colour}able curves do not exhaust all possible genus $g$ curves.
 
 \paragraph{Example of encoding} See  in Fig.\ \ref{newfigc} the example of the bi{\colour}ed diagram described by 
 $\sigma=[3,5,7,1,2,6,4,8] = (1,3,7,4)(2,5)(6)(8)$ and $\tau=[5,3,1,7,6,2,8,4] = (1,5,6,2,3)(4,7,8)$, hence $\rho=[2,1,4,3,6,5,8,7]$.

\medskip
{We summarize  the above method as follows: }

\begin{mytheo}
\label{theo2}   
 Call  $\rho = (1, 2)(3, 4)\ldots (2n-3, 2n-2)(2n-1, 2n) \in [2^{n}] \subset S_{2n}$, using cycle notation, 
 and 
 $\CC_\rho = C(S_{2n}, \rho) $, the centralizer of $\rho$ in  $S_{2n}$. 
Bi{\colour}ed circle immersions of the  unoriented circle in an oriented surface of genus $g$, or 
UOc immersions for short, are in bijection with the orbits of $\CC_\rho$ acting by conjugacy on 
$S_{2n}$  whose representatives $\sigma$ solve 
          \bea         \nonumber
       &\rho\,\sigma\,\rho\,\sigma^{-1} \ {\rm has \ } 2\ {\rm equal\ cycles},\ i.e.\  \rho\,\tilde\rho \in [n^2]\, \    {{\rm with \ } \tilde\rho = \sigma\,\rho\,\sigma^{-1}} 
       & \qquad\qquad\,{\rm (I')}
       \\
       \nonumber 
        &c(\sigma)+ c(\sigma\rho)=n+2-2g  & \qquad\qquad{\rm(II')}_{{g}}
   \eea
$c(x)$ being the function that gives the number of cycles (including singletons) of the permutation $x$. 
 \end{mytheo} 

 {Remarks. \\
 (i) In the wording of this theorem we chose a particular value of $\rho$ in the conjugacy class $[2^{n}]$, {namely $\rho = \rho_0$}, because it is simple and convenient,  but we could have made any other choice in the same class since this just corresponds to a relabelling of some edge labels.
We shall see in Sect.\ \ref{gaugefixing} how to further restrict the choice of~$\sigma$.\\
(ii) It is useful to remember that  $  
\rho^2 = \widetilde \rho^2 = 1$,  that $\widetilde \rho =  \rho^\sigma$ since, by definition,  $\rho^\sigma = \sigma \rho \sigma^{-1}$, and that
 $\tau = \sigma \rho = \widetilde \rho \sigma$. \\
(iii) The set $Y^\prime$ defined by condition ${\rm(I')}$
alone can also be written $Y^\prime = \{ \sigma \in S_{2n} \, : \, \sigma^\rho  \sigma^{-1} \in [n^2] \}$.}
%%%%%%%%%%%%%%%%%    
\paragraph{Example.}  
As an example we give in Fig.\ \ref{newfigc} the diagram of Fig.\ \ref{newfiga} in this new description.
 \begin{figure}[htbp]
 \centering
\includegraphics[height = 6pc]{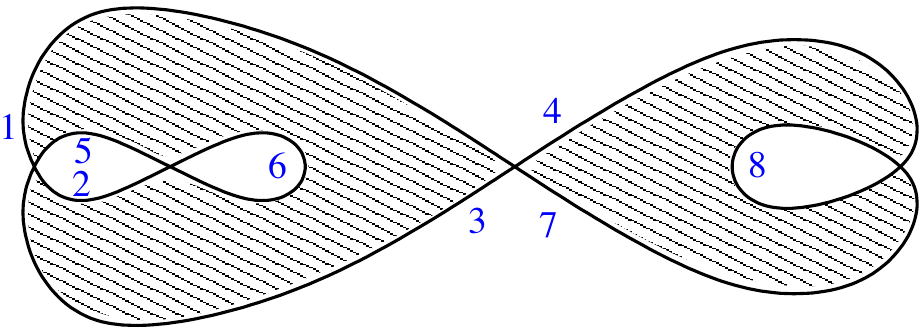}
\caption{\label{newfigc} The bi{\colour}ed diagram encoded by\\ $\sigma=[3,5,7,1,2,6,4,8]  = (1,3,7,4)(2,5)(6)(8) $}
\end{figure}    
\paragraph{ Structure of the centralizer.}
\label{pairS2nCrho}
The centralizer $\CC_\rho$ of $\rho_0$ is generated by transpositions of $(1,2)$, $(3,4)$,\dots $(2n-1,2n)$, times
a permutation of these $n$ pairs. Its order is thus $2^n n!$.  
This group is called the hyperoctahedral group $BC_n$, as it  is the group of symmetries of the $n$-cube. It admits several
different geometric and algebraic presentations.
One construction is as follows (see for example \cite{CecheriniSliberstein}).
The symmetric group $S_{2n}$ acts on  $\{1, 2, . . . , 2n\}$ and therefore also on the set of partitions of the latter consisting of two-element subsets.
Fix an element in this set (we shall choose $\{ \{1,2\}, \{3,4\}, \{5,6\}, \ldots, \{2n-1,2n\}\}$) and denote by $\CC_\rho$ its stabilizer. 
Clearly, $\CC_\rho$ permutes the $n$ two-element subsets among themselves and it is equal to the centralizer, in $S_{2n}$ of the permutation $\rho=(1,2)(3,4)\ldots(2n-1,2n)$, whence the notation.
The subgroup $\CC_\rho$ of $S_{2n}$ then appears as the semidirect product of $S_2 \times \ldots \times S_2$ ($n$ times) and $S_n$, the latter acting by permuting the factors of the former (wreath product). 

In order to put in perspective what will be done in Sect.\ \ref{gaugefixing}, let us make a few additional remarks.
Call $\varpi$ the map $S_{2n} \mapsto S_{2n}$ defined by $\varpi(x) = \rho x \rho^{-1}$. 
Clearly $\varpi$ is a group homomorphism and an involution, moreover the subgroup $\CC_\rho$ is the set of fixed points of this involution:   $\CC_\rho = \{x \, : x \in S_{2n} \, | \, \varpi(x) = x\}$.  
Define the map (not a group morphism) $\varphi : S_{2n} \mapsto S_{2n}$ by $\varphi(x) = \varpi(x^{-1}) x$.  
Notice that $\varphi(x^{-1})$ and $\varphi(x)^{-1}$ belong to the same $S_{2n}$ - conjugacy class since they are conjugated in $S_{2n}$ :  $\varphi(x)^{-1} = x^{-1} \varphi(x^{-1}) x$.
We have also $\varphi(x)^{-1} = \rho x \rho x^{-1}$, so that $\rho$ being fixed,  the condition characterizing the one-componentness of the permutation $x$ encoding a curve reads simply $\varphi(x) \in [n^2]$.
The reader will easily notice (see also \cite{AkerCan}) that, for any $k$ in $\CC_\rho$, $\varphi (k x) = \varphi(x)$ and $\varphi (x k) = k \varphi(x) k^{-1}$.
Therefore $\varphi$  induces a map from the space of double cosets\footnote{A graphical way to encode these double cosets is described in  \cite{MacDo}, p. 401,  see also  \cite{AkerCan}.} $\CC_\rho \backslash S_{2n}/\CC_\rho$ to the set of conjugacy classes in $S_{2n}$.  Actually, we shall see in Sect.\ \ref{gaugefixing} that the counter image $Y^\prime = \varphi^{-1} ([n^2])$, considered as a subset of $\CC_\rho \backslash S_{2n}/\CC_\rho$,  contains only one element:  the double coset $\CC_\rho \backslash \beta /\CC_\rho$, where $\beta = (1,2,3,\ldots 2n)$. As a double coset, $Y^\prime$ is then a disjoint union of left cosets $\sigma \CC_\rho$, with $\sigma \in S_{2n}$ (they will be identified with the sets $V(r)$ of Sect.\ \ref{gaugefixing}), parametrized by the homogenous space $R=\CC_\rho/ \CC_\rho \cap \CC_\rho^\beta = \CC_\rho/D_n$, where $D_n$ is the dihedral group. The space $R$, with $\vert  \CC_\rho \vert / \vert D \vert$ elements, will be described in Sect.\ \ref{gaugefixing} as parametrizing the ``gauge condition''  (in physicist's parlance).

\paragraph{Orbits of $Y^{\prime}$.} 
  One finds $\card{ Y'}=2^{2n-1}(n-1)! n!$, see  Appendix C.4 for a proof   based on a simple integral. 
 In practice, the orbits of $Y^{\prime}$  are  obtained by  methods similar to those of Sect.\ \ref{Xpmethod}, 
 (see also Appendix A).
\paragraph{Orbits of $Y^{\prime \prime} = Y^\prime_0$ and of $Y^\prime_g$.} 
Once the orbits of $Y^{\prime}$ are known, filtering 
according to their genus yields the orbits of each $Y^\prime_g$, in particular of $Y''=Y^\prime_0$.
For $n=10$, we had to rely on a random sampling method (see Appendix A), but as we have no {\it a priori} knowledge of $\card{ Y''}$, we have no way to check the correctness of the result.
 The figures entered in red in Table \ref{TableY} below, for $n=10$,  are thus likely estimates, awaiting an independent confirmation.

\paragraph{Orbit lengths.} Lengths of orbits of  $Y^\prime$
may be read off the following table, 
with the notation $k^{\# {\rm orbits \ of\ length\ } \card{\CC_\rho}/k}$
 $$\begin{matrix} n=1 & 2^2  \\ 
 n=2 & 1^1 & 2^2\\
 n=3 &1^4 & 2^6 & 3^2 & 6^2 \\ 
 n=4 & 1^{44} & 2^6 &4^4\\  
 n=5& 1^{352}& 2^{62}& 5^{4}& 10^2 \\
n=6& 1^{3803}& 2^{62}& 3^{15}& 6^6 \\
n=7&1^{45696}& 2^{766}& 7^6& 14^2 \\
 n=8&1^{644736}& 2^{752}& 4^{28}& 8^8 \\
n=9& 1^{10315716}& 2^{12264}& 3^{202}& 6^{22}& 9^8& 18^2 \end{matrix}$$

\paragraph{Results.} 
The numbers of orbits for $g=0$ are given in Table \ref{TableY}; 
for higher values of $g$ they are gathered in Table \ref{TableXYZ1}, under the entry UOc. 

%%%%%%%%%%%%%%%%%%%%%%%%%%%%%%%%%%%%%%%%%%%%%%%%%
\afterpage{%
    \clearpage% Flush earlier floats (otherwise order might not be correct)
\begin{landscape}
%\newpage
\def \t{} 
\def\t{\tiny}
  %%%%%%%%%%%%%%%%%%%%%%%%BEGIN_TABLE_CODE_FROM_Y_subsets
  \begin{table}[ht]
 \captionof{table}{Orbits for $Y$ subsets.  \\
The numbers in blue give the asymptotic estimate of the number of orbits. \\
Numbers of spherical UO bi{\colour}ed immersions appear on the line $\#\, \CC_\rho$-orbits in $Y^{\prime \prime}$. \\
Total numbers of UO bi{\colour}ed immersions (all genera): line $\# \, \CC_\rho$-orbits in $Y^\prime$. \\
Last entry  $F_n^{(0,1)}$ of the Table is defined in Appendix C.\\ {Here and below in this paper,  figures in red
 are still awaiting  confirmation, see above and Appendix A for explanations.} }
\centering
 % \hglue-15mm  
 {\tiny
  \begin{tabular}{ | l || c|c|c|c|c|c|c|c|c|c|}
       \hline
         $n$ & 1& 2 & 3 & 4 & 5 & 6&7&8&9&10\\
      \hline
      $\card{\CC_\rho}= 2^n n!$ &2 & 8 &  48 & 384 & 3840 & 46\,080 &  {645\,120} 
      & 10\,321\,920 & 185\,794\,560& 3\,715\,891\,200  \\
      \hline
      $\card{Y}=\card{S_{2n}}=(2n)!$ & 2 & 24 & 720 & 40\,320& 3\,628\,800&  {\t 479\,001\,600} &  {\t 87\,178\,291\,200}
    &  {\t 20\,922\,789\,888\,000} & {\t 6\,402\,373\,705\,728\,000}&2\,432\,902\,008\,176\,640\,000
      \\ 
      \hline
        $\# \CC_\rho$-orbits in $Y=S_{2n}$  &2&8&34&182&1300&{12\,634}& {153\,598 } & 2\,231\,004 &
                    37\,250\,236 &  699\,699\,968
        \\ \hline
 $\card{Y'}=2^{2n-1}(n-1)! n!$ & 2 & 16 & 384 & 18\,432 & 1\,474\,560 &  {\t 176\,947\,200} &  {\t 29\,727\,129\,600}&
  {\t 6\,658\,877\,030\,400 } &{\t 1\,917\,756\,584\,755\,200} &{690\,392\,370\,511\,872\,000}
 \\ \hline
 \rowcolor{Apricot}   $\#$ $\CC_\rho$-orbits in $Y^{\prime}$  \hfill (I')\quad  & 2 & 3 & 14 & 54 & 420 &  3886 & {46\,470} & 645\,524&{10\,328\,214}& \\ 
  \hline
  \Blue{$\card{Y'}/ \card{\CC_\rho}=2^{n-1}(n-1)!}$ &\Blue1& \Blue 2& \Blue 8& \Blue{48}& \Blue{ 384}& \Blue{ 3840}& \Blue{ 46\,080} 
  &\Blue{645\,120}&\Blue{10\,321\,920}  & \Blue{185\,794\,560}
        \\
\hline
 $\card{Y''}$  & 2 & 16& 336 & 12\,480  & 689\,664 &
 {\t 51\,440\,640 } &  {\t 4\,870\,932\,480} &561\,752\,432\,640 &76\,597\,275\,525\,120 & \red{12\,077\,498\,082\,263\,040}
  \\ \hline
  \rowcolor{Apricot}   $\#$ $\CC_\rho$-orbits in $Y^{\prime\prime}$ 
\hfill (I')$\cap$(II')  
&2&3&12&37&198 & {1143 } &{7658} & 54\,559 & {413\,086} & \red{3\,251\,240}
 \\
\hline
 \Blue{$\card{Y''}/\card{\CC_\rho}$} & \Blue{1.} & \Blue{2.} & \Blue{7.}& \Blue{32.5} & \Blue{179.6} & \Blue{1116.33} & \Blue{7550.43} &
 \Blue{54423.3}&\Blue{412\,268.66}&
 \Red{3\,250\,229.2}\\ \hline
 $F_n^{(0,1 )}= 
  \frac{\card{Y''}}{2 (2n)!!} $ & $\frac{1}{2}$& 1 &$\frac{7}{2}$ &$\frac{65}{4}$ & $\frac{449}{5}$& $\frac{3349}{6}$&
$\frac{ 52\,853}{14}$  & $\frac{217693}{8}$ & $\frac{618403}{3}$ & \red{$\frac{8125573}{5}$} \\ 
\hline 
     \end{tabular}    
         \label{TableY}}
      \end{table}
   %%%%%%%%%%%%%%%%%%%%%%%%END_TABLE_CODE_FROM_Y_subsets
  %%%%%%%%%%%%%%%%%%%%%%%%BEGIN_TABLE_CODE_FROM_Z_subsets
    %\newpage
 \begin{table}[ht]
\captionof{table}{Orbits for $Z$ subsets. \\
Numbers of spherical OO immersions: line $\#\, \CC^\prime_\rho$-orbits in $Z^{\prime \prime}$. \\
Total numbers of general OO  immersions (all genera): line $\#\,\CC^\prime_\rho$-orbits in $Z^\prime$.}
\centering
    \def\t{\tiny}
    {\tiny
 \begin{tabular}{ | l || c|c|c|c|c|c|c|c|c|c|} 
       \hline
         $n$ & 1& 2 & 3 & 4 & 5 & 6&7&8&9&10\\
      \hline
      $\card{ \CC'_\rho}= n!$ &1 & 2 &  6 & 24 & 120 & 720 &   5040 & 40\,320 & 362\,880 & 3\,628\,800
        \\        \hline
 $\card{ Z'}=(2n-1)! $ & {\t 1} & {\t 6} & {\t 120} & {\t 5\,040}  & {\t 362\,880} &  {\t  39\,916\,800} &  {\t 6\,227\,020\,800}&
  {\t 1\,307\,674\,368\,000} &{\t 355\,687\,428\,096\,000} & {\t 121\,645\,100\,408\,832\,000}
     \\ \hline
   \rowcolor{Apricot}  $\# \CC'_\rho$-orbits in $Z^{\prime}$ \quad  & 1 & 4 & 22 & 218 & 3028 & 55540  & {1\,235\,526} &{32\,434\,108} &{980\,179\,566
}& 33\,522\,177\,088\\ 
  \hline
  \Blue{$\card{ Z'}/ \card{ \CC'_\rho}=\frac{(2n-1)!}{n!}}$ &\Blue{\t 1}&\Blue{\t 3} & {\Blue{\t 20}}& \Blue{\t 210}&\Blue{\t 3024}& \Blue{\t 55\,440}& \Blue{\t 1\,235\,520} 
  &\Blue{\t 32\,432\,400}&\Blue{\t 980\,179\,200}  & \Blue{\t 33\,522\,128\,640}
   \\ \hline
    $\card{ Z''}$ &1&4&42&780&21\,552& 803\,760&    {\t 38\,054\,160}    &{\t 2\,194\,345\,440}&{\t 149\,604\,053\,760}
    & \red{11\,794\,431\,720\,960}\\ \hline 
  \rowcolor{Apricot}  $\#$ $\CC'_\rho$-
orbits in $Z^{\prime\prime}$ 
&1&3&9&37&{182}  & {1143 } &{ 7553} & {54\,559} & 412\,306 & \red{3\,251\,240} 
    \\ \hline
   \Blue{$\card{ Z''}/ \card{ \CC'_\rho} =2 F_n^{(0,1)}$} &\Blue{\t 1}&\Blue{\t 2}&\Blue{\t 7}&\Blue{\t $\frac{65}{2}$}&\Blue{\t $\frac{898}{5}$}&\Blue{\t $\frac{3349}{3}$}&\Blue{\t $\frac{52853}{7}$}&\Blue{\t $\frac{217693}{4}$}&\Blue{\t $\frac{1236806}{3}$} &\Red{\t $\frac{16251146}{5}$ }\\ \hline 
     \end{tabular}}
     \label{TableZ} 
     \end{table}
  %%%%%%%%%%%%%%%%%%%%%%%%END_TABLE_CODE_FROM_Z_subsets
    \end{landscape}
    \clearpage% Flush page
}
 %%%%%%%%%%%%%%%%%%%%%%%%%%%%%%%%%%%%%%%%%%%%%%%%

 %%%%%%%%%%%%%%%%%%%%%%%%%%%%%%%%%%%%%%%%%
\subsection{The left coset $U = \beta \, \CC_\rho$ and immersions of type UOc and OOc (``U method")}
\label{gaugefixing}
In a nutshell : we shall see in this section that, in order to determine the number of immersions of type UOc, we can replace the set $Y^\prime$ studied in the previous section by a particular subset $U$ (a particular left coset of $\CC_\rho$) 
and the adjoint action of $\CC_\rho$ by its restriction to the dihedral subgroup $D_n$, which is much smaller. Moreover, 
by replacing the adjoint action of $D_n$ by the adjoint action of $\mathbb{Z}_n$ (a particular cyclic subgroup of the latter), one obtains the number of immersions of type OOc. 
In Sect.\ \ref{fromorbtoimm} we shall see how, from this study, and by introducing several involutions, one can obtain the various types of immersions. 
As a side result we shall also see how the stratification of $U$ into subsets of genus $g$ allows us to recover (in genus $0$) the classification of the so-called ``long curves'' and to obtain new classifications when $g>0$.

\paragraph{a. The set $R$.}
With  $\rho=(1,2)(3,4)\cdots (2n-1,2n)$ fixed as before, what can be said about the values
of $\tilde \rho= \sigma \rho\sigma^{-1}$ as $\sigma\in Y'$ ? Consider the sets
\be R:=\{\tilde\rho | \tilde\rho \rho\in [n^2]\} \ee
and for $r\in R$, 
\be V(r):=\{\sigma | \sigma \rho\sigma^{-1} =r\} \,. \ee
It is readily seen that $V(r)$ is a left coset of $\CC_\rho$, since
$\sigma,\sigma'\in V(r) \Leftrightarrow \sigma \rho\sigma^{-1}=\sigma' \rho\sigma'^{-1}
 \Leftrightarrow \sigma'^{-1}\sigma \rho=\rho \sigma'^{-1}\sigma$,
 hence $\sigma'^{-1}\sigma\in \CC_\rho$ and $\sigma \in \sigma'\CC_\rho$. This property 
 of being a left coset will be used shortly. 
This implies that $\card{V(r)}=\card{\CC_\rho}$ and from the fact that $Y'$ may be partitioned 
 into $V(r)$, $Y'=\sqcup_{r\in R} V(r)$, it follows, using the values of $\card{Y'}$ and $\card{\CC_\rho}$
calculated above, 
 that $\card{R}= \card{Y'}/\card{\CC_\rho}= 2^{n-1}(n-1)!$.

\paragraph{b. Further gauge fixing.}
One may now restrict further the set of admissible $\sigma$ by imposing 
the additional condition (on top of $\rho$ fixed as above) 
$$ \tilde\rho \rho =\sigma \rho\sigma^{-1}\rho=\alpha \ {\rm fixed\ in\ } R\rho\,, $$
or equivalently $\sigma \in V(\alpha \rho)$.  {\it For example} one may demand that 
$ \sigma \rho\sigma^{-1}\rho$ be the product of the two cycles
\be\label{newgauge} \sigma \rho\sigma^{-1}\rho=\alpha_0:=(1,3,5,\cdots,2n-1)(2, 2n,2n-2,\cdots,4)\,.\ee
This latter choice $\alpha_0$ corresponds to a sequential labelling of edges 
by $(1,2,3,\cdots 2n)$ as the curve is travelled  one way or the other.
We call $U$ the set of $\sigma$ such that 
\be\label{newconstr} U=\{ \sigma| \sigma \rho \sigma^{-1}\rho=\alpha_0\}=V(\alpha_0 \rho)\,, \ee
and we recall that $$ |U|=|C_\rho|=2^n n!\,.$$
\begin{myprop}The general solution of (\ref{newconstr}) is 
%\be
\label{generalsol} 
$\sigma= \beta \xi\,,$
%\ee
  with $\beta$ the cyclic permutation 
 % \be
  \label{beta}
$\beta=(1,2,3,\cdots 2n)$
%  \ee
   and $\xi$ arbitrary in $\CC_\rho$.  
In other words, $U=\beta \CC_\rho$, a particular $\CC_\rho$-left coset, in agreement with the previous argument. 
\end{myprop}

\begin{proof} It is easy to check that $\beta\rho \beta^{-1}\rho=\alpha_0$, hence upon the change of
variable $ \sigma= \beta \xi$,  equ.\ (\ref{newconstr})
reads  $\beta\xi \rho \xi^{-1}\beta^{-1}\rho =\beta\rho \beta^{-1}\rho$, hence $ \xi \rho \xi^{-1} =\rho $, $\xi\in \CC_\rho$.
\end{proof} 
%The parametrization (\ref{generalsol}) of $\sigma$  as an element of the left coset $U$ therefore automatically 
The above parametrization of $\sigma$  as an element of the left coset $U$ therefore automatically 
implies condition (I$^\prime$). This is very useful in practice (see a comment at the end of Appendix A).

\paragraph{c. The remaining reparametrization groups.}
In this new ``gauge", the remaining labelling freedom 
 on a given $\sigma$ is the choice of the 
origin (edge number 1), and the direction of travel if one considers unoriented curves. 
Accordingly the group of reparametrization, $\CC_\rho\cap \CC_\alpha$, 
where $ \CC_\alpha$ is the centralizer of $\alpha$ in $S_{2n}$,
is the dihedral group $D_n$ (of order $2n$) 
if one considers unoriented curves, and the cyclic group $\mathbb{Z}_n$ if the curves are
oriented. Unlabelled curves are thus in one-to-one correspondance with orbits 
of the set $U$ under the adjoint action of $D_n$ (unoriented curves) or of $\mathbb{Z}_n$ (oriented ones).\\
Remark. This occurrence of the dihedral or cyclic group makes clear that the length of orbits which must be 
divisors of the orders of these groups, are divisors of $2n$ or $n$, a ``well known" fact.\\
Warning:  $Y^\prime$ is stable under the adjoint action of $\CC_\rho$ and can be decomposed into the corresponding orbits, but its subset $U$ is not stable under this action, although it intersects all the orbits of $Y^\prime$ (not only once,  in general);
$U$ is however stable under the action of $D_n$. See Appendix A for more details.

\paragraph{d. Back to $\vert Y^\prime \vert $ and $\vert R \vert $.}
 The number of left cosets contained in a double coset $K \backslash g / K$, for $g$ an element of a group $G$,  and $K$, a subgroup of $G$, is equal to the index, in $G$, of the subgroup $K \cap K^g$, where $K^g = g K g^{-1}$. In the present situation, with $G=S_{2n}$,  $K = \CC_\rho$, and $g = \beta$ (the above cyclic permutation),  we have $\CC_\rho \cap \CC_\rho^\beta=D$ where $D$ is the dihedral subgroup
 of $\CC_\rho$ (see also \cite{MacDo}, p. 402). The previous index is therefore $\vert \CC_\rho \vert/2n$. Since all $\CC_\rho$ left cosets  have the same number of elements, the number of permutations contained in the double coset  $Y^\prime = \CC_\rho  \backslash \beta / \CC_\rho$  is  equal to  $ \vert \CC_\rho \vert \times \vert \CC_\rho\vert /2n $: we recover the number of elements of $Y^\prime$.  As a double coset, $Y^\prime$ is a disjoint union of left cosets $V(r)$ parametrized by the homogenous space $R=\CC_\rho/ \CC_\rho \cap \CC_\rho^\beta = \CC_\rho/D$.  

 \paragraph{e. The set $U_g$ of long curves.}
 The set $U$ just defined may be partitioned into $U_g$ according to genus, as was done
 before for $Y'$, and each $U_g$ may be interpreted as the set of {\it rooted 
 maps} on an oriented surface $\Sigma$ of genus $g$,  or in other words, of (equivalent classes of)
 {\it open} (and oriented) curves drawn in $\Sigma$, sometimes dubbed
 {\it long curves}. In genus $g=0$ their number have been
 computed in \cite{GZD} up to $n=10$, and in \cite{JZJ} up to $n=19$ crossings using transfer matrix techniques. (Their
asymptotic behavior has also been studied using a method of random sampling \cite{SZJ}.)\\
\begin{proof}
 Consider  a {\it rooted} 4-valent genus $g$ map with $n$ crossings and one component: 
 the marked half-edge that we label by 0 may be regarded as cut open, which transforms the map into  a ``long curve".
 We then label by $1,2,\cdots , 2n$ the successive edges encountered along the curve. The curve may then be 
 bi{\colour}ed  by assigning to the left of the marked edge the {\colour} say white, and then alternating {\colour}s 
 as we go from a face to an adjacent one.  To completely describe the pattern of crossings of the curve, 
 it remains to give the permutation 
 $\sigma$ satisfying the rules of the previous formalism, namely conditions  (\ref{newgauge}) and genus $g$. 
 There is a bijection between these rooted maps and
 elements of the set $U_g$. There is no reparametrization freedom left, hence no orbit to take, once 
 the root has been fixed. Then any such open curve may be closed by identifying edges of labels 0 and $2n$. 
 Topologically distinct closed curves, \ie images of immersions, correspond to orbits of the set  $U_g$
 by the reparametrization group, namely the cyclic group $\mathbb{Z}_n$ or the dihedral group $D_{2n}$ depending on whether 
 the curve is oriented or not (OOc resp. UOc).
 \end{proof}

\noindent Thus one finds a decomposition of the $2^n n!$ curves of $U$, ($n=0,1,\cdots,9$), according to genus as 
\begin{equation} {
\footnotesize
 \# {\rm \ open\ curves\ } 
= \pamatrix{1\\2\\ 8\\ 48\\ 384 \\ 3840\\ 46\,080 \\ 645\,120\\ 10\,321\,920\\ 185\,794\,560}=
\pamatrix{ 1\\ 2& \\  8& \\ 42&6 \\ 260&116 & 8 \\ 1796 &1700 & 344\\ 13\,396 & 22\,528 & 9700 & 456\\ 
105\,706&    284\,284 & 220\,570 & 34\,560 \\  870\,772 & 3\,488\,904 & 4\,392\,820 & 1\,506\,576 & 62\,848 \\
7\,420\,836& 42\,074\,568&79\,951\,716 &49\,572\,528 & 6\,774\,912}  }
\end{equation}
with the first column (genus $g=0$) in agreement with \cite{GZD, JZJ}.
Notice that the sum over all genera is of course equal to $\vert \CC_\rho\vert=2^n n!$.

\paragraph{f. Orbits of $U$ and UOc and OOc immersions.}
 \label{OOcimm}
The same sort of counting of  orbits 
 that was done  in the sets $Y^\prime$ and $Y^\prime_g$ 
may be carried out in the sets $U$ and  $U_g$. 
{}From the  previous discussion 
it follows that UOc immersions are orbits of $U$ under the action of $D_n$
while its $\mathbb{Z}_n$-orbits are
 what may be called {\it OOc immersions}. The numbers of UOc immersions have been
 computed before,  see Table \ref{TableY}, using the $\CC_\rho$ action on $Y^\prime$, but 
 can be recovered in a more economic way, using the $D_n$ action on the set $U$.
 Here are the numbers of OOc immersions and their distribution according to genus
 for $n=1,\cdots, 9$.
 \be {\footnotesize \#{\rm  \ curves\ of\ type\ OOc\ }=
 \pamatrix{2\\ 6\\ 20 \\ 108\\ 776 \\ 7772\\ 92\,172\\1\,291\,048\\ 20\,644\,140 }=\pamatrix{2\\ 6\\ 18& 2\\ 74 & 32 & 2\\
 364 & 340 & 72 \\ 2286 & 3780 & 1630 & 76\\ 15\,106& 40\,612 & 31\,510 & 4944\\
 109\,118& 436\,368 & 549\,334 & 188\,356 &7872 \\ 824\,612 &4\,675\,012 & 8\,883\,620 &5\,508\,120 &752\,776}} \ee
 See also Table \ref{TableXYZ1}. 
 For even $n$, the numbers of such orbits are just the double of those of type UOc, while 
 for $g=0$ these numbers are the double of OO immersions, see the proof below in Sect.\ \ref{discrOO},
 Theorem \ref{Theo3c}.

%%%%%%%%%%%%%%%%%%%%%%%%%%%%%%%%%%%%%%%%%%%%

\section{From orbits to various types of immersions}
\label{fromorbtoimm}

 \subsection{Preamble}\label{preamble}
In this section we examine the effect of three involutive transformations on orbits of bi{\colour}ed immersions: the 
 {\colour} swapping or {\it swap} in short, denoted by $s$; the {\it mirror} transformation, $m$; and the {\it orientation reversal}
 $r$. These three involutions commute. Their explicit form depends  on the class of orbits on which they act, as 
 we shall see below. Given an orbit $o$ belonging to a set $O$ and an involution $I$, if $o_I$ denotes the transform of
 $o$ under $I$, there are two cases: either $o=o_I$, or $o\ne o_I$,  a truism !, %(La-Palisse theorem \cite{LaPalisse}), 
 and we define
 \be\label{rsI} r_I= \#\{ o\in O| o=o_I\}\,,\quad s_I=\#\{\{o,o_I\} | o\ne o_I\}\, \ee
 (\ie, $s_I=\#$ {\it unordered} pairs of distinct $o,o_I$). 
  In the case of two commuting involutions $I$ and
 $J$, there are five cases:
 \begin{equation*}
 \begin{split}
 1) \quad o=o_I=o_J=o_{IJ}\,, & \qquad
 2) \quad o=o_I\ne o_J=o_{IJ}\,,
\\
3) \quad o=o_J\ne o_I=o_{IJ}\,,& \qquad
4) \quad o=o_{IJ}\ne o_I=o_{J}\,,
 \\
5)  \quad o,o_I, o_J,o_{IJ}\ & {\rm all\ distinct}\,,  
  \end{split}
 \end{equation*}
 and we call 
  \bea \nonumber
 x_{IJ}=x_{JI} &=& \#\{o\in O | o=o_I=o_J 
 = o_{IJ}          \}\\ \nonumber
 y_{IJ}&=& \#\{\{o,o_J\}| o=o_I\ne o_J=o_{IJ}\}\\ \label{xyzvwIJ}
  z_{IJ}=y_{JI}&=& \#\{\{o,o_I\}| o=o_J\ne o_I=o_{IJ}\}\\ \nonumber
   v_{IJ}=v_{JI}&=& \#\{\{o,o_{I}\}| o=o_{IJ}\ne o_I=o_J\}\\ \nonumber
   w_{IJ}=w_{JI}&=& \# \{ \{o,o_I, o_J, o_{IJ}\} |  o,o_I, o_J,o_{IJ}\ {\rm all\ distinct }\}
 \eea
 For $I$ and $J$ standing for  the mirror and the orientation reversal, those are the five cases  
 discussed by Arnold \cite{Arnold}. We note that the relation between (\ref{rsI}) and (\ref{xyzvwIJ}) is
 $$ r_I= x_{IJ}+ 2 y_{IJ}\,,\quad s_I= z_{IJ}+ v_{IJ}+2w_{IJ}\,. $$
 For three 
 involutions, there would be 15 cases (in general, the number of cases is given by a Bell number) but we shall refrain from listing them here. 
 
 %%%%%%%%%%%%%%%%%%%%%%%%%
 
\subsection{The \dual image of a map}
\label{bicoloredtobicolorable}
 
We first examine the effect of ({\colour}) swapping (or equivalently, of interchanging  all overcrossings and undercrossings
in a knot diagram). 
Consider a bi{\colour}ed curve described by some $\sigma\in Y'$ and its ({\colour}) swap described by $\sigma_s$. 
What is the relation 
between $\sigma$ and $\sigma_s$? Let $\rho$ be fixed equal to $\rho_0$ as above,
$\tau=\sigma \rho_0$ and $\tilde \rho= \sigma \rho_0 \sigma^{-1}$. 
Then  swapping {\colour}s implies to exchange 
$\rho$ and $\tilde \rho$, and $\sigma$ and $\tau^{-1}$, 
but also to change the 
labelling of edges in such a way that $\tilde \rho$ takes the form $\rho_0$. 
A  permutation $\gamma$ that carries over that change of labelling must satisfy
$$ \tilde \rho= \sigma \rho_0 \sigma^{-1}=\gamma^{-1}\rho_0 \gamma\,,$$
{the general solution of which is $\gamma= \gamma'\sigma^{-1}$ with $\gamma'\in\CC_\rho$. Up to $\CC_\rho$-equivalence
we may just choose $\gamma=\sigma^{-1}$.
Then after conjugation by $\gamma$, 
\bea\nonumber \sigma_s&=&\gamma \tau^{-1}\gamma^{-1}=\gamma \rho_0 \sigma^{-1}\gamma^{-1}=
\gamma \sigma^{-1} \gamma^{-1}\, \gamma \sigma \rho_0\sigma^{-1}\gamma^{-1}\\
&=&
\gamma \sigma^{-1} \gamma^{-1}  \, \gamma \tilde \rho \gamma^{-1} =\gamma\sigma^{-1}\gamma^{-1}
\rho_0\,,\eea
which for the above choice $\gamma= \sigma^{-1}$ 
reduces to
\be \label{newsigma}\sigma_s= \sigma^{-1} \rho_0 
\Longleftrightarrow\ \sigma \sigma_s =\rho_0\,.  \ee
Hence the {\colour}ed curve (or the alternating knot 
diagram) and its swapped version are described by $\sigma$ and $\sigma_s=\sigma^{-1} \rho_0$. 
We refer to the $\CC_\rho$-orbits of $\sigma$ and $\sigma_s$ as {\it swapped orbits} $o$ and $o_s$.

If $n$ is odd, the signature of $\rho_0$, a product of an odd number of transpositions, is $-1$, and  $\sigma$ and $\sigma_s=
\sigma^{-1}\rho_0$ cannot be conjugate
in $S_{2n}$, and {\it a fortiori} cannot belong to the same orbit  under the action of $\CC_\rho$: 
$\sigma\nsim \sigma_s$, 
where $\sim$ and its negate $\nsim$ refer to conjugacy with respect to the group $\CC_\rho$.
Another argument is that 
$\sigma\sim \sigma_s$ would imply that the numbers of white and shaded faces are equal, 
hence $\#$ faces is even, in contradiction with Euler formula for $n$ odd.

In general, using the terminology of (\ref{rsI}), 
for given $n$ and  genus $g$, let $r_s$ be the number of {\it self-swapped orbits}, \ie such that $o=o_s$, 
and $s_s$ be the number of {\it pairs} of non self-swapped orbits $\{o,o_s\}$, \ie such that $o\ne o_s$. 
Thus $r_s=0$ for $n$ odd and all genera, while for example, in genus 0, we find 
\bea\nonumber
n=2 & r_s=1 & s_s=1 \\ \nonumber
n=4 & r_s=5 & s_s=16 \\
n=6 & r_s=33 & s_s=555\\ \nonumber
n=8 & r_s=249 & s_s=27\,155\\ \nonumber
n=10 & r_s={2036} & s_s=1\,624\,602\,.\nonumber
\eea

For any genus $g$, the number of $Y^{\prime}_g$ orbits, \ie of {\it bi{\colour}ed UO curves} of genus $g$ is thus given by $r_s+2s_s$, while those 
in which we identify the two {\colour}s, namely the {\it bi{\colour}able} UO curves, have a cardinality equal to $r_s+s_s$.
As we discussed already, for $g=0$, bi{\colour}ability is not a constraint,  and we recover the number of 
UO curves found in Sect.\ \ref{Xpmethod}, while for $g>0$, the UOc bi{\colour}able curves are a subset of the UO curves, 
see below Sect.\ \ref{UOUUimm} for a general discussion.

\normalcolor

\subsection{Mirror image of a map}
On maps/orbits of $Y'_g$ we may also define a mirror transformation.
The latter swaps $\sigma$ and $\tau$, hence, if $\rho=\rho_0$ is fixed, changes $\sigma$ into $\sigma
\rho_0$. Maps are either  ``achiral", if $\sigma$ and $\sigma_m:=\sigma\rho_0$ belong to the same orbit, and
we write $o=o_m$, 
or appear in chiral pairs $\{\sigma, \sigma_m\}$, when 
$\sigma_m\,\, \nsim\sigma$, or $o\ne o_m$. 
Again, for $n$ odd, as $\rho_0$ has an
odd signature,  $\sigma$ and $\sigma_m$ cannot belong to the same orbit. 

In general, for given $n$ and  genus $g$, let $r_m$ be the number of achiral orbits, \ie such that 
$o=o_m$, and $s_m$ be the number of chiral {\it pairs} of  orbits $\{o,o_m\}$, $o_m\neq o$. 
Thus $r_m=0$ for $n$ odd and all genera, while for example, in genus 0, we find 
\bea\nonumber
n=2 & r_m=1 & s_m=1 \\ \nonumber
n=4 & r_m=5 & s_m=16 \\
n=6 & r_m=15 & s_m=564\\ \nonumber
n=8 & r_m=97 & s_m=27\,231\\ \nonumber
{n=10} & r_m=592 & s_m=1\,625\,324\,. 
\eea

The number of orbits in $Y'_g$, \ie of bi{\colour}ed UOc curves of genus $g$ is thus given by $r_m+2s_m$, while 
those in which we identify the two mirror images, \ie the two orientations of the target surface, dubbed UUc, 
have a cardinality equal to $r_m+s_m$, see below Sect.\ \ref{UOUUimm} for a general discussion.

\subsection{Discrete operations on UOc immersions: from UOc to UOb, UUc and UUb} 
\label{UOUUimm}

Following the discussion of Sect.\ \ref{preamble}, we may analyse the interplay between swap and mirror transformations
on $\CC_\rho$-orbits of $Y'_g$ (UOc immersions) by introducing 
\begin{eqnarray*}
x_{sm} &=&\# \{\text{orbits}\ o \,\vert\,  o=o_s=o_{m}=o_{sm}  \}  \; 
, \ie  \\
{} &=& \#\{ \text{orbits that are both achiral and self-swapped}\},\\
y_{sm} &=&\#  \{\text{unordered pairs}\ \{o,o_m\} \,\vert \,   o=o_s\ne  o_{m}=o_{sm}\}  \; 
,\ie \\
{} &=&  \#\{\text{chiral pairs of self-swapped orbits}\},\\
z_{sm} &=& \#  \{\text{unordered pairs}\ \{o,o_s\}  \,\vert\, o=o_{m}\ne o_s=o_{sm} \}  \;  
,\ie \\
{} &=& \#\{ \text{\dual pairs of achiral orbits}\}, \\
v_{sm} &=&  \#  \{\text{unordered pairs} \; \{o, o_s\} \;  
 \vert \; o \neq o_s ,\;  {o=o_{sm}} \; \text{and} \;  o_s=o_{m}\}, \\
w_{sm} &=&  
\#\{ \text{4-plets of orbits} \; \{o, o_s, o_{m}, o_{sm} \} \; \vert o, o_s, o_{m}, o_{sm}\; \text{all distinct}\},
\end{eqnarray*}
hence
$$ \begin{array}{l c r}
r_s &=& x_{sm}+2y_{sm}\\
s_s &=& z_{sm} +v_{sm}+2w_{sm}\\
r_m &=& x_{sm}+2z_{sm}\\
s_m &=& y_{sm} +v_{sm} +2w_{sm}
\end{array}\,$$
In particular for $n$ odd, the vanishing of $r_s$ and $r_m$ implies %that 
$x_{sm}=y_{sm}=z_{sm}=~0$. 

The five independent quantities $x_{sm},\,y_{sm},\,z_{sm},\,v_{sm},\, w_{sm}$ must be determined in each $Y'_g$, 
their values are gathered in Appendix B.1.
Then counting how many times each class of orbits contributes to each type of immersions,
one obtains, for every genus:
\begin{eqnarray}\nonumber
\card{\rm{UOc}} &=& r_s+2s_s=r_m+2s_m = x_{sm}+2y_{sm}+2z_{sm}+2v_{sm}+4w_{sm}\\
\nonumber
\card{\rm{UOb}} &=&  r_s+s_s =x_{sm}+2y_{sm}+z_{sm}+v_{sm}+2w_{sm}\\
\label{uocetal}
\card{\rm{UUc}} &=& r_m+s_m=
x_{sm}+y_{sm}+2z_{sm}+v_{sm}+2w_{sm}\\
\nonumber
\card{\rm{UUb}} &=& x_{sm}+y_{sm}+z_{sm}+v_{sm}+w_{sm}
\end{eqnarray}

 For example in genus 0, and for $n=1,\dots,10$, we obtain:
  \be\label{immersions3} 
   \begin{split}
&\mathrm{Unoriented\ }S^1 \mathrm{\ in\ unoriented\ } S^2: \\
\# & \mathrm{UU\ immersions\  } = {1,2,6,19,76,376, 2194, 14614, 106\,421, {\red{823\,832}}}
    \end{split} 
      \ee
 thus extending the OEIS sequence A008989 \cite{OEIS} and Valette's recent results \cite{Valette}. 

 \normalcolor
 {}From the values given in Appendix B.1 we see that
   $x_{sm}=y_{sm}=z_{sm} =0$ for odd $g$, an empirical observation for which we have no explanation yet.  
 
  %%%%%%%%%%%%%%%%%%%%%%%%%%%

\subsection{Discrete operations on OOc immersions: from OOc to OOb, UOc, UOb, OUc and OUb}
\label{discrOO}
The previous discussion, that was applied  to the set $Y^\prime$  and its $\CC_\rho$ orbits
(or, equivalently, to the set $U$ and its $D_n$-orbits)
of type UOc may be applied 
to the set $U$ of Sect.\ \ref{gaugefixing} and its $\mathbb{Z}_n$-orbits of type OOc.

\begin{myprop}
 For $\sigma $ belonging to some $\mathbb{Z}_n$-orbit of $U$\\
(i) $\sigma\mapsto \sigma_r:=r \sigma r$ belongs to the reversed orbit, with 
$r=[2n,2n-1,\cdots,2,1]$, (remember that $r^2=1$);
\\
(ii) $\sigma\mapsto \sigma_{m}:=\tau=\sigma \rho$ belongs to the mirror image of the  orbit.\\
(iii) $\sigma\mapsto \sigma_{s}:= \beta^{-1}\tau^{-1} \beta=  \beta^{-1}\rho\sigma^{-1} \beta$ 
belongs to the \dual orbit, with $\beta$ the cyclic 
permutation $(1,2,\cdots, 2n)$ as above.\\
(iv) $\sigma\mapsto \sigma_{rm}:=r\tau r=r\sigma r \rho$ belongs to the reversed mirror orbit, 
and likewise for the other compositions 
of the commuting involutions $s,r,m$.\end{myprop}
\begin{proof}
In each case, it is clear that the transform of $\sigma$ carries out the required 
transformation.  The important point is that if $\sigma$ belongs to $U$, \ie satisfies (\ref{newgauge}), 
so do $\sigma_r, \sigma_m$ and $\sigma_s $. This is obvious for $\sigma_m$; it 
follows from the identities 
$r \rho r=\rho$ for $\sigma_r$ and $r\alpha_0 r=\alpha_0$ for $\sigma_r$;  and from
$\beta^{-1}\rho \beta \rho=\alpha_0$ for $\sigma_s $, if one remembers
that $\sigma=\beta \xi$, $\xi\in\CC_\rho$:
$$ \sigma_s \rho \sigma_s^{-1}\rho=\beta^{-1}\rho \sigma^{-1} \beta \rho \beta^{-1}\sigma \rho \beta \rho
= \beta^{-1} \rho \xi^{-1} \rho \xi \rho \beta \rho=\beta^{-1} \rho \beta \rho=\alpha_0\,.$$
\end{proof} 

Now define once again along the lines of (\ref{xyzvwIJ})
\bea \nonumber
x_{sr} &=& \# 
\{ \mathbb{Z}_n{\rm -orbits}\ o| o=o_s=o_r=o_{sr}\}\\ \nonumber
y_{sr} &=& \# 
\{{\rm pairs}\ \{o,o_r\}| o=o_s\ne o_r=o_{sr} \}\\ 
z_{sr} &=& \# 
\{{\rm pairs}\ \{o,o_s\}| o=o_r\ne o_s=o_{sr}  \}\\ \nonumber
v_{sr} &=& \# \{ {\rm pairs}\ \{o, o_{s}\} | o=o_{sr}\ne o_s=o_{r}\} \\ \nonumber
w_{sr}&=& \# \{ {\rm quadruplets}\ (o, o_{sr}, o_r, o_s), \rm{\ all\ non\  equal} \}
\eea
Their values are gathered in Appendix B.2.

Then
\bea \nonumber
\card{\rm OOc} &=& x_{sr} + 2y_{sr} + 2z_{sr} + 2v_{sr} + 4w_{sr}\\ \nonumber
\card{\rm OOb} &=& x_{sr} + 2y_{sr} + z_{sr} + v_{sr} + 2w_{sr}\\ \label{systxyzvwr}
\card{\rm UOc} &=& x_{sr} + y_{sr} + 2z_{sr} + v_{sr} + 2w_{sr}\\ \nonumber
\card{\rm UOb} &=& x_{sr} + y_{sr} + z_{sr} + v_{sr} + w_{sr}
\eea

A similar discussion can be carried out on the action of the involutions $s$ and $m$ on the orbits of OOc,
 expressing $\card{{\rm OOc}}$, $\card{{\rm OOb}}$, $\card{{\rm OUc}}$ and $\card{{\rm OUb}}$ in terms of
new numbers $x_{sm},\, y_{sm},\, z_{sm},\, v_{sm},\, w_{sm}$\footnote{By ``new", we mean that they are relative 
 to the OOc orbits and that their values differ from those defined and listed below in App. B.1.}}. 
  The values of these five parameters are gathered in Appendix B.3.
Then
 \bea \nonumber
\card{\rm OOc} &=& x_{sm} + 2y_{sm} + 2z_{sm} + 2v_{sm} + 4w_{sm}\\ \nonumber
\card{\rm OOb} &=& x_{sm} + 2y_{sm} + z_{sm} + v_{sm} + 2w_{sm}\\ \label{systxyzvwm}
\card{\rm OUc} &=& x_{sm} + y_{sm} + 2z_{sm} + v_{sm} + 2w_{sm}\\ \nonumber
\card{\rm OUb} &=& x_{sm} + y_{sm} + z_{sm} + v_{sm} + w_{sm}
\eea
 
 {}From the values given in Appendix B.2 and Appendix B.3, one observes
that $x_{sr}$, $x_{sm}$, $y_{sr}$ and $y_{sm}$  vanish for all $(n,g)$, meaning that $\sigma\sim
\sigma_s$ never occurs. Moreover, for $n$ even,  $z_{sr}=v_{sm}=0$, and for $n$ odd, $z_{sm}=v_{sr}=0$. Those are
general features:

\begin{myprop}\label{prop2}{If $\sim$ means the equivalence with respect to the adjoint action 
of $\mathbb{Z}_n$, \\
(i) 
for any $n$ and $g$, $\#\{\sigma\in U | \sigma\sim \sigma_s\}=0$, hence $x_{sr}=y_{sr}=x_{sm}=y_{sm}=0$;\\
(ii)
for any {\bf even} $n$ and any genus $g$, $\#\{\sigma\in U | \sigma\sim \sigma_r\}=0$, hence $z_{sr}=0$;
and $\#\{\sigma\in U | \sigma\sim \sigma_{sm}\}=0$, hence $v_{sm}=0$.\\
(iii) for any {\bf odd} $n$ and any genus, $\#\{\sigma\in U | \sigma\sim \sigma_m\}=0$, hence $z_{sm}=0$; 
and $\#\{\sigma\in U | \sigma\sim \sigma_{sr}\}=0$, hence $v_{sr}=0$.
 }
 \end{myprop}

\begin{proof} First, one notices that $\sigma\nsim \sigma_s$ is certainly true for $n$ odd, see 
Sect.\ \ref{bicoloredtobicolorable}. We thus turn to $n$ even. 
%We write $\sigma =\beta \xi$ as above in (\ref{generalsol}) and note that $\beta^2$ is a generator of the $\mathbb{Z}_n$ group and  that $\rho,\, \beta^2$ and  $\xi$ are
We write $\sigma =\beta \xi$ as  in Prop.\ \ref{generalsol} and note that $\beta^2$ is a generator of the $\mathbb{Z}_n$ group and  that $\rho,\, \beta^2$ and  $\xi$ are
in the centralizer $\CC_\rho$. By the homomorphism $\phi$ introduced in Sect.\ \ref{pairS2nCrho},
$\rho$ is mapped to the identity permutation of $S_n$ and
 $\beta^2$ to the cyclic permutation $(1,2,\cdots,n)$, which is odd for $n$ even.
Then\\ 
 (i) $\sigma\sim \sigma_s= \beta^{-1}\tau^{-1}\beta=\beta^{-1}\rho\, \xi^{-1}$ 
  means $\exists\, p\in \{0,\cdots,n-1\}$ s.t. $\beta^{2p}(\beta \xi)\beta^{-2p}= \beta^{-1}\rho\, \xi^{-1}$, or $\beta^{2p+2} \xi \beta^{-2p}= \rho\, \xi^{-1}$. 
  If we take the image of both sides by $\phi$, 
  the signature of the lhs is minus the signature of $\phi(\xi)$ while the rhs 
  has the signature of $\phi(\xi)$. There is a contradiction, q.e.d.\\
  (ii) Suppose likewise that $\sigma\sim \sigma_r= r \sigma r = \beta^{2p}(\beta \xi) \beta^{-2p} $.
  Conjugation of a permutation by $r$ shifts the labels by $-1$ and reverses its
  cycles; in particular $r\beta r=\beta^{-1}$. Thus the images of $r \xi r$ and $\xi$
  by $\phi$ 
 have the same signature.
 Notice that $r \sigma r = r \beta \xi r  = (r \beta r) (r \xi r)$, using the fact that $r^2=1$, 
 so that $\sigma_r = \beta^{-1}  r \xi r$. Supposing that $\sigma$ and $\sigma_r$ are 
 $\mathbb{Z}_n$-conjugates therefore amounts to supposing that $r \xi r$ is conjugate with 
 $\beta^{2p+2} \xi \beta^{-2p} = \beta^2 \beta^{2p} \xi \beta^{-2p}$. However the image 
 of $\beta^2$ by $\phi$ is odd for $ n$ even.
 This contradiction completes the proof of the first part of (ii). For the second part, $\sigma_{sm}\buildrel{?}\over{\sim} \sigma$, \ie
 $\sigma_{sm} =\beta^{-1}\xi^{-1}\buildrel{?}\over{=}
 \beta^{2p+1} \xi  \beta^{-2p}$, it leads to $\xi^{-2}\buildrel{?}\over{=}\beta^2$ again in contradiction with signatures for $n$ even, q.e.d.\\
  (iii) is  again a trivial consequence of the parity of permutations:  for $n$ odd,
 $\sigma\sim \sigma_m=\sigma \rho$, or $\sigma\sim \sigma_{sr}= r\beta^{-1} \rho \sigma^{-1}\beta r^{-1}$ are impossible, since $\rho$ is odd.
 \end{proof}

\begin{mytheo}\label{Theo3c} For any genus $g$, \\ 
\bea\label{ooc} \card{{\rm OOc}}&=&2\card{{\rm OOb}} \ {\rm for\ any\ } n\\
 \label{uoc} \card{\rm UOc}&=& \begin{cases}  \card{\rm OOb} &{\rm if}\ $n$\  {\rm even}\\
 2  \card{\rm UOb} & {\rm if}\ $n$\ {\rm odd}\end{cases}\,.\\
 \label{ouc} \card{{\rm OUc}}  &=&\begin{cases}
                                                      2 \card{{\rm OUb}} &{\rm if}\ $n$\  {\rm even} \\
                                                       \card{{\rm OOb}} &  {\rm if\ } n\ {\rm odd} \end{cases}\,.\\                                                       
\label{uuc} \card{\rm UUc}&=& \begin{cases} \card{\rm OUb} &{\rm if}\ $n$\  {\rm even}\\
 \card{\rm UOb} & {\rm if}\ $n$\ {\rm odd}\end{cases}\,.
 \eea
 \end{mytheo}

\begin{proof} Those are consequences of  relations (\ref{systxyzvwr}) and of Proposition \ref{prop2}:
 (\ref{ooc}) follows from $x_{sr}=y_{sr}=0$; (\ref{uoc}) follows from $z_{sr}=0$ for $n$ even and from $v_{sr}=0$ for $n$ odd.
For (\ref{ouc}), we perform a similar analysis relating the sets 
OOc, OOb, OUc, OUb: one finds that
$$2\card{{\rm OUb}} - \card{{\rm OUc}}= v_{sm}= \#\{{\rm pairs}\ \{o,o_{s}\} | o=o_{sm}\ne o_s\}$$
which vanishes for $n$ even,   according to Prop.\ \ref{prop2} (ii),
and that
$$\card{{\rm OUc}}- \card{{\rm OOb}}= z_{sm} = \#\{{\rm pairs}\ \{o,o_s\} 
| o=o_m\ne o_s\}$$
which  vanishes for $n$ odd, according to Prop.\ \ref{prop2} (iii). 
\\
Finally (\ref{uuc}) may be derived from the same analysis for the sets
OUc, OUb, UUc and UUb:
one finds that $$\card{{\rm OUb}}  - \card{{\rm UUc}}=\#\{o\in {\rm OUc} | o=o_r\}$$
which vanishes for $n$ even; for the second relation (\ref{uuc}) one may appeal to (\ref{uocetal}) together with 
$x_{sm}=y_{sm}=z_{sm}=0$ for $n$ odd.  \end{proof}
Remark. Recall that for genus $0$, $\card{\rm OOb}=\card{\rm OO}$ and thus, from Theorem 4, we have $\card{\rm UOc}=\card{\rm OO}$ if $n$ is even, and  $\card{\rm UOc}=2 \card{\rm UO}$ if $n$ is odd. \\[5pt]

Note that as a by-product of this discussion, we have obtained now the number of
spherical ($g=0$) immersions of types OO and  OU, 
 \be\label{immersionsOO}\!\!\!\!\!\!\!\!\!
 \begin{split}
&\mathrm{Oriented\ }S^1 \mathrm{\ in\ oriented\ } S^2: \\
 \# & \mathrm{OO\ immersions\ }  = 1,3,9,37,182,1143,7553, {54\, 559}, 412\,306,  {\red{3\,251\,240}},\cdots
   \end{split}
   \ee
   \be\label{immersionsOU} \!\!\!\!\!\!\!\!\!
    \begin{split}
&\mathrm{Oriented\ }S^1 \mathrm{\ in\ unoriented\ } S^2: \\
 \# & \mathrm{OU\ immersions\ }=1,2,6,21,97, 579, 3812, 27328, 206\,410,  {\red{1\,625\,916}},  \cdots
    \end{split} \ee
thus extending the OEIS sequences A008986,  A008988\, \cite{OEIS} and Valette's recent results \cite{Valette}. 

%%%%%%%%%%%%%%%%%%%%%%%%%%%%%%%%%%%%%%%%%%%%%%%%%%%%%%%%%%%%%%%%
   %                                                          Section 5
%%%%%%%%%%%%%%%%%%%%%%%%%%%%%%%%%%%%%%%%%%%%%%%%%%%%%%%%%%%%%%%%

\section{Immersions of types OO,  UO, UO and UU from cyclic permutations of $S_{2n}$}
\label{Zpmethod}

\subsection{The subset $Z^\prime =[2n]$ of $ Z =S_{2n}$ and its orbits  for the adjoint action of a particular $S_n$ subgroup
(``Z method")}

\label{Z}
In the present section, we shall count the number of orbits in 
a particular conjugacy class of $Z=S_{2n}$, namely the set $Z^{\prime}$ of its cyclic permutations, 
under the action of a particular subgroup $C'_\rho$ isomorphic with $S_n$.
Our goal is to determine the numbers of immersions of  a circle in a Riemann surface  of given genus, {\bf irrespective of the bicolourability condition} that we introduced in the previous section. To achieve this,   we first consider an oriented circle and   make use  of another labelling of maps by permutations of $S_{2n}$.

\paragraph{The ``Z method".} 
Consider a map, the edges of which are oriented in a consistent way for our purpose,
namely with incoming edges at each vertex next to one another, see Fig.\ \ref{versionv3}.
Let us  label the  edges of such a map by an index $i$ running from 1 to $2n$. At each vertex, there is an involution $\rho\in [2^n]
\subset S_{2n}$
which exchanges the labels  of the two incoming edges, and a permutation $\pi$ that yields the labels of the outgoing edges,
see Fig.\ \ref{versionv3}. The condition that the map has a single component amounts to saying that 
$\pi$ has a single cycle, $\pi\in [2n]$. As before, we can fix $\rho$, for example to be equal to
$\rho_0=(1,2)(3,4)\cdots(2n-1,2n)$, a product of $n$ disjoint transpositions. 
With that convention, the integers $(2a-1,2a)$, $a=1,\cdots n$, label the $a$-th pair of incoming edges, ordered, say, in a clockwise way.
Then  the number of topologically  inequivalent oriented maps equals the number of orbits of $Z'=   [2n]$ under the conjugate action of a subgroup of $S_{2n}$,  
made of permutations that map odd (resp. even) labels onto odd (even) labels and commute with $\rho_0$.
As it consists of permutations of the $n$ pairs $(2a-1,2a)$, $a=1,2,\cdots,n$, it is isomorphic with $S_n$ and has order $n!$.
We shall usually denote it by  $\CC'_\rho$.

\begin{figure}[htbp]
 \centering
\includegraphics[width=20pc]{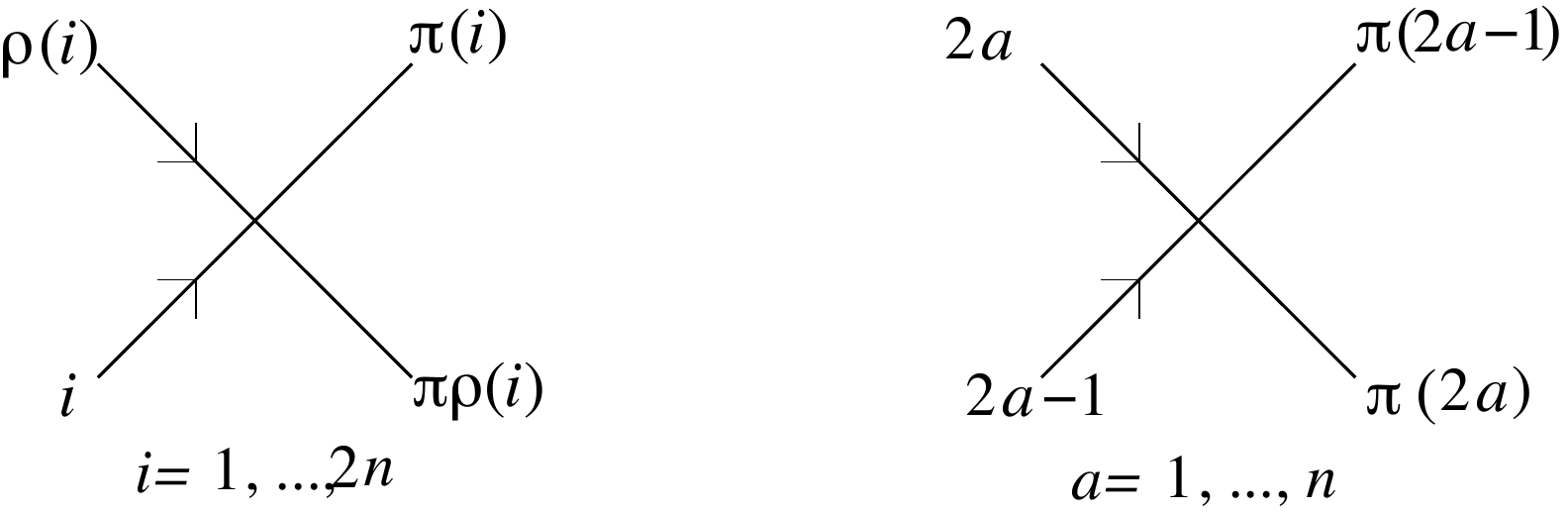}\\[10pt]  
\includegraphics[width=10pc ]{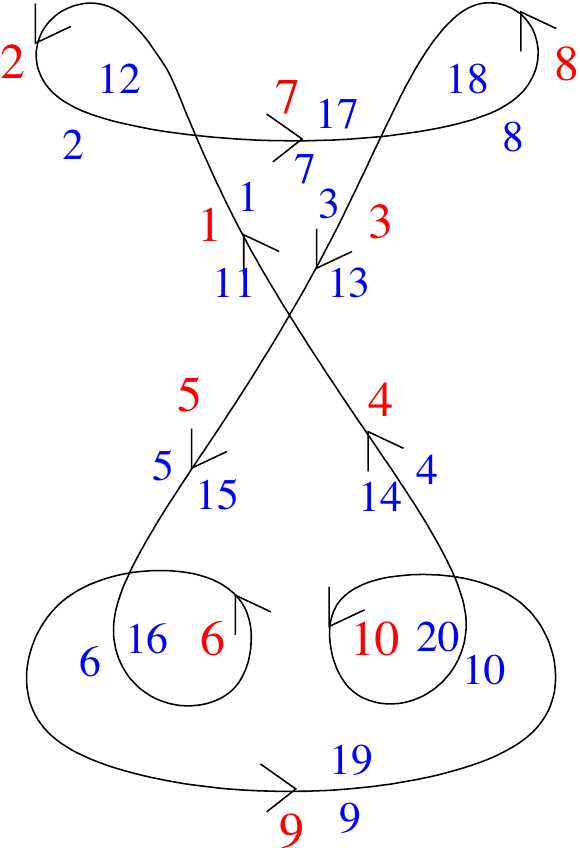}
\caption{\label{versionv3} Above, left : labelling oriented edges; right: special choice of $\rho$. 
Below, defining the new labelling of edges: in red the original labelling from $1$ to $2n$, in blue 
the new one from $1$ to $4n$.} 
\end{figure}

{In order to study the genus of the corresponding map, we now associate with the permutation $\pi\in S_{2n}$ 
another permutation $\psi_\pi\in S_{4n}$. The idea is to duplicate the edge labels so as to label 
separately the left and the right sides of each edge (or in the fat graph picture\,\cite{tH}, to label independently 
each of the double lines):  
we choose to label the right side of the oriented edge originally labelled 
by $i\in\{1,2n\}$ by $i$, and its left side by $i+2n$, see Fig.\ \ref{versionv3} bottom  for an
illustration. The permutation $\psi_\pi$ then describes 
the succession of these labels as each face is travelled clockwise. }
The transformation $\pi\mapsto \psi_\pi$ (not a group homomorphism) is easily implemented,  for $1\le i\le 2n$,
\bea\nonumber 
\quad \psi_\pi(i) &=&\begin{cases}
 \pi(i+1) & \mathrm{if}\ i\  \mathrm{is\ odd}, \\
i-1+2n & \mathrm{if}\ i\  \mathrm{is\ even} \end{cases}
\\
\quad  \psi_\pi(i+2n)&=&\begin{cases} \pi^{-1}(i)+1+2n& \mathrm{if}\ \pi^{-1}(i)\  \mathrm{is\ odd}\\  \pi(\pi^{-1}(i)-1) & \mathrm{if}\ \pi^{-1}(i)\  \mathrm{is\ even} \end{cases}
\eea

%%%%%%%
\paragraph{Example of encoding.}
As an example, the bottom diagram of Fig.\ \ref{versionv3} is encoded by permutations
\bea\nonumber
\pi&=&[2,7,5,1,6,9,8,3,10,4]  =  (1,2,7,8,3,5,6,9,10,4) \\
\psi_\pi &=& [7,11,1,13,9,15,3,17,4,19,5,12,8,10,14,16,2,18,6,20]\\ \nonumber
 &=&(1,7,3)(2,11,5,9,4,13,8,17)(6,15,14,10,19)(12)(16)(18)(20)\,.
\eea
The genus of the map is then given  by the Euler characteristics, 
\be c(\psi_\pi)  =n + 2 -2g\,.\ee
Filtering the set $Z'=[2n]$, resp. its orbits under the action of $\CC^\prime_\rho$, with that criterion yields the sets $Z^{\prime}_g$, resp. their orbits; the number of
orbits of $Z^{\prime\prime}:= Z'_0$  
is the  number of immersions of an oriented circle in the oriented sphere.

\begin{mytheo}
\label{theo3}   
  Call $\rho = (1, 2)(3, 4)\ldots (2n-3, 2n-2)(2n-1, 2n) \in [2^{n}] \subset  S_{2n}$, using cycle notation, and  $\CC'_\rho$, the subgroup of $Z = S_{2n}$, isomorphic with $S_n$, that commutes with $\rho$ and
   permutes odd resp. even integers among themselves.
Circle immersions of the oriented circle in an 
{oriented} surface of genus $g$, or OO immersions for short, are in bijection with the orbits of $\CC'_\rho$ acting by conjugacy on the set of permutations $\pi$ that belong to $Z^\prime = [2n]$, the subset of cyclic permutations of $Z$, and such that the associated permutation 
$\psi_\pi\in S_{4n}$, defined previously,  satisfies the condition
\be  c(\psi_\pi) = n+2-2g\,,  \ee 
$c(x)$ being the function that gives the number of cycles (including singletons) of the permutation $x$. 
 \end{mytheo} 

Remarks. \\
The group $\CC'_\rho$, isomorphic with $S_n$, is contained in the centralizer $\CC_\rho$ of $\rho$ in $S_{2n}$.  It is generated by the pairs of transpositions $(1,3)(2,4)$, $(1,5)(2,6)$, $(1,7)(2,8)$, $\ldots$, $(1,2n-1) (2, 2n)$.
Notice that $C(S_{2n}, \CC'_\rho) = \{1, \rho \}$. 
 One can see that $\CC'_\rho$ is precisely the subgroup $S_{n}$ of $S_{2n}$ that allows one to build the subgroup $\CC_\rho$ used in the previous two sections as a wreath product (see the comments in Sect.\ \ref{pairS2nCrho}).\\
As already mentioned in the introduction, the 4-valent maps that we consider also define cellular embeddings of particular graphs called ``simple assembly graphs without endpoints'' in \cite{BDJSV}.
In this reference the authors introduce  the notion of genus range of a given graph (the set of all possible genera of surfaces in which the graph can be embedded cellularly), a notion that is also studied and generalized    in \cite{Arredondo}\footnote{We thank an anonymous referee for pointing out these two references.}.
Their work uses the same ribbon (or fat) graph construction as ours, a construction that was described  in a quantum field theory context\,\cite{tH}, and in \cite{Carter} as a tool for classification of immersed curves;   their encoding of maps use chord diagrams and Gauss codes. In contrast, the methods presented in this section do not use chord diagrams but introduce a way to encode the relevant graphs (and their fat partners) in terms of  permutations and relate systems of representatives for different 
types of immersions to double cosets of appropriate finite groups (see below).

\paragraph{From cyclic permutations on $2n$ elements to simple closed curves with $n$ crossings in Riemann surfaces.}
We described how to associate a cyclic permutation to the image of an immersion in an oriented Riemann surface of genus $g$, more precisely to a closed curve, drawn in the plane, with $n$ regular crossings, and some number of virtual crossings.
Conversely, associating a closed simple curve with a given cyclic permutation $\pi$ belonging to $S_{2n}$  
is straightforward.
One draws $n$ four-valent vertices : four half-edges at each vertex, two ingoing, two outgoing, obeying the usual transversality (crossing) condition.
If $j$ is odd, $\pi(j)$ labels an in-going half-edge, $\pi(j)+1$ labels the in-going half-edge at the same vertex and located immediately to the right of the previous half-edge (using a clockwise orientation), $\pi(j+1)$ labels the outgoing half-edge corresponding to $\pi(j)$, and $\pi(\pi(j)+1)$ the outgoing half-edge corresponding to $\pi(j)+1$.
One starts with $j=1$ and obtains in this way the four half-edges associated with some vertex. One then considers, in turns,  $j=3$, $j=5$, etc, and the construction terminates since there is a finite number $n$ of vertices. The last operation is to connect the half-edges carrying the same labels. 
Of course, the obtained closed curve, drawn on a plane, will have usually more that $n$ crossings, but only $n$ of them -- those defined by the permutation $\pi$ -- should be consider as regular crossings (the others being virtual). 
The genus is determined by considering the associated fat graph, \ie the permutation $\psi(\pi)$, and using the Euler formula. 
The fact that the obtained curve indeed corresponds to the image of an immersion is taken care of by the necessity of choosing the possibly non-zero genus determined by $\psi(\pi)$.
%%%%%%%%%%

\subsection{A partition of $Z^\prime$}
In order to get representatives, for each genus, of the orbits of  $Z^\prime$, one may first determine the orbits, and then filter them according to the genus, this is what we shall actually do.
However one can also start by partitioning the set $Z^\prime$ according to the genus (sets $Z^\prime_g$) and determine the orbits, for each $g$, in a second step. 
This latter method is, in practice, slower than the first. 
It produces as a by-product, and for each positive integer $n$, a family of numbers $\vert Z^\prime_g \vert$ that add up to $(2n-1)!$ since this is the cardinality of $Z^\prime = [2n]$. 
The same numbers could also be obtained with the first method, proceeding backwards, by using the orbit-stabilizer theorem for each orbit of $Z^\prime_g$. 
These integers are gathered, for the first values of $n$, in Table \ref{cardZprimeg}. 
Notice that each member of the above partition can itself be decomposed into strata corresponding to different sizes of the orbits: for instance, taking $n=5$, one gets $21552$ orbits corresponding to 
the union of $21480$ and $72$ orbits with respective centralizers of order $1$ and $5$. We shall not display that information.
 \begin{figure}[htbp]
\begin{gather*}
  1 \\
  4 \qquad 2 \\
  42 \qquad 66 \qquad 12 \\ 
  780 \qquad 2652 \qquad 1608 \\ 
  21552\qquad 132240\qquad 183168\qquad 25920 \\
  803760 \qquad 7984320 \qquad 20815440 \qquad 10313280 \\
  \ldots 
\end{gather*}
\caption{\label{cardZprimeg} Numbers of elements in the sets $Z^\prime_g$. Each line adds up to an odd factorial.}
\end{figure}
 
\subsection{Counting orbits and their lengths}    
\label{dblecosZ}

In order to obtain the genus decomposition for the various kinds of immersions we are interested in (OO, UO, OU and UU types), one has to 
use explicit cyclic permutations for these different kinds of immersions, together  with the method (filtering by genus) previously described in Sect.\ \ref{Z} for OO immersions.
We shall see later, in \ref{DiscreteTransfoZ}, how the introduction of discrete transformations (orientation reversal and mirror symmetry) on OO orbits allows us to refine the method and obtain the numbers of immersions for the types OU, UO and UU.
The appearance of non-trivial stabilizers complicates the counting of orbits:  in practice, one possibility is to select a random permutation $\sigma$ from the given set ($Z^\prime$), determine its conjugates, see whether or not one of them has already been selected, 
take the decision about keeping $\sigma$, or not,  throw away its conjugates, and start again (a similar method amounts to implementing  an appropriate variant of the Jeu de Taquin).  Initially, this is what the authors did.
However, the problem of selecting one representative for each orbit is simply related to a similar problem for double cosets,
 and one can take advantage of the fact that several computer algebra programs (in particular Magma, see our comments in Appendix A),  provide fast algorithms to determine such representatives. 
This other approach allowed the authors to recover and extend their previous results.
The relation between the two problems is summarized in the following proposition: 

\begin{myprop}\label{doublecosetmeth} 
For each type OO, UO, OU, or UU, of immersions of the circle, a system of representatives for the orbits of the subgroup $\CC^\prime_\rho \simeq S_{n}$ acting by conjugation on the set of cyclic permutations $[2n]$, is given by the elements of the set\footnote{There are two possible conventions for the product of two permutations. In this paper we use the right-to-left product, which is {\sl not} the convention used in \cite{Mathematica} or \cite{Magma}. Also, $\beta^y = y \, \beta \, y^{-1}$. With the other convention, one should replace $1/x=x^{-1}$ by $x$ in the next formula, or replace the pair $(H,K)$ by $(K,H)$.} $\{ \; \beta^{1/x} \; \}$, with $x \in H\backslash G / K$, where $\beta = (1, 2, 3, \ldots, 2n)$, where the subgroups $H$ and $K$ of $G=S_{2n}$ are as indicated below, and where it is understood that we choose one representative element  $x$ (a permutation) in each double coset. 
\bei
\item For OO,  one takes $H=Z_\beta$, the centralizer of the cyclic permutation $\beta$ in $G$, and $K=\CC^\prime_\rho$. %\\
\item For UO, one takes   $H= < \!Z_\beta, \sigma_r\!>$,  the subgroup of $S_{2n}$ generated by $Z_\beta$ and the permutation $(2, 2n)(3, 2n-1)(4, 2n-2)\ldots (n, n+2)$ that conjugates $\beta$ and $\beta^{-1}$ in $S_{2n}$ and implements orientation reversal of the source (circle),  and $K=\CC^\prime_\rho$. %\\
\item For OU, one takes   $H=Z_\beta$ and $K=< \!C^\prime_\rho, \rho\!>$,  the subgroup of $S_{2n}$ generated by $C^\prime_\rho$ and the permutation $\rho$ which describes mirroring in the target, see Sect. \ref{DiscreteTransfoZ}. %\\
\item For UU,  one takes  $H=< \!Z_\beta, \sigma_r\!>$ and $K=< \!C^\prime_\rho, \rho\!>$.
\eei
\end{myprop}

The proof of this proposition, in the OO case, relies on Theorem \ref{Coqtheo}. We already know that the orbits for the adjoint action of $K$ on the conjugacy class of $\beta$ (the cyclic permutations) are in one-to-one correspondence with the double cosets of  $H\backslash G / K$. The above proposition makes this correspondence explicit in the present situation: taking $x$ and $x^\prime$ in the same double coset, we write $x^\prime = h \, x \, k$, with $h \in H=Z_\beta$ and $k \in K =\CC^\prime_\rho$ and see immediately that $\beta^{1/x^\prime} = k^{-1} \, \beta^{1/x} \, k$, so the two permutations $\beta^{1/x^\prime}$ and $\beta^{1/x}$, which are cyclic since both conjugated of $\beta$ -- which is cyclic itself -- by an element of $G$, are also conjugated by $K$ and therefore characterize the same OO orbit.
The justification for the choice of the other subgroups, appropriate to handle the immersions of types OU, UO and UU, ultimately relies on a discussion that will be carried out in the next section (discrete transformations).

Notice that if one is interested only in counting the {\sl total} number of immersions, \ie summing over all genera for each of the types OO, OU, UO and UU, one does not need to determine double coset representatives since only the total number of double cosets matters. The latter can be computed up to large values of $n$ by using Frobenius'  formula  (\ref{frobenius}) -- we remind the reader that it uses only the knowledge of the cyclic structure and size for the usual conjugacy classes: this is both simpler and  faster. 
The results for OO, \ie the number of orbits in $Z'$  are displayed in Table \ref{TableZp}, up to $n=20$.  The corresponding results for types OU, UO, UU,  can also be determined\footnote{ They have been added to the OEIS:
OO A260296, UU  A260912, UO A260847, OU A260887.}  
by using Frobenius'  formula in virtually no time up to $n=20$, those up to $n=10$, are given in Table \ref{TableXYZ2}.

The drawback of this last method is that it is not constructive, so that one has to rely on the previous approaches (brute force determination of the orbits or use of double coset representatives) to go to the next step:  filtering according to the genus. 
Actually, this last part is the bottleneck of the process as the function $\psi_\pi$ defined in the previous section takes its values in $S_{4n}$.

    \begin{table}[ht]
\caption{Number of orbits in $Z'$ (OO case)}
\centering
    \def\t{\footnotesize}
  {\footnotesize
 \begin{tabular}{ | c | r || c | r | } 
       \hline &&& \\[-3pt]
        1&1  &11& 1\,279\,935\,820\,810 \\  
2&4   &12 & 53\,970\,628\,896\,500\\ 
3& 22 & 13 & 2\,490\,952\,020\,480\,012 \\ 
4& 218  & 14 & 124\,903\,451\,391\,713\,412 \\ 
5& 3028 & 15 &  6\,761\,440\,164\,391\,403\,896  \\ 
6& 55540  & 16 & 393\,008\,709\,559\,373\,134\,184 \\ 
7& {1\,235\,526} & 17 &  24\,412\,776\,311\,194\,951\,680\,016\\ 
8&{32\,434\,108} & 18 & 1\,613\,955\,767\,240\,361\,647\,220\,648 \\ 
9&{980\,179\,566} & 19 &113\,146\,793\,787\,569\,865\,523\,200\,018\\
 10& 33\,522\,177\,088 & 20 &   8\,384\,177\,419\,658\,944\,198\,600\,637\,096  \\[2pt] \hline 
     \end{tabular}}
     \label{TableZp} 
     \end{table}
       
       \bigskip
     \noindent For $n$ a prime integer, {we found} an explicit formula for this number of orbits.
\begin{myprop}{For $n$ a prime integer, 
\be\label{formulaZp} \#\mathrm{orbits\ in }\ Z' =n-1+\frac{(2n-1)!}{n!} \,. \ee}
\end{myprop}
\begin{proof} For any $n$, orbits of $Z^{\prime}$ have length $\card{\CC'_\rho} /d$ with $d$ a divisor of $n$, and 
 we claim there are exactly $n$ orbits of length  $\card{\CC'_\rho} /n$. 
The diagrams of {these  orbits}  have cyclic $n$-fold symmetry. 
At the possible price of introducing ``virtual crossings", (see Introduction), these diagrams may always be drawn 
in the plane in
such a way that the outmost edges 
form a convex regular $n$-gon, travelled in the clockwise or counter-clockwise way, with vertices 
numbered from 1 to $n$, and with one pair of oriented edges connecting vertices $i$ and $i+1$, and the 
other pair $i$ and $i+k \mod n$, for $k=0,1,\cdots n-2$ ($k=n-1$ yields a $n$-component
diagram). \end{proof}

\begin{figure}[htbp]
 \centering
\includegraphics[width=30pc]{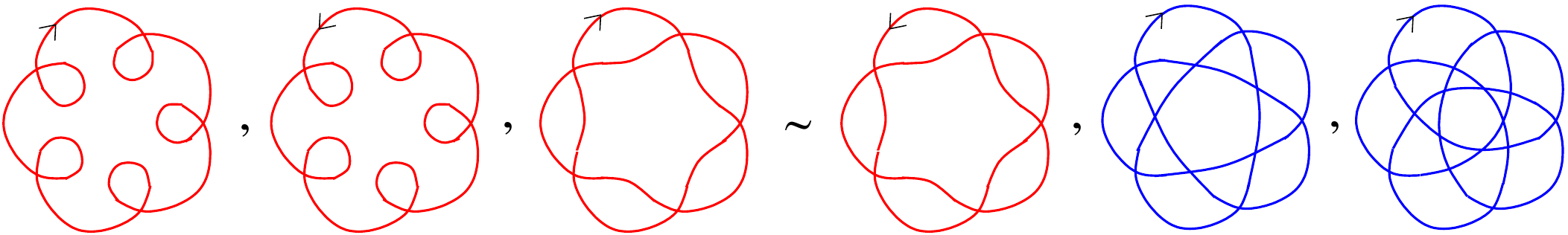}
\caption{\label{symmetry5} The $n=5$ orbits of $Z^{\prime}$ with 5-fold symmetry. Only the first three (in red) are  spherical,
the two others have higher genus (2 here); the last three are equivalent to (\ie in the same orbit as) their reversed;
in the last two, only the outmost vertices are double points, the others are ``virtual crossings" 
as explained in the Introduction. } 
\end{figure}

\begin{figure}[htbp]
 \centering
\includegraphics[width=30pc]{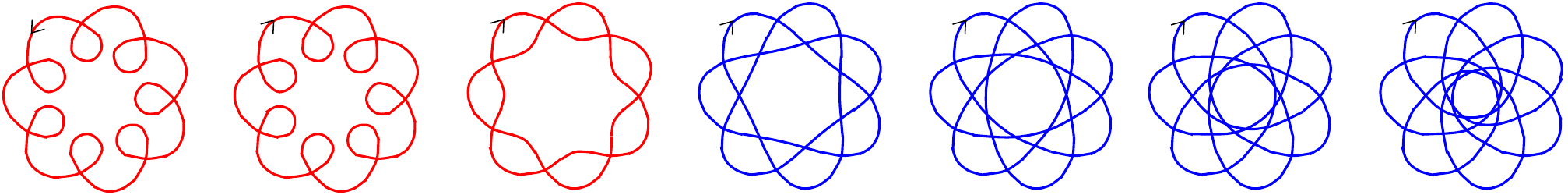}
\caption{\label{symmetry7} The $n=7$ orbits of $Z^{\prime}$ with symmetry of order 7. Only the first three (in red) are spherical,
the others have higher genus (3 here). }
\end{figure}

%%%%%%%%%%%%%%%%%%%%%%%%%%%%%%%%%

\subsection{From orbits of $Z^{\prime\prime}$ to spherical immersions of type OO}
According to Theorem \ref{theo3}, the numbers of   spherical immersions of type OO 
may be determined from the numbers of orbits of $Z''=Z'_0$ and agree with those computed above in (\ref{immersionsOO}). They also appear in  Table \ref{TableZ}.   {The numbers in red are still to be double-checked\dots}

\subsection{Discrete transformations of OO immersions. Immersions of  OU, UO and UU types}
\label{DiscreteTransfoZ}

We now consider the effect of the discrete transformations $r$ and $m$ on immersions of OO type.
\paragraph{Orientation reversal.}
Consider the effect of changing the orientation of the circle: it simply corresponds to $\pi\mapsto \pi^{-1}$.
  Orbits of $Z^{\prime}$ (for a given genus, in particular those of $Z''$) 
  split into two classes: those for which $\pi$ and $\pi^{-1}$ belong to the same orbit
  may be called {\it reversible} immersions;
   the others  form pairs of {irreversible} immersions.  
  \paragraph{Mirror symmetry.}
  If some immersion is described (for $\rho$ fixed as above) by (the orbit of) some $\pi$, its mirror image is 
associated with the orbit of $\pi'=\rho\pi\rho$. We call again achiral the immersions such that $\pi$ and
$\pi'=\rho\pi\rho$ belong to the same orbit, while the other form chiral pairs.

\paragraph{The five types of symmetries.}
$o$ being an orbit of $Z^\prime$ (of given genus), we call
%\begin{itemize}
$\bullet$ $o_r$ the orientation reversal image of $o$,   \\
$\bullet$ $o_{m}$ the chiral image of $o$,\\
$\bullet$ $o_{rm}$ the chiral image of the orientation reversal  image of $o$ (or the other way around).
%\end{itemize}

By combining the two previous transformations, we thus find five types of immersions that match Arnold's classification of symmetries
  \cite{Arnold}. Following our notations of Sect.\ \ref{preamble}, 
   we call $x_{rm},y_{rm},z_{rm},v_{rm},w_{rm}$ their numbers of  elements: 
\begin{eqnarray*}
x_{rm}&=&\# \{\text{orbits} \,\vert\,  o=o_r=o_{m}=o_{rm}  \} 
\\
{} &=& \text{%\ie the 
 number of  orbits that are both achiral and reversible.} \\
y_{rm}&=& \#  \{\text{orbit\ pairs} \, \{o,o_m\}\,\vert\,    o=o_r, o_{m}=o_{rm}\}  
\\
{} &=& \text{%\ie the 
number of chiral pairs of reversible orbits.}\\
z_{rm}&=& \#  \{\text{orbit pairs} \,\{o,o_r\}\,\vert\, o=o_{m}, o_r=o_{rm} \}  
\\
{} &=& \text{%\ie the 
 number of irreversible pairs of achiral orbits.} \\
v_{rm} &=& \#  \{\text{orbit pairs}  \; \{o, o_r\} \,  |\; o \neq o_r \; \text{but such that} \; {o=o_{rm}} 
\ {\Leftrightarrow} \ o_r=o_{m}\} . \\
w_{rm} &=&  
\# \{\text{4-plets of orbits} \; \{o, o_r, o_{m}, o_{rm} \}, \;
\\
{}&{}& \text{where all members of each 4-plet should be distinct.}
\end{eqnarray*}

The values of those five parameters are gathered in Appendix B.4,
{\normalcolor
and one obtains, for every genus:
\begin{eqnarray}\nonumber
\card{\text{OO}} &=& 
 x_{rm}+2y_{rm}+2z_{rm}+2v_{rm}+4w_{rm}\\ \nonumber
\card{\text{UO}} &=& x_{rm}+2y_{rm}+z_{rm}+v_{rm}+2w_{rm} 
\\
\card{\text{OU}} &=& 
x_{rm}+y_{rm}+2z_{rm}+v_{rm}+2w_{rm}\\ \nonumber
\card{\text{UU}} &=& x_{rm}+y_{rm}+z_{rm}+v_{rm}+w_{rm}\,.
\end{eqnarray}}

In that way, we recover the number of  immersions for $g=0$ found in Sect.\ \ref{fromorbtoimm}, which yields a non-trivial check on  both methods. 

From Theorem 4  (eq.\ (\ref{uuc})) one shows that, for $g=0$ and $n$ odd,  
$v_{sm} = x_{rm}+z_{rm}+v_{rm}$  and  $w_{sm} = y_{rm}+w_{rm}$, the $v_{sm}$ and $w_{sm}$ parameters being those defined in Sect.\ \ref{UOUUimm} -- see also Appendix B1 (not B3 !); remember also that $x_{sm} = y_{sm} = z_{sm} = 0$ in that case.  For $g=0$ and $n$ even, we notice a one-to-one equality between the same $5$-plet $(x_{sm},y_{sm},z_{sm},v_{sm},w_{sm})$ and $(x_{rm},y_{rm},z_{rm},v_{rm},w_{rm})$, but this is only an observation for which we have no explanation, and it suggests a direct correspondence between the corresponding orbit types.

%%%%%%%%%%%%%%%%%%%%%%%%%%%%%%%%%%%%%%%%%%%%%%%%%%%%%%%%
\section{Miscellaneous comments}
\subsection{Asymptotics}

The number of points in the three sets $X^\prime$, $Y^{\prime}$, $U$ or $Z^{\prime}$  are known explicitly
 and grow factorially 
\bea \nonumber \card{X'_n}  &=& (4n-2)!! 
\\
 \card{Y'_n}  &=& 2^{2n-1} n! (n-1)!  
 \\ \nonumber
  \card{U_n}  &=& 2^{n} n! 
 \\ \nonumber
  \card{Z'_n} &=& (2n-1)! 
  \eea
Then, using the classical fact that ``almost all maps are asymmetric" \cite{RW}, the asymptotic 
numbers of orbits in $X^\prime$, $Y^{\prime}$ or $Z^{\prime}$ are given by 
\bea \nonumber
\#  \CC_\sigma \mathrm{-orbits\ in\ } X'_n  &\sim& \frac{(4n-2)!!}{4^n n!} = \frac{1}{2} \frac{(2n-1)!}{n!}\sim n! \,\frac{2^{2n-1}}{2 \pi^{1/2}n^{3/2}}\\
\#  \CC_\rho \mathrm{-orbits\ \,in\ }\, Y'_n  &\sim&\# D_n\mathrm{-orbits\ in\ } U_n \sim 2^{n-1} (n-1)! \\ \nonumber
\# \mathbb{Z}_n\mathrm{-orbits\ in\ } U_n  &\sim& 2^{n} (n-1)! \\ \nonumber
\#  \CC^\prime_\rho \mathrm{-orbits\  in\ }\, Z'_n  &\sim& \frac{(2n-1)!}{n!}\sim 2 \,\#  \CC_\sigma\mathrm{-orbits\ in\ } X'_n 
\eea
Unfortunately we have no similar exact formulae for 
orbits in  $X^{\prime\prime}$, $Y^{\prime\prime}$ or $Z^{\prime\prime}$, and we have to appeal to
 empirical estimates derived by physicists in similar contexts, see for example \cite{DFGG}. 
 For each of the above quantities, one expects an exponential growth of the form
\be \label{asymp} \#_n \sim \kappa   n^{\gamma-3} a^n \ee
with $a$, $\gamma$ the ``string susceptibility" and $\kappa$ some constant,  depending on the problem at hand.\\
Here, according to Schaeffer and Zinn-Justin\cite{SZJ}, $\gamma= -\frac{1+\sqrt{13}}{6} $, corresponding to a central charge $c=-1$ in KPZ formula 
\be\label{kpz} \gamma_{\rm KPZ}(c)=\frac{c-1-\sqrt{(25-c)(1-c)}}{12} \ee
(see  for instance \cite{DFGG}, eq.\ (4.2)). 

\def\g{g}
In genus $\g$, one expects $\gamma$ in asymptotic behavior (\ref{asymp}) to be replaced by 
 \be \gamma\mapsto \gamma(\g) =(1-\g)\gamma
 \ee
 (which makes the ``double scaling limit" possible). 
Unfortunately, the order $n=10$ that we have reached is certainly much too low
to enable one to observe the onset of this asymptotic behavior. See \cite{SZJ} for a 
discussion of the logarithmic corrections to that asymptotic behavior.

\subsection{Knot census}

Applying the previous counting of maps to the census of (alternating) knots requires eliminating various types
of redundancies, kinks aka nugatory crossings, non prime and flype equivalent diagrams, following
Sundberg and Thistlethwaite's procedure \cite{MSTh}. See \cite{PZJ-website}
for a beautiful implementation including virtual knots and links.

%%%%%%%%%%%%%%%%%%%%%%%%%%%%%%%%%%%%%

%%%%%%%%%%%%%%%%%%%%%%%%%%%%%%%%%

\section{Results and conclusion}
Our results on the numbers of curves of different types and different genera are gathered in 
Tables \ref{TableXYZ1} (bi{\colour}able and/or bi{\colour}ed immersions) and \ref{TableXYZ2} (general immersions).
Subtables of Table \ref{TableXYZ1} can all be obtained from the $U$ method, and also from the $Y$ method for cases UOc, UOb, UUc and UUb. 
   Subtables of genus 0 are also obtained from the $X$ method.
 Subtables of Table \ref{TableXYZ2} are obtained from the $Z$ method. 
 Subtable UO is also obtained from the $X$ method.
       The reader will verify the various identities stated in Theorem \ref{Theo3c} between numbers of
different types of immersions.

As the case of immersions of a circle in the sphere is of particular interest, we
summarize their numbers in the following tables. Recall that in that case (genus 0), there
is no distinction between bi{\colour}able and general immersions.
As the vast majority of immersion diagrams contain a ``simple loop" (aka ``kink", see for example all the diagrams in 
Fig.\ \ref{imm2}),  it is 
suggested to discard them and to count only diagrams without such simple loops. In a next step, one may 
concentrate on diagrams that are {\it irreducible} and {\it indecomposable}, \ie not made disconnected by
removal of a vertex, resp. by cutting two distinct lines\footnote{In the knot theory terminology, diagrams with a
simple loop or reducible are referred as having a ``nugatory" crossing, and indecomposable ones as ``prime".}.

In the  formalism of Sect.\ \ref{Ypmethod}, no simple loop means that neither
$\sigma$ nor $\tau=\sigma\rho$ has a cycle of length~$1$.  Imposing indecomposability and irreducibility requires a more detailed analysis, incorporated in a Mathematica code\footnote{ 
The numbers for OU and UU irreducible and indecomposable immersions  (Table \ref{Table4b}) have  already appeared in the literature, actually up to $n=11$ crossings,  see OEIS sequences A089752 and A007756 and \cite{Kirkman}. After completion of the first version of the present work, we learnt from B\'etr\'ema  that he had been able to compute 
the numbers of OO, OU and UU irreducible and indecomposable immersions  up to $n=14$ \cite{Betrema}.
}.
%  \vglue-8mm
% \hglue-15mm    
    \begin{table}[ht]
    \caption{Counting of spherical immersions}
        \centering
 \begin{tabular}{ | l || c|c|c|c|c|c|c|c|c|c|}
       \hline
         \hfill$n$ & 1& 2 & 3 & 4 & 5 & 6&7&8&9&10\\
         \hline
         OO %immersions 
      		&  1&3& 9 & 37 & 182 &1143 & {7553} &{54\,559} & 412\,306 &\Red{3\,251\,240} \\ \hline
         UO %immersions 
         	 &1&2&6&21&99&588&3829&27\,404& 206\,543& \Red{1\,626\,638} \\ \hline
	 OU %immersions 
       		& 1 & 2 & 6 & 21 & {97} & {579} & {3812} &{27\,328}&206\,410&\Red{1\,625\,916} \\ \hline
        UU %immersions 
        		&1 & 2 & 6 & 19 & 76 & 376 & 2194 &14\,614 & 106\,421  &\Red{823\,832} \\ \hline
       	UOc %immersions 
	&2 & 3 & 12 & 37 & 198 & 1143 & 7658 & 54\,559 & 413\,086 & 3\,251\,240 
		 \\ \hline
          \end{tabular}
          \label{Table4}
                \end{table}                   
      %\hglue-15mm
 %     \vglue-12mm       
    \begin{table}[h]
    \caption{Counting of spherical immersions with no simple loop}
     \centering
 \begin{tabular}{ | l || c|c|c|c|c|c|c|c|c|c|}
       \hline
         \hfill$n$ & 1& 2 & 3 & 4 & 5 & 6&7&8&9&10\\   
                  \hline
         OO immersions 
          	&0  &0& 1& 1 & 2 & 9 &  29&{133} & 594  & \red{2864}  \\ \hline
	UO immersions 
 		& 0&0&1&1&2&6&19 &74&320& \red{1469} \\ \hline	
	OU immersions 
         	& 0 & 0 & 1 & 1 & {2} & {5} & {18} &{70}& 313 &\red{1440}   \\ \hline
         UU immersions 
         	&0&0&1&1 &2&5&16&52& 205 & \red{863} \\ \hline
	bi{\colour}ed UO immersions & 0 & 0 & 2 & 1 & 4 & 9 & 38 & 133 & 640 & \red{2864}
        \\ \hline
          \end{tabular}
          \label{Table4a}
              \end{table}
                   
 % \hglue-15mm    
 
   \begin{table}[h]
   \caption{Counting of irreducible indecomposable spherical immersions.}
       \centering
 \begin{tabular}{ | l || c|c|c|c|c|c|c|c|c|c|}
       \hline
         \hfill$n$ & 1& 2 & 3 & 4 & 5 & 6&7&8&9&10\\
                  \hline
         OO immersions 
     		 & 0 &0& 1 &1  & 2 &  6& 17 &{73} & 290 &{1274}   \\ \hline
         UO immersions 
         	&0&0&1&1&2&4&12 &41&161&\red{658}  \\ \hline
	OU immersions 
     		 & 0 & 0 &  1& 1 & {2} & {3} & {11} &{38}&156&{638}  \\ \hline
         UU immersions 
       		& 0&0  & 1 &1 & 2& 3 &10 &27& 101  &{364} \\ \hline
 bi{\colour}ed UO immersions & 0 & 0 & 2 & 1 & 4 & 6 & 24 & 73 & 322 & {1274}
        \\ \hline
          \end{tabular}
          \label{Table4b}
            \end{table}
            
            %%%%%%%%%%%%%%%%%%%%%%%%%%%%%% THE TWO BIG TABLES 
%\newpage
     \begin{table}[H]
   \caption{Counting of {\bf general} immersions of a circle in a surface of arbitrary genus $\g$, up to stable geotopy.  
   U = Unoriented, O=Oriented.
   \ommit{Subtables OO, UO, OU and UU are obtained from the $Z$ method.
   Subtable UO is also obtained from the $X$ method.}   {\ Figures in red should be confirmed.}
   }
   \centering
%\hglue-12mm   
{ \small
  \begin{tabular}{ | l || c|c|c|c|c|c|c|c|c|c|}
       \hline
         \hfill$n$ & 1& 2 & 3 & 4 & 5 & 6&7&8&9&10\\
         \hline      
                               \Blue{ OO, total}  & 1 & 4 & 22 & 218 & 3028 & 55540  & {1\,235\,526} &{32\,434\,108} &980\,179\,566& 33\, 522\, 177\,088 \\ \hline
\rowcolor{Apricot} OO, $\g=0$ &  1&{3}& 9 & {37} & 182 &{1143} & {7553} &{54\,559} &  412\,306 &\Red{3\,251\,240 }  \\ \hline
             OO, $\g=1$ &0&1&11&113&1102&11\,114&112\,846 &1\,160\,532&12\,038\,974& \\ \hline
              OO, $\g=2$ &0&0&2&68&1528&28\,947& 491\,767 &7\,798\,139&117\,668\,914&  \\ \hline
               OO, $\g=3$  &0&0&0&0&216&14\,336&554\,096 &16\,354\,210&407\,921\,820& \\ \hline
               OO, $\g=4$   &0&0&0&0&0&0&69\,264&7\,066\,668&397\,094\,352&  \\ \hline
                 OO, $\g=5$   &0&0&0&0&0&0&0&0& 45\,043\,200&  \\ \hline\hline
                            \Blue{UO, total }  &1& 3& 13& 121& 1538& 28\,010& 618\,243  &16\,223\,774&490\,103\,223& 16\,761\,330\,464  \\ \hline
 \rowcolor{Apricot}    UO, $\g=0$   &1&2&6&21&99&588&3829&27\,404& 206\,543& \Red{1\,626\,638}  \\ \hline
             UO, $\g=1$  &0&1&{6}&64 &{559 }&{5656}&56\,528 &581\,511&6\,020\,787& \\ \hline
              UO, $\g=2$  &0&0&1& 36&772&14\,544&246\,092 &3\,900\,698&58\,838\,383&  \\ \hline
               UO, $\g=3$   &0&0&0&0&108& 7222& 277\,114&8\,180\,123&203\,964\,446& \\ \hline
               UO, $\g=4$   &0&0&0&0&0&0& 34\,680 &3\,534\,038&198\,551\,464&  \\ \hline
               UO, $\g=5$   &0&0&0&0&0&0&0&0&22\,521\,600&  \\ \hline\hline
                               \Blue{OU, total }  &1&3&14&120&1556&27\,974&618\,824 &16\,223\,180  &490\,127\,050& 16\,761\,331\,644  \\ \hline
\rowcolor{Apricot}  OU, $\g=0$ & 1 &   {2} & 6 &  {21} & {97} &  {579} & {3812} & {27\,328}&206\,410&  \Red{1\,625\,916}\\ \hline
               OU, $\g=1$ & 0& 1& 6& 62& 559& 5614 &56\,526&580\,860&6\,020\,736& \\ \hline
              OU, $\g=2$ &0&0&2&37&788&14\,558&246\,331 &3\,900\,740&58\,842\,028&  \\ \hline
                OU, $\g=3$  &0&0&0&0&112& 7223 &277\,407&  8\,179\,658& 203\,974\,134& \\ \hline
                OU $\g=4$   &0&0&0&0&0& 0&34\,748 &3\,534\,594&198\,559\,566& \\ \hline
                OU, $\g=5$   &0&0&0&0&0&0&0&0&22\,524\,176&  \\ \hline\hline
                               \Blue{UU, total } &1& 3& 12& 86& 894& 14\,715 & 313\,364& 8\,139\,398&245\,237\,925 & 8\,382\,002\,270 \\ \hline
  \rowcolor{Apricot} UU, $\g=0$  &1 & 2 & 6 & 19 & 76 & 376 & 2194 &14\,614 & 106\,421  &\Red{823\,832}   \\ \hline
            UU, $\g=1$  &0 & 1& 5& 45& 335& 3101& 29\,415&295\,859&3\,031\,458& \\ \hline
             UU, $\g=2$ &0&0&1& 22& 427& 7557& 124\,919&1\,961\,246&29\,479\,410&  \\ \hline
              UU, $\g=3$  &0&0&0&0&56& 3681& 139\,438&4\,098\,975&102\,054\,037&  \\ \hline
              UU, $\g=4$   &0&0&0&0&0&0&17\,398 &1\,768\,704&99\,304\,511&  \\ \hline
               UU, $\g=5$   &0&0&0&0&0&0&0&0&11\,262\,088&  \\ \hline\hline\hline
                \end{tabular}}
                \label{TableXYZ2}
      \end{table}
         
 %%%%%%%%%%%%%%%%%%%%%%%%%%%%%%%%%%%%%%%%%%%%%%%%%%%%%%%%%%%%%

   \begin{table}[H]
%     \vglue-4mm 
   \caption{{\small Counting of {\bf bi{\colour}able} immersions of a circle of arbitrary genus $\g$,  up to stable geotopy.  
    \\ U = Unoriented, O=Oriented, Oc=Oriented bi{\colour}ed, Ob=Oriented bi{\colour}able, etc.\\ Figures in red should be confirmed.
       \ommit{Subtables are all obtained from the $U$ method, and also from the $Y$ method for cases UOc, UOb, UUc and UUb. 
   Subtables of genus 0 are also obtained from the $X$ method.}
        }}
   \centering
 % \hglue-9mm 
{ \small
  \begin{tabular}{ | l || c|c|c|c|c|c|c|c|c|c|}
       \hline
         \hfill$n$ & 1& 2 & 3 & 4 & 5 & 6&7&8&9&10\\
         \hline
          \Blue{OOc  total}&2&6&20&108&776&7772&92\,172 &1\,291\,048&20\,644\,140&  \\ \hline 
            OOc $\g=0$ &2&{6}&18& 74&364 & 2286  &15\,106& {109\,118} &824\,612 & \Red{6\,502\,480} \\ \hline
            OOc $\g=1$&0&0&2&32&340&3780&40\,612 &436\,368 & 4\,675\,012 & \\ \hline   
              OOc $\g=2$&0&0&0&2&72&1630&31\,510&549\,334&8\,883\,620&  \\ \hline 
               OOc $\g=3$ &0&0&0&0&0&76& 4944 & 188\,356&5\,508\,120& \\ \hline  
               OOc $\g=4$ &0&0&0&0&0&0&0 &7872&752\,776 &  \\ \hline  
                \hline                              
         \Blue{OOb  total}&1 &3&10 &54&388&3886& 46\,086 &645\,524 & {\ 10\,322\,070} & \\ \hline 
\rowcolor{Apricot}      OOb $\g=0$ &1&{3}& 9&{37}&182 & {1143 } & 7553  & {54\,559} &412\,306  & \Red{3\,251\,240} \\ \hline
      OOb $\g=1$&0&0& 1&16& 170&1890& 20\,306 &218\,184 &2\,337\,506  & \\ \hline  
       OOb $\g=2$&0&0&0&1& 36&815&15\,755 &274\,667& {\ 4\,441\,810} &  \\ \hline 
        OOb $\g=3$ &0&0&0&0&0&38&2472  & 94\,178&2\,754\,060 & \\ \hline  
        OOb $\g=4$ &0&0&0&0&0&0&0 &3936&376\,388 &  \\ \hline  
                \hline         
         \Blue{UOc  total}&2&3&14&54&420&3886&46\,470 &645\,524&10\,328\,214&  \\ \hline 
     UOc $\g=0$ &2&{3}&12&{37}&198 & {1143 } &{7658} & {54\,559} & {413\,086} & \Red{3\,251\,240} \\ \hline
             UOc $\g=1$&0&0&2&16&186&1890&20\,516 &218\,184 & 2\,340\,106 & \\ \hline   
              UOc $\g=2$&0&0&0&1&36&815&15\,812 &274\,667&4\,443\,518&  \\ \hline 
               UOc $\g=3$ &0&0&0&0&0&38&2484 & 94\,178&2\,754\,988& \\ \hline  
               UOc $\g=4$ &0&0&0&0&0&0&0 &3936&376\,516&  \\ \hline  
                \hline 
                  \Blue{UOb, total }  &1&2&7&30&210&1973& 23\,235 &323\,182&5\,164\,107& \\ \hline
\rowcolor{Apricot}    UOb, $\g=0$   &1&2&6&21&99&588&3829&27\,404& 206\,543& \Red{1\,626\,638}  \\ \hline
           UOb, $\g=1$  &0&0&{1}&{8}&{93}&{945}&10\,258 &109\,092&1\,170\,053& \\ \hline
            UOb, $\g=2$  &0&0&0&1&18&421&7906 &137\,585&2\,221\,759&  \\ \hline
            UOb, $\g=3$   &0&0&0&0&0&19&1242 &47\,089&1\,377\,494&  \\ \hline
             UOb, $\g=4$   &0&0&0&0&0&0&0 &2012&188\,258&  \\ \hline\hline
              \Blue{OUc  total}&1&4&10&60&388&3920&46\,086 &645\,928&  {\ 10\,322\,070} &  \\ \hline 
              OUc , $g=0$ &1 &4&9&42&182&1158& 7553 &54\,656 & {412\,306}&\Red{3\,251\,832}\\ \hline
              OUc , $g=1$ & 0&0&1&16&170&1890&20\,306&218\,184&2\,337\,506&\\ \hline
              OUc , $g=2$ & 0&0&0&2&36&834&15\,755&274\,922&  {\ 4\,441\,810} &\\ \hline
              OUc , $g=3$ & 0&0&0&0&0 &38&2472&94\,178&2\,754\,060&\\ \hline
              OUc , $g=4$ & 0&0&0&0&0&0&0&3988&376\,388&\\ \hline
              \hline
              \Blue{OUb  total}&1&2&7&30&210&1960&23\,276 &322\,964&{\ 5\,165\,732}&  \\ \hline 
\rowcolor{Apricot}  OUb , $g=0$ & 1&2&6&21&97&579& 3812& 27\,328&206\,410&\Red{1\,625\,916}\\ \hline
              OUb , $g=1$ & 0&0&1&8&93&945&10\,256&109\,092&1\,170\,002&\\ \hline
              OUb , $g=2$ & 0&0&0&1&20&417&7948&137\,461&{\ 2\,222\,562}&\\ \hline
              OUb , $g=3$ & 0&0&0&0&0&19&1260&47\,089&1\,378\,256&\\ \hline
              OUb , $g=4$ &0 &0&0&0&0&0&0& 1994& 188\,502&\\ \hline
              \hline
                             \Blue{UUc, total } &1&2&7&30&210&1960&23\,235&322\,964&5\,164\,107&\\ \hline
          UUc, $\g=0$  &1 & {2} & 6 & {21} &  99 & {579} & 3829  & {27\,328} & 206\,543& \Red{1\,625\,916} \\ \hline
          UUc, $\g=1$ &0& 0& 1& 8&93 & 945&  10\,258 &109\,092& 1\,170\,053 & \\ \hline
            UUc, $\g=2$ &0& 0& 0& 1&18 & 417& 7906&137\,461&2\,221\,759& \\ \hline
             UUc, $\g=3$  &0&0&0&0&0&19& 1242&47\,089&1\,377\,494 &  \\ \hline
             UUc, $\g=4$   &0&0&0&0&0&0&0 & 1994&188\,258&  \\ \hline
                          \hline
              \Blue{UUb, total } &1&2&7&26&152&1168&12\,548&165\,742& 2\,605\,526&\\ \hline
\rowcolor{Apricot}     UUb, $\g=0$  &1 & 2 & 6 & 19 & 76 & 376 & 2194 &14\,614 & 106\,421  &\Red{823\,832}   \\ \hline
          UUb, $\g=1$ &0& 0& 1& 6& 63& 539& 5508  &56\,067&592\,457& \\ \hline
            UUb, $\g=2$ &0& 0& 0& 1& 13& 242& 4183 &70\,118&1\,119\,180& \\ \hline
             UUb, $\g=3$  &0&0&0&0&0&11&663&23\,907&692\,749 &  \\ \hline
             UUb, $\g=4$   &0&0&0&0&0&0&0 &1036 &94\,719&  \\ \hline
                       \hline
              \end{tabular}}
                \label{TableXYZ1}
      \end{table}
   %%%%%%%%%%%%%%%%%%%%%%%%        

      Since our method is constructive and not only enumerative, we have not only the number of
      orbits or of immersions, but also their list, encoded in the various ways explained in the previous sections
      (the full listing up to $n=10$ is available on request). 
           This also enables us to {\it draw} images of these immersions.
      See the UU immersions for $n=8$ and $n=9$ in Fig.\ \ref{ImmUU8} and \ref{ImmUU9v3a}-\ref{ImmUU9v3b}
      respectively. These figures have been prepared {using {\tt DrawPD}, a routine to draw planar diagrams, 
      within the Mathematica package ``KnotTheory"
      written by Redelmeier \cite{EmilyR}} (the distinction between under- and over-crossings is irrelevant 
      in the current discussion).

 \begin{figure}[htbp]
 \centering
\includegraphics[width=14pc]{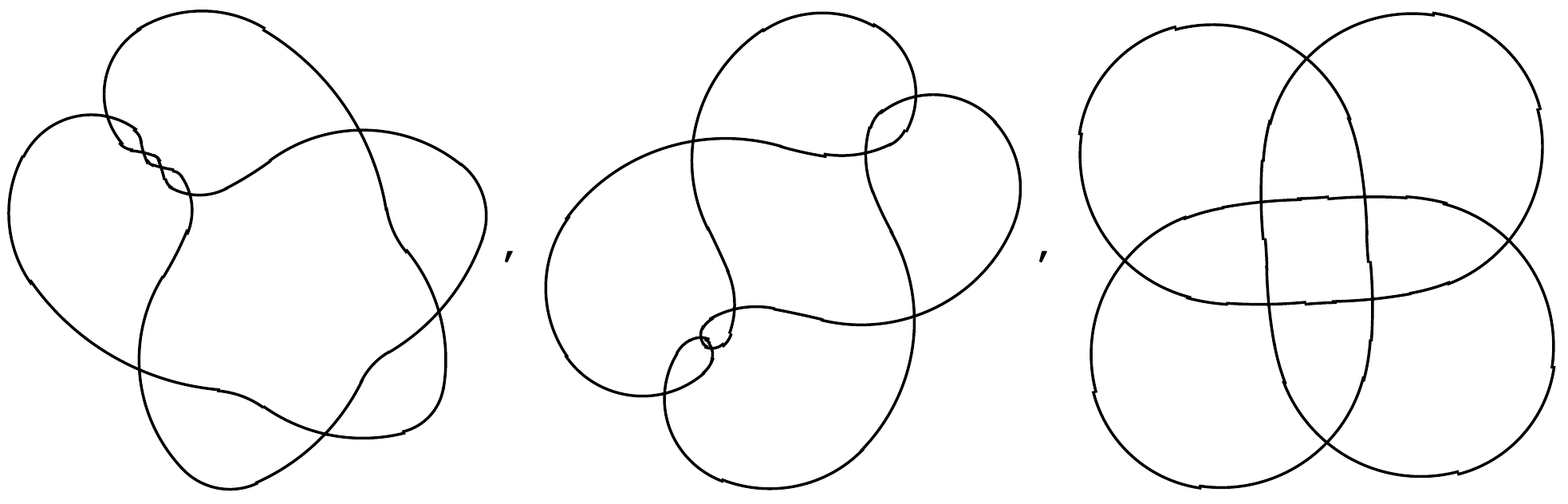}\raisebox{30pt}{;} 
\includegraphics[width=13.5pc]{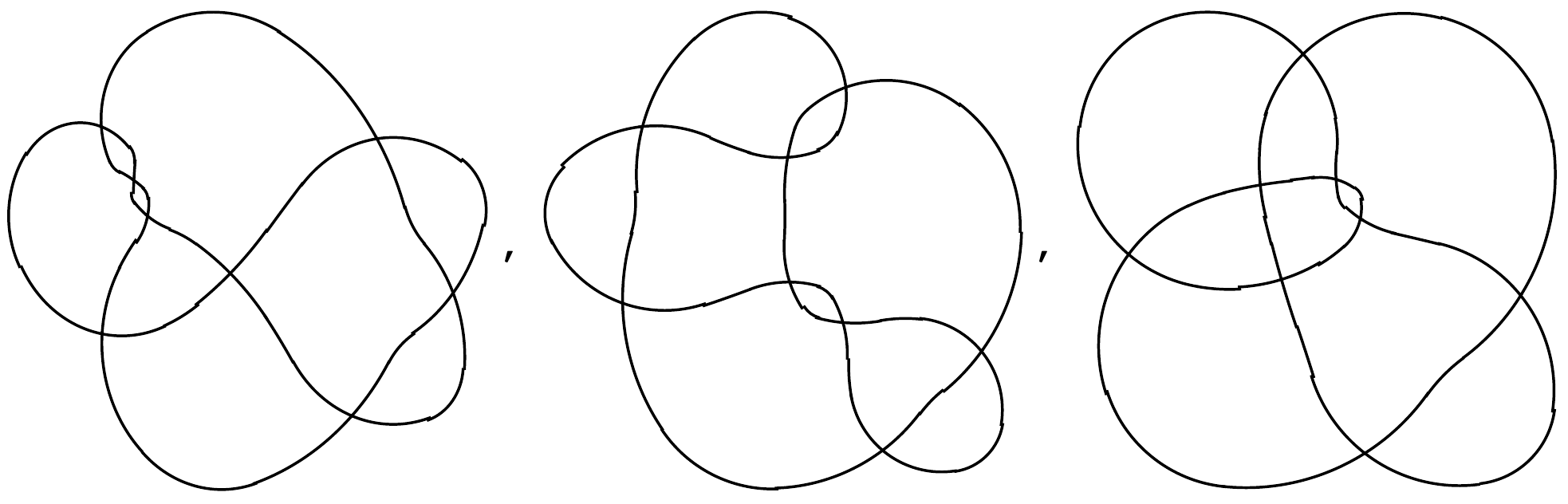}\raisebox{30pt}{;}
\includegraphics[width=25pc]{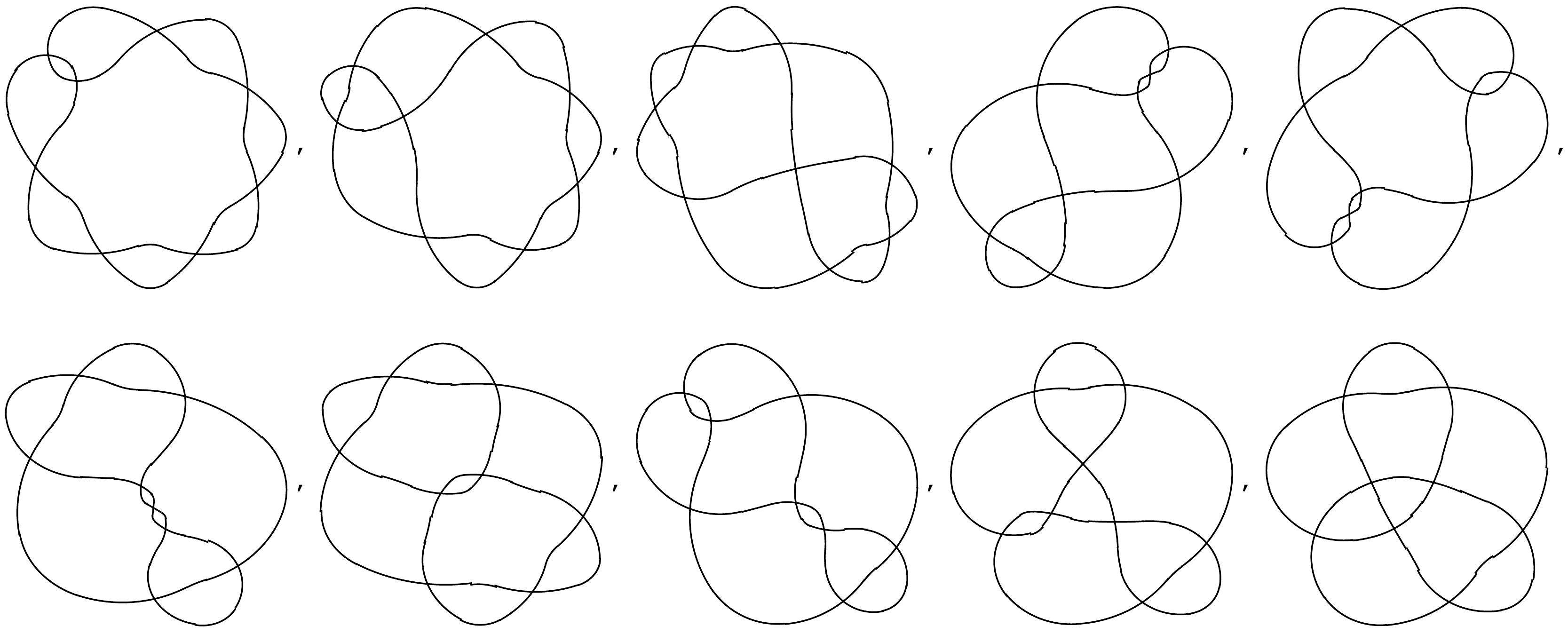}\raisebox{30pt}{;}
\includegraphics[width=20pc]{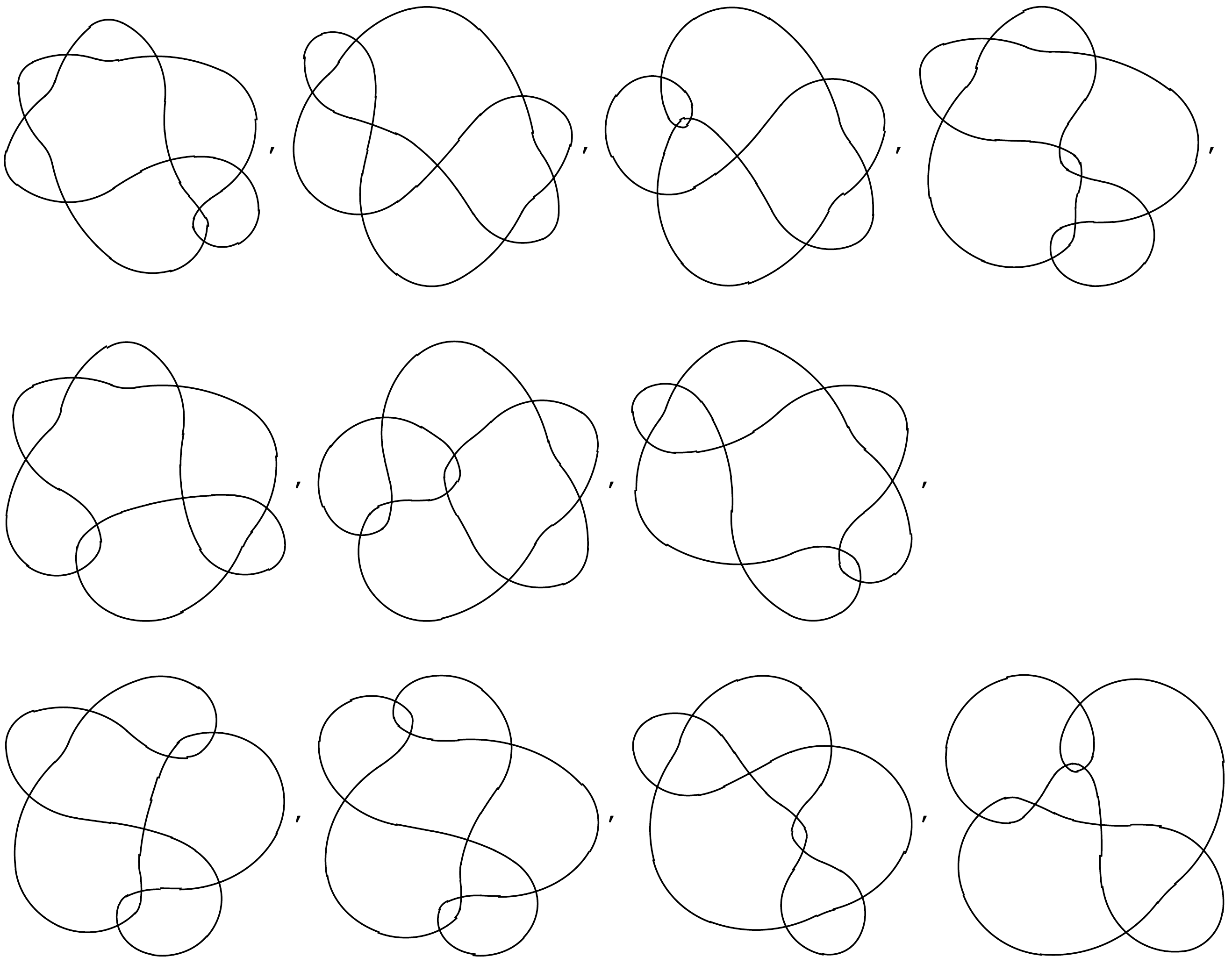}
\caption{\label{ImmUU8} The 27 indecomposable irreducible immersions of an unoriented circle in an unoriented sphere with 
$n=8$ double points. They may also serve as {\colour}ed and/or oriented immersions: the first three are invariant both under 
\duality ({\colour} swap or undercrossing $\leftrightarrow$ overcrossing) and mirror symmetry;
the next three are \duality invariant but have a mirror partner; 
the next 10 have identical swap,  mirror and orientation-reversal images ; and the last 11 
give rise to four images under \duality and mirror symmetry. 
In the notations of Sect.\ \ref{UOUUimm}, 
 the values of $x_{sm},y_{sm},z_{sm},v_{sm},w_{sm}$ restricted to this set of indecomposable irreducible immersions 
read $(3,3,0,10,11)$. (For all, the effect of orientation-reversal is the same as swapping.) We thus have 
3+3+10+11=27 immersions of type UU ; 3+6+10+22= 41 of type UO; 3+3+10+22=38 of type OU; 
and 3+6+20+44=73 OO or bi{\colour}ed UO immersions.}
\end{figure}

\begin{figure}[htbp]
 \centering
\includegraphics[height=18mm]{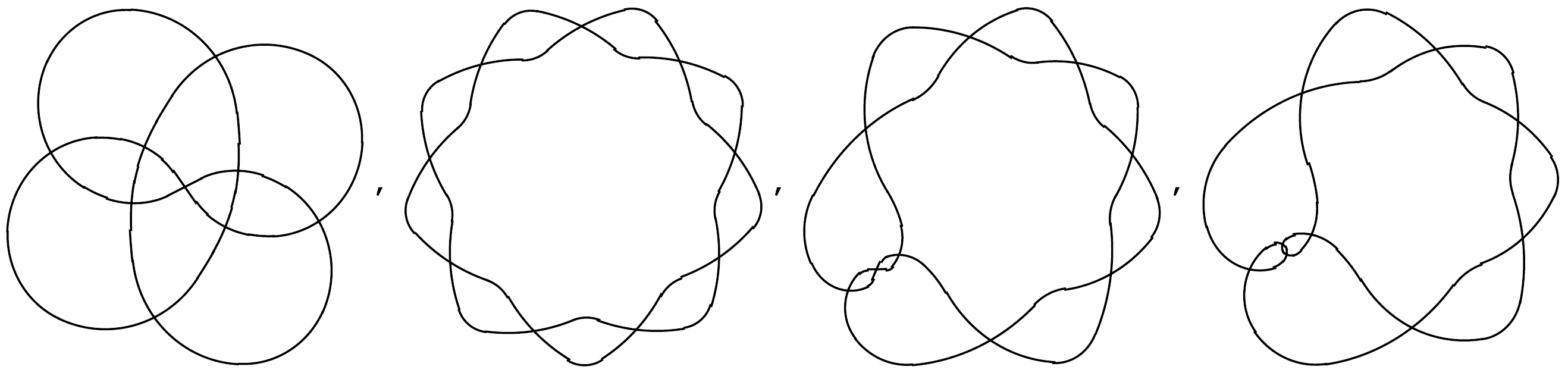}
\includegraphics[height=18mm]{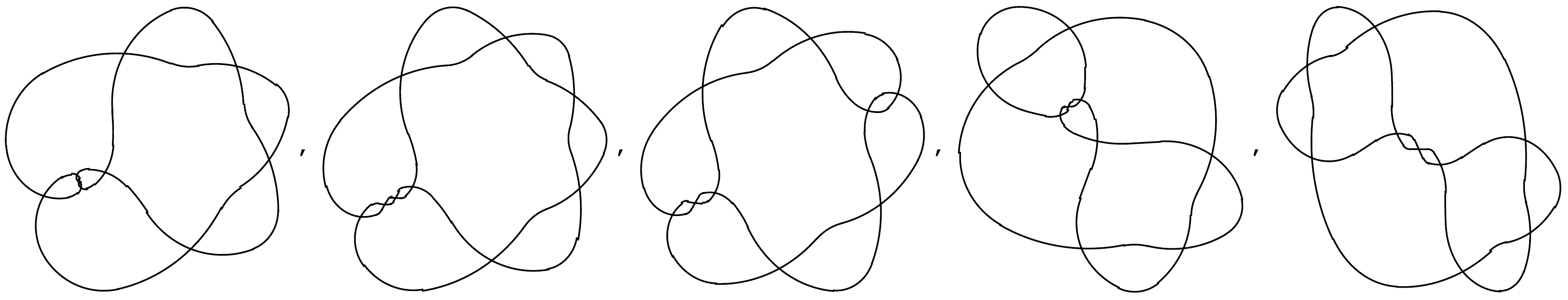}
\includegraphics[height=18mm]{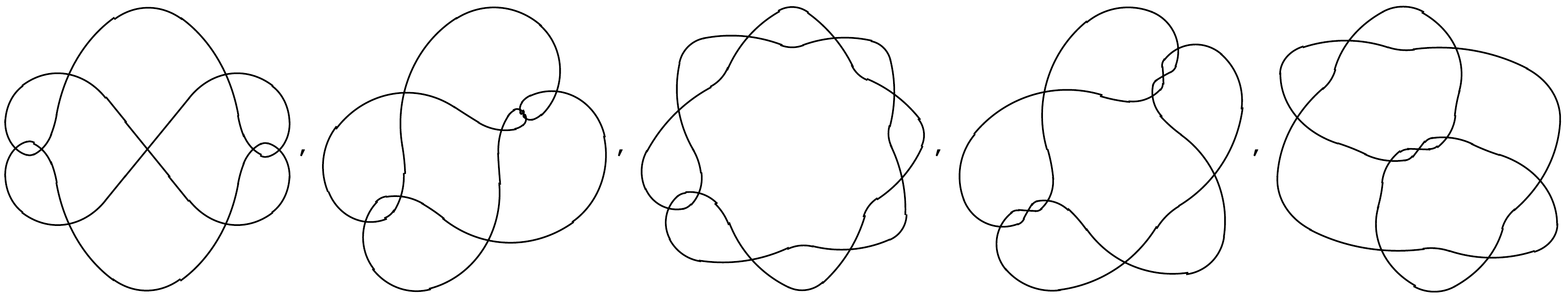}\ \raisebox{30pt}{;} 
\includegraphics[height=40mm]{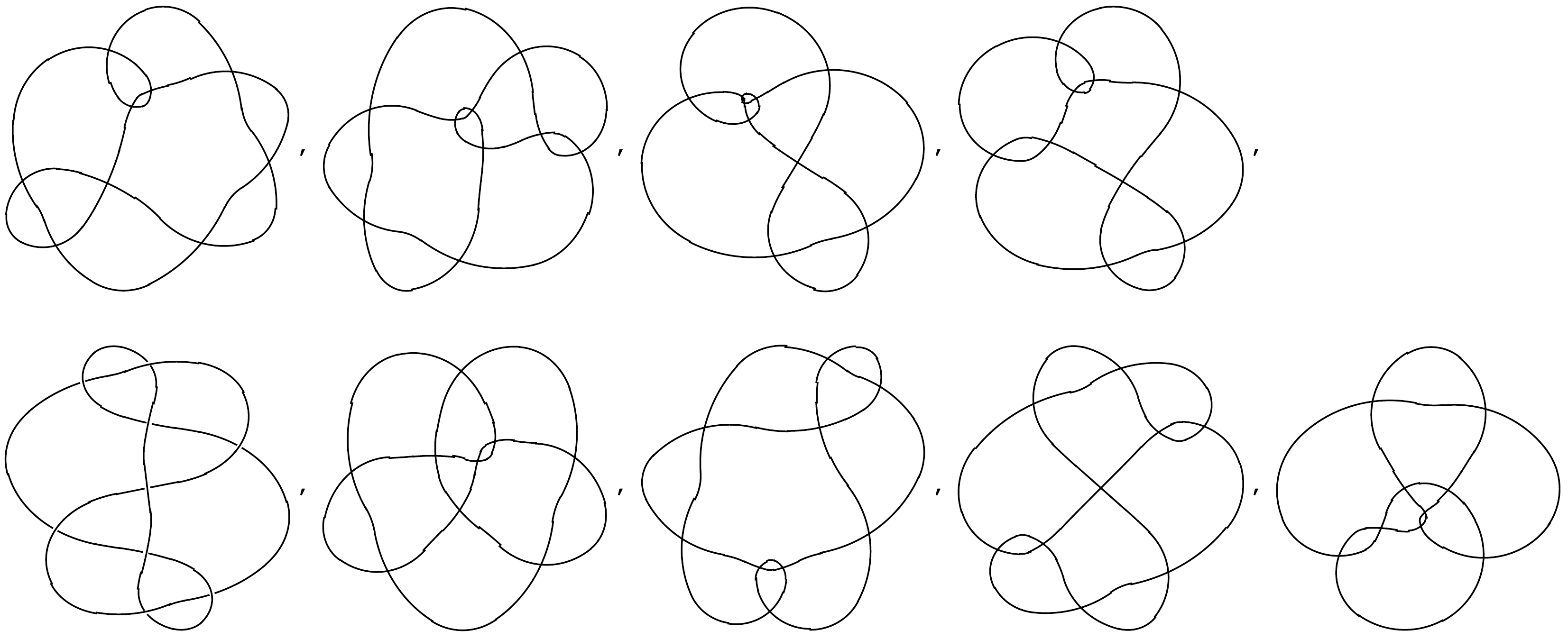}
\ \raisebox{30pt}{;}
\includegraphics[height=18mm]{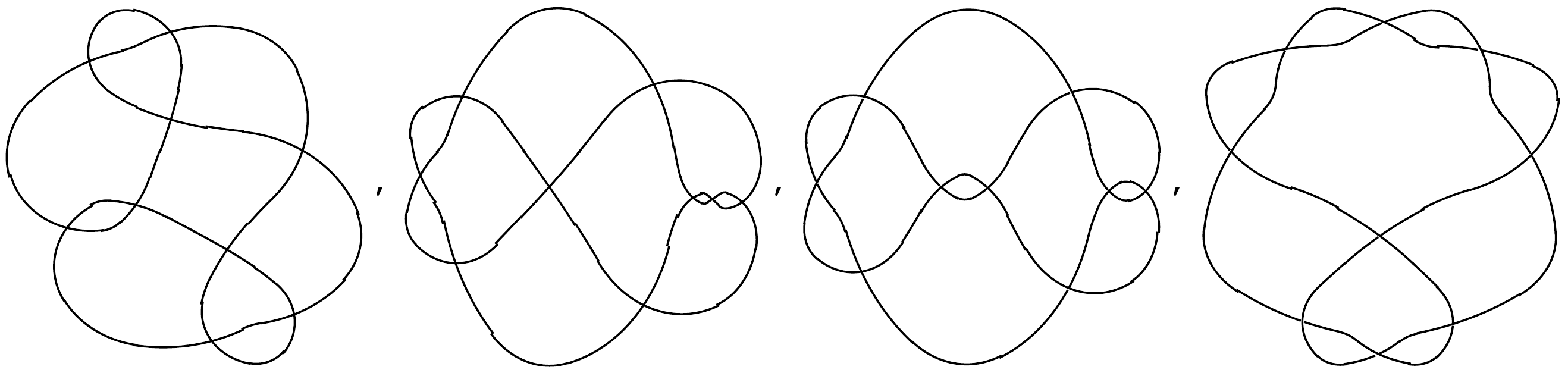}\ \raisebox{30pt}{;}
\includegraphics[height=18mm]{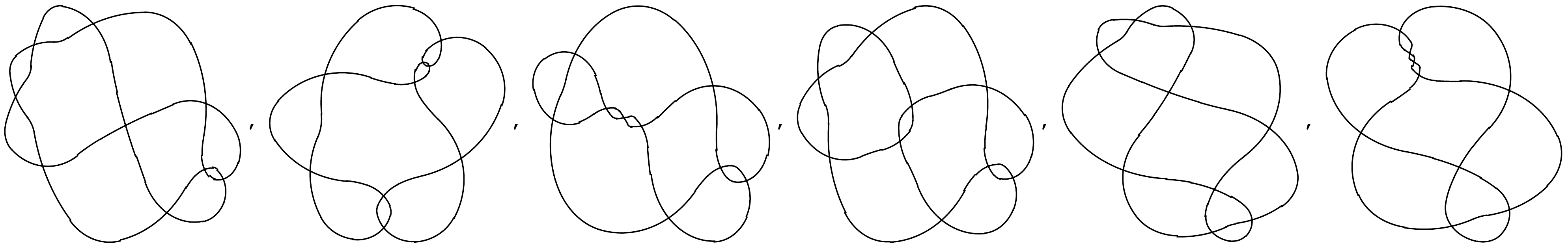}
\includegraphics[height=18mm]{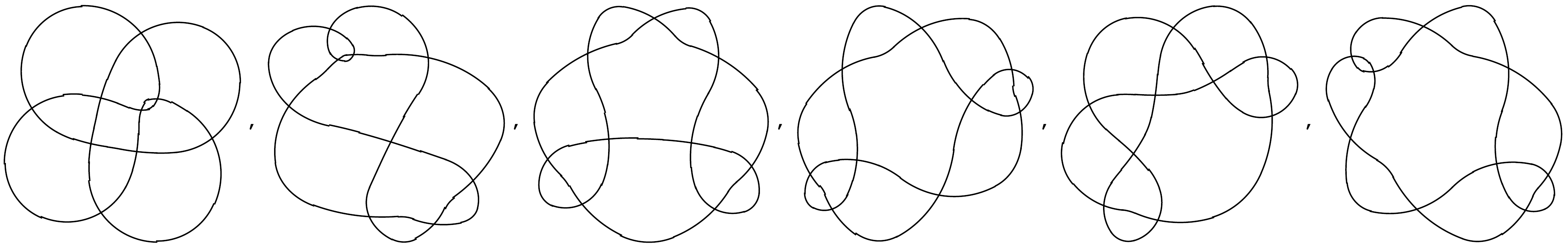}\
\includegraphics[height=18mm]{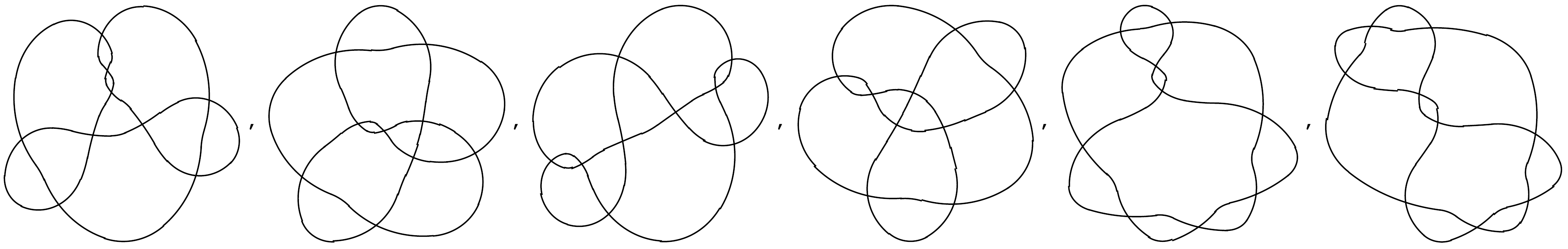}\
\includegraphics[height=18mm]{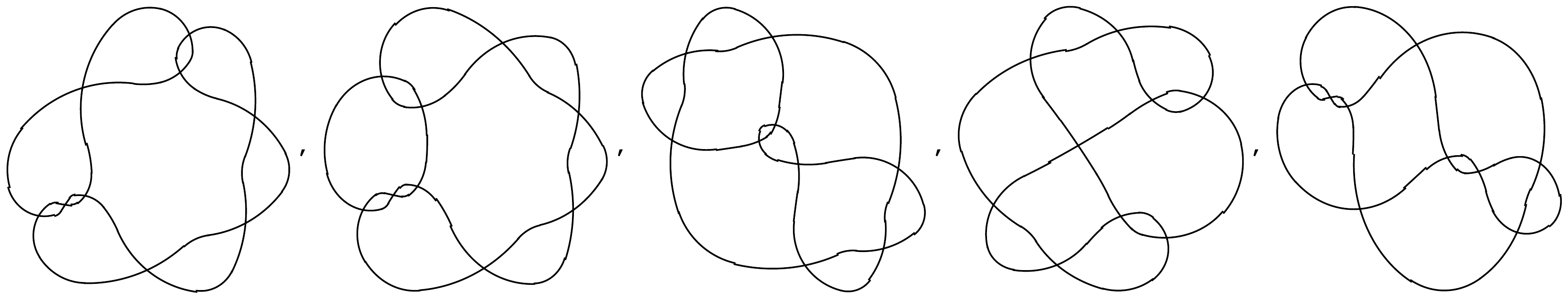}\
\ \raisebox{30pt}{;}
\caption{\label{ImmUU9v3a} The 101 indecomposable irreducible immersions of an unoriented circle in an unoriented sphere with 
$n=9$ double points: on that figure, the $x_{rm}=14$ immersions such that $\sigma\sim\sigma_m\sim\sigma_r$; the $y_{rm}=9$   ones such that 
$\sigma\sim\sigma_r\nsim\sigma_m$;  the $z_{rm}=4$ ones such that $\sigma\sim\sigma_m\nsim\sigma_r$;
the $v_{rm}=23$ ones such that $\sigma\sim\sigma_{rm}\nsim\sigma_r$; next figure, the $w_{rm}=51$ with no symmetry.}  
\end{figure}

\begin{figure}[htbp]
 \centering
\includegraphics[height=18mm]{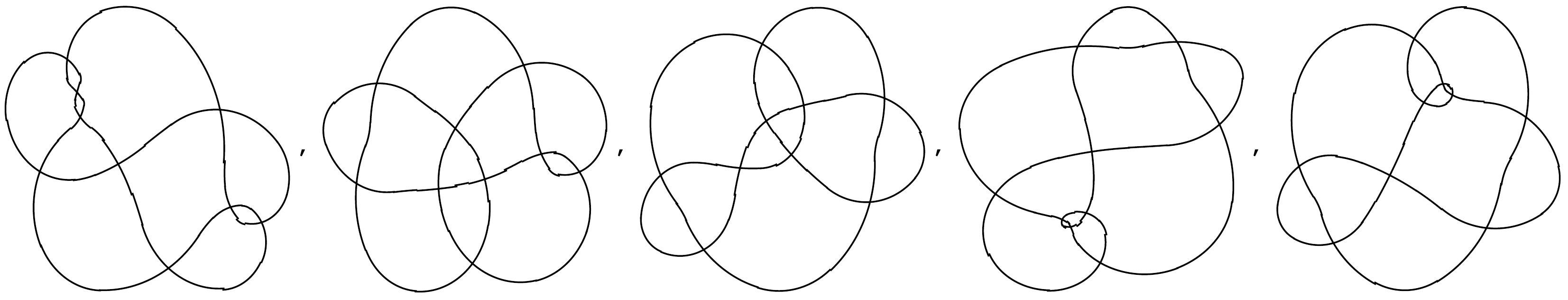}
\includegraphics[height=18mm]{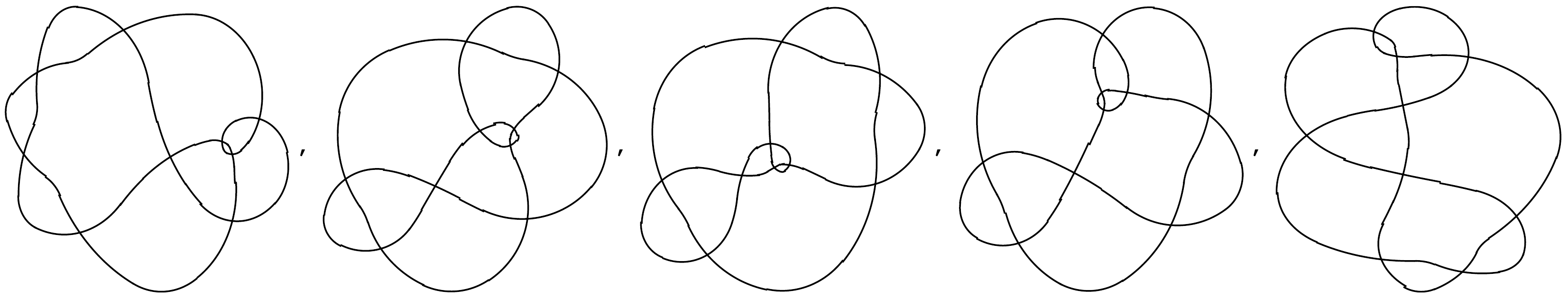}
\includegraphics[height=18mm]{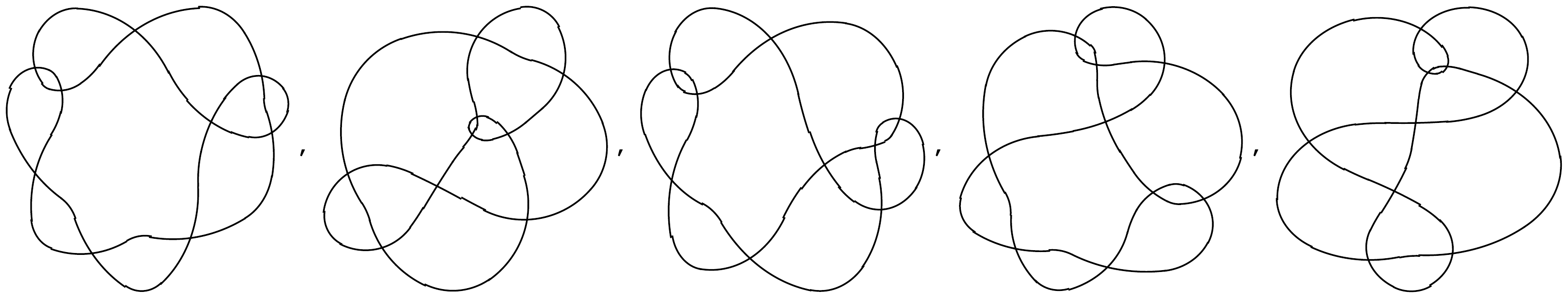}
\includegraphics[height=18mm]{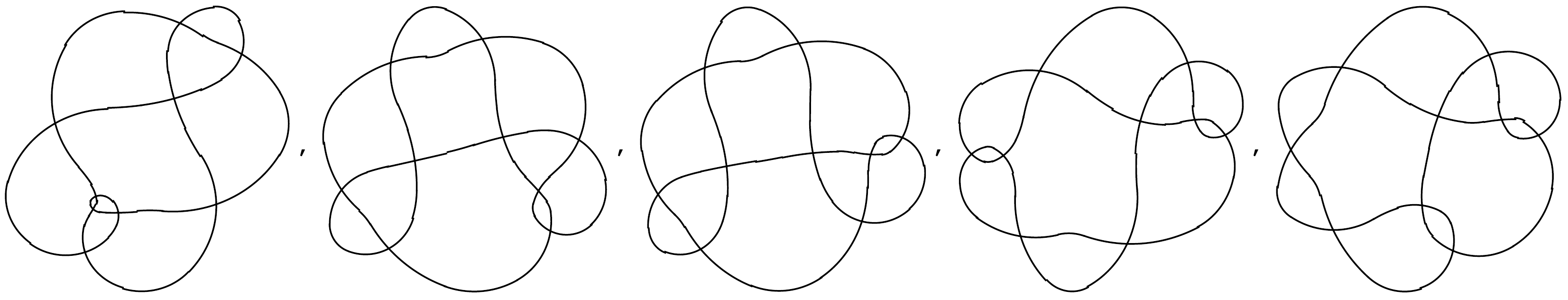}\ \raisebox{30pt}{;}
\includegraphics[height=18mm]{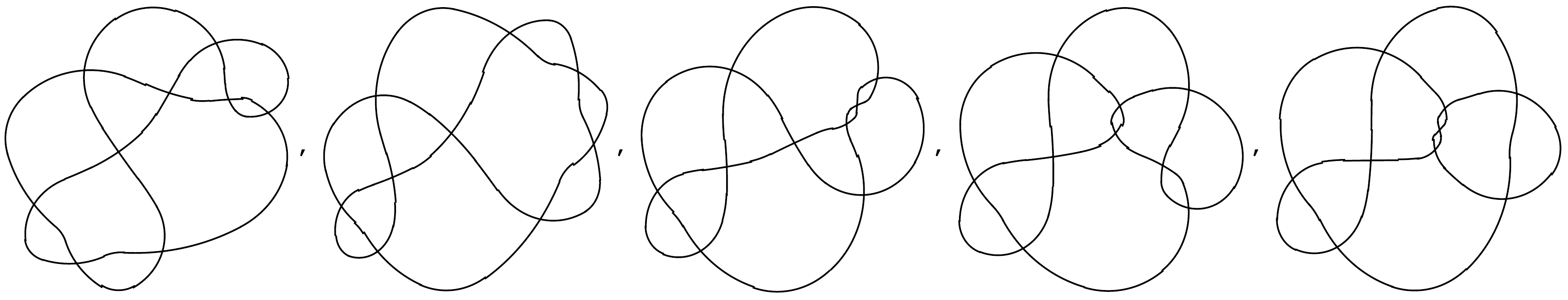}
\includegraphics[height=18mm]{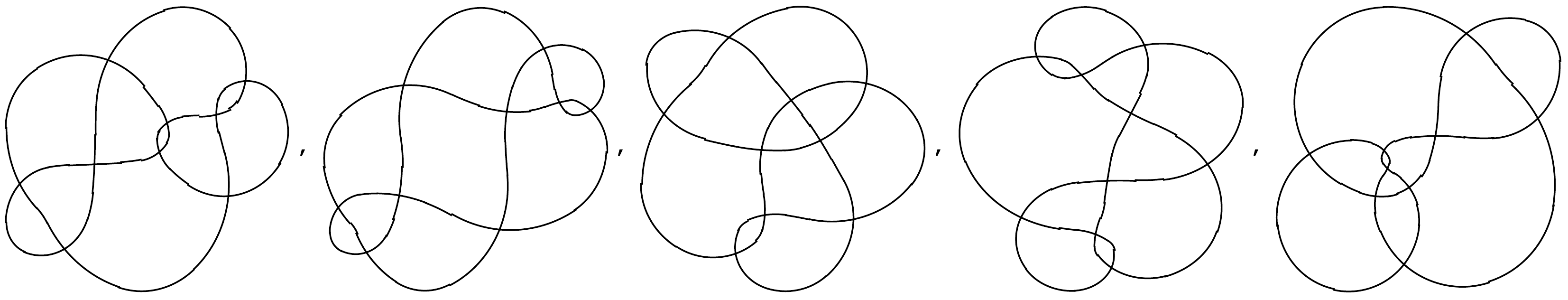}
\includegraphics[height=18mm]{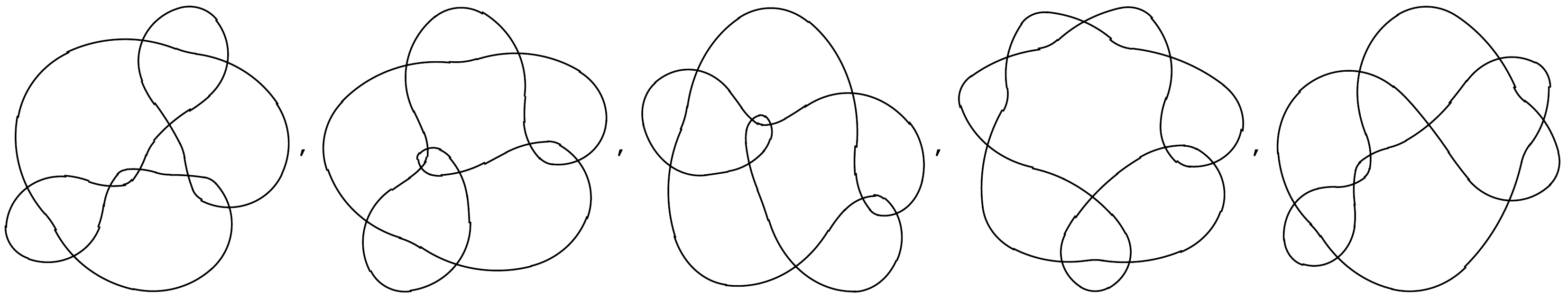}
\includegraphics[height=18mm]{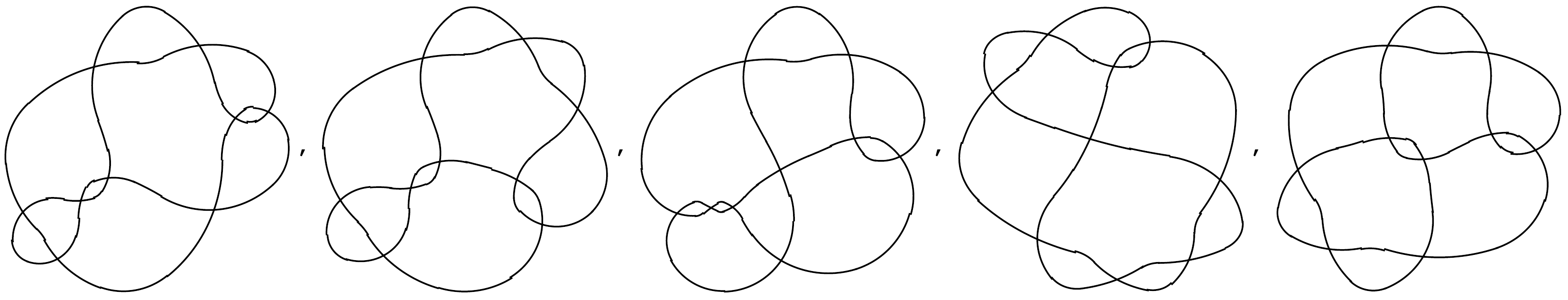}
\includegraphics[height=18mm]{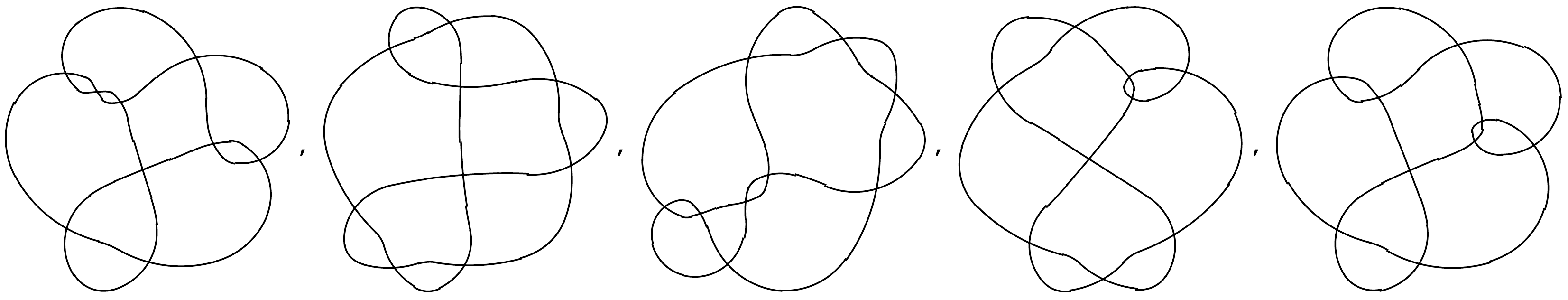}
\includegraphics[height=18mm]{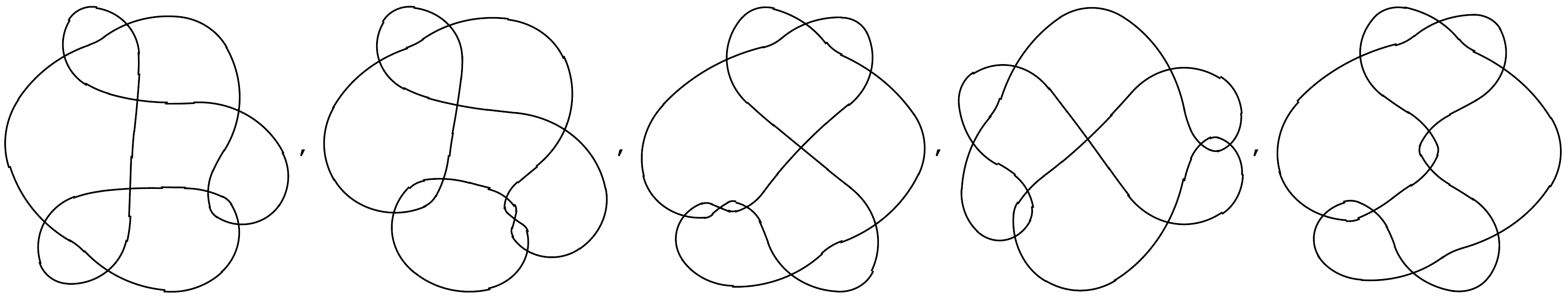}\ \raisebox{-10pt}{\includegraphics[height=28mm]{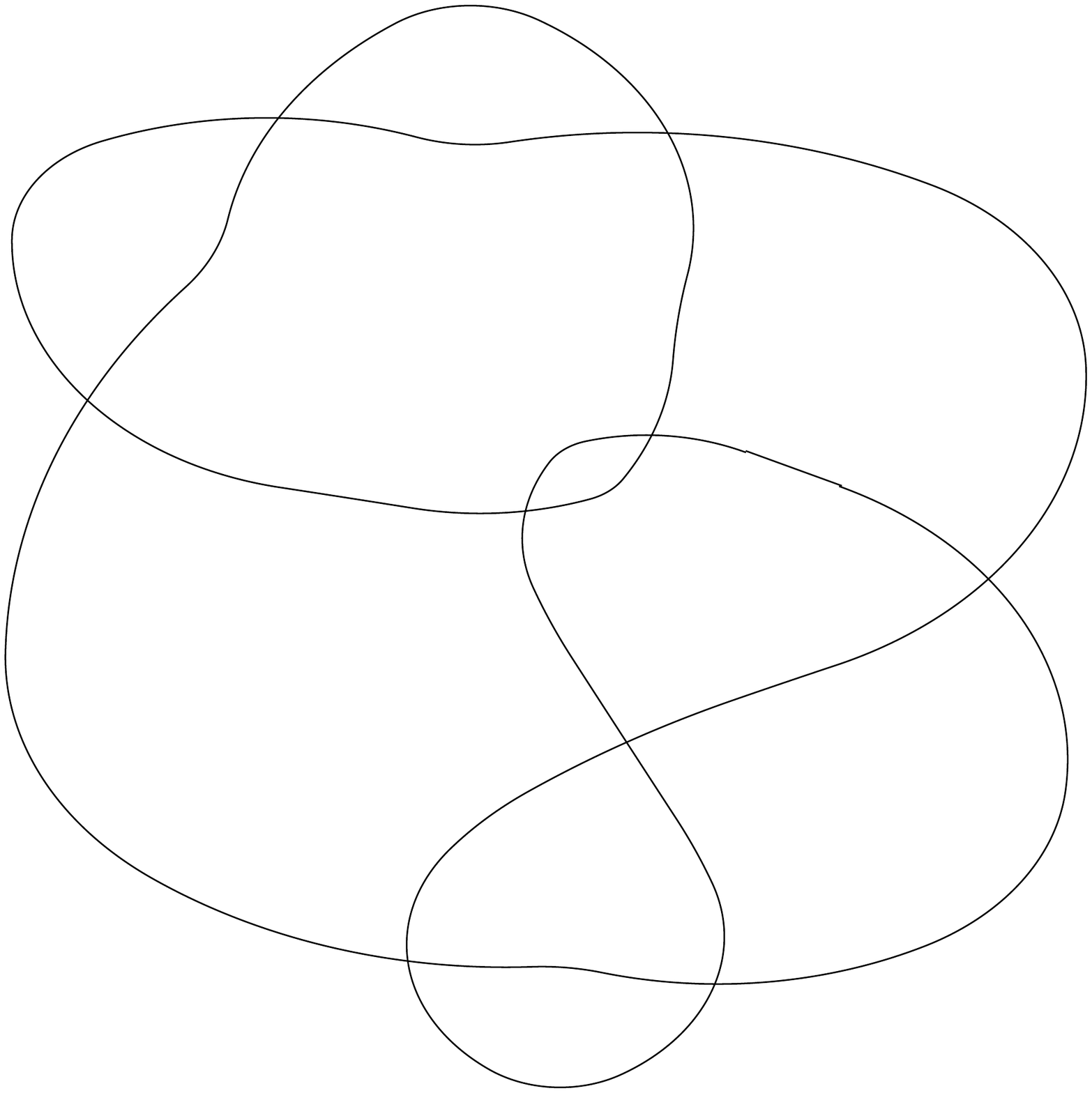}}
\ \raisebox{30pt}{.}
\caption{\label{ImmUU9v3b} The 101 indecomposable irreducible immersions of an unoriented circle in an unoriented sphere with 
$n=9$ double points (cont'd): the 51 immersions with no symmetry.}  
\end{figure}

\vfill \eject
%\vskip 1cm
 To summarize what we have achieved in this paper, 
  \vskip 0.5cm
 
 \noindent
% \begin{itemize}
 %\item 
 $\bullet$ we have emphasized the role of bi{\colour}ability and made explicit  12 different types of immersions that may be
 considered;\\
 %\item 
  $\bullet$ we have extended  existing series of numbers of spherical immersions to $n=10$ crossings;\\
% \item 
 $\bullet$  we have given tables of immersions (given here for $n=8$ and 9 for irreducible indecomposable immersions, see 
 Fig.\ \ref{ImmUU8}, 
 \ref{ImmUU9v3a}, \ref{ImmUU9v3b}, 
 but they are available on request for the other known cases);\\
 %\item 
  $\bullet$ we have extended to non zero genus  the counting of immersions and provided their cardinals up to $n=9$ or 10 
 crossings;\\
% \item 
 $\bullet$  we have discovered and proved novel relations between numbers of immersions of different types, see Theorem 
 \ref{Theo3c}. 
% \end{itemize}

\def\mapover#1{\buildrel #1 \over{\mapsto}}
\newpage
\section*{Appendix A : Details about the algorithms}
\label{AppA}
 All the tables found in the present paper, using the methods and algorithms discussed in the different sections, have been generated using computer programs written both in Mathematica \cite{Mathematica} and Magma \cite{Magma}.
Magma implements fast algorithms to determine explicitly the conjugates of a chosen group element with respect to some subgroup of the permutation group,  and to test whether  two elements are conjugated, this allows one to determine orbit representatives.
 Magma can also  determine very quickly the centralizer of a group element in a given subgroup of a permutation group; this feature is used in many places in our calculations,  for instance  when we determine the orbit sizes. 
 \ommit{It can also be used as a check: in some cases,  for high values of $n$,  {rather than considering} all the permutations belonging to a given subset we may use randomly generated permutations;  determining the order of their centralizer with respect to an appropriate group of reparametrization allows us to determine the corresponding orbit sizes and to check that we have generated enough orbit representatives,  since the given subset of interest is often of known cardinality.}
 We implemented in Magma the Frobenius formula (\ref{frobenius}) that only uses the cardinality of the absolute or relative conjugacy classes (\ie relative to the whole permutation group,  or relative to specific subgroups);  
 as the determination of the size of such conjugacy classes, together with representatives elements for each class, is very fast in Magma, our algorithm turns out to be much faster than the available commands giving the size of double cosets.
    
In  Sect.\ \ref{X},  we work in $S_{4n}$ to study immersions with $n$ crossings, and the number of permutations to be handled becomes unfortunately very high, even for modern processors; it becomes time and memory consuming to go beyond $n=6$ by this technique.   

In  Sect.\ \ref{Y}, for low values of $n$ (up to 6) a direct enumeration of all elements of $Y^{\prime}$ and an explicit construction of their orbits, together with the different kinds of immersions, was possible both in Mathematica 
and Magma. Initially, our first method, for larger values of $n$, up to 9, was to perform, using Mathematica,  a random sampling of $Y^{\prime}$ followed by the determination of a typical representative of each orbit of $Y^{\prime}$,  therefore giving a list of orbits. The sampling was continued until the results stabilize and the procedure was finally certified by checking the sum rule $\sum_{{\rm orbits }\, o}\, \ell_o=\card{ Y'}=2^{2n-1}(n-1)! n!$, where the length $ \ell_o$  of each orbit $o$ was determined independently by use of Magma (determination of the order of stabilizers of orbits points). 

Replacing the sets $Y$ (actually $Y^{\prime}$) by the $\CC_\rho$ left coset $U$, and the adjoint action of $\CC_\rho$ by the action of its dihedral subgroup or  of the  appropriate cyclic subgroup of the latter, 
 allowed us, at a later stage (see Sect. \ref{gaugefixing}), to recover all these results, including the determination of representatives for all orbits of all kinds of immersions, up to $n=9$, by a direct enumeration of all elements of $U$, using Magma. A comparison between the lengths of orbits obtained for these different group actions will be done below, together with a particular example. 
 For $n=10$ we could not determine representatives for the orbits of $Y^{\prime}$ or $U$.
We used again the same random sampling method for genus 0 until the results stabilize, but, unfortunately as we have no {\it a priori} knowledge of $\card{ Y''}$, we had no way to check the correctness of the result by using a sum rule.

Finally the orbits for the adjoint action of (a particular subgroup) $S_n$ on the cyclic permutations of $S_{2n}$, leading to the number of immersions of type OO, OU and UO, with no constraint of bicolourability (the ``Z method'' of Sect.\ \ref{Zpmethod}), were obtained both using Mathematica (random sampling) and Magma (full enumeration of orbits), up to $n = 9$ for all genera, and {$n=10$ in genus~$0$}.
Remember that representative elements of orbits are needed in order to consider the effects of the five types of symmetries that match  Arnold's classification. 
The number of orbits in $Z^\prime = [2n]$ itself (Table \ref{TableZp}), aka the total number of immersions of OO type (summing over all genera), and its variants of types UO, OU, UU,  
were calculated using both Magma and a Frobenius formula on double cosets,  see also our comments in Sect.\ \ref{dblecosZ}.
The number of immersions of type OO, OU, UO and UU were then quickly recovered by using double cosets and Prop.\ \ref{doublecosetmeth}; this latter method gives however slighly less information than the former (full enumeration of orbits)  since it does not determine the five parameters $x_{rm}, y_{rm}, z_{rm}, v_{rm}, w_{rm}$ describing symmetries of orbit types.

\paragraph{More on the action of $\CC_\rho$ and $D_n$ on $U$.}
The restriction of the action of $\CC_\rho$ to its subgroup $D_n$ defines  an action of the latter for which the set $U$ is stable: the points of intersection between $U$ and a given orbit of  $\CC_\rho$ define an orbit of $D_n$, of length 
$\vert D_n \vert / k$ = $2n/k$, where $x \in U$ and  $k=|C(D_n,x)|$.  On the other hand, the orbit of $\CC_\rho$ going through $x$ has length $\vert \CC_\rho \vert /k^\prime$ where $k^\prime=|C(\CC_\rho,x)|$. \\
We shall prove below that $k^\prime = k$ but let us take this property for granted at the moment. 
As the discussion can be carried out independently for the different genera, let us call $\mu(k)$,  the number of elements of the left coset $U$, of fixed genus  (call $U_g$ this subset), whose centralizer in $D_n$ has order $k$. 
These $\mu(k)$ elements can be gathered into $\mu(k)/ (2n/k)$  orbits of $D_n$, but each orbit of $D_n$ determines one orbit of $\CC_\rho$, so these $\mu(k)$ elements of $U$ determine $\lambda(k) = \tfrac{k}{2n} \mu(k)$ orbits of $\CC_\rho$, of length $\vert\CC_\rho\vert/k$. This discussion is summarized in the following proposition:

\begin{myprop}\label{prop4} 
For any genus $g$, and for all $x$ in $U_g$, the following two centralizer subgroups are equal: $C(\CC_\rho,x) = C(D_n,x)$. 
Denoting by $k$ their common order, we call $\mu(k) = \#\{x \in U_g \; : \; \vert C(D_n,x) \vert = k \}$. 
The number of orbits of length $2n/k$, for the adjoint action of $D_n$ on $U_g$, is equal to  $\lambda(k) = \tfrac{k}{2n} \, \mu(k)$.
With the notations of the text, $\lambda(k)$ is also the number of orbits  of length $\vert\CC_\rho\vert/k$, for the adjoint action of $\CC_\rho$ on the set $Y^\prime_g$.
\end{myprop}
\begin{proof}
It remains to prove that  $C(\CC_\rho,x) = C(D_n,x)$ for all $x \in U$, hence $k=k'$ as stated previously.
One inclusion ($C(D_n,x)  \subset C(\CC_\rho,x) $) is obvious, since $D_n \subset
\CC_\rho$. Now take $y$ in $C(\CC_\rho,x)$, so $y \in \CC_\rho$ and $y x = x y$;
since $U = \beta \CC_\rho$, one can write $x = \beta \ z$ for some $z \in \CC_\rho$,
and the commutation property reads 
$y \beta z = \beta z y$, equivalently $y = \beta (z y z^{-1}) \beta^{-1}$. But $z y
z^{-1} \in \CC_\rho$ so $y \in \CC_\rho^\beta$. The conclusion is that $y \in
\CC_\rho \cap \CC_\rho^\beta$, but the latter subgroup coincides with $D_n$
(this way of defining $D_n$ was used in Sect.\ \ref{Y} and \ref{gaugefixing}).
So we have also $C(\CC_\rho,x) \subset C(D_n,x)$, hence the equality.
\end{proof}

 Proposition \ref{prop4} has a practical value:  identifying distincts orbits of $Y^\prime$ under the adjoint action of $\CC_\rho$ is a time-consuming task that
is replaced by the calculation of the order of a (small) finite group associated with the elements of a left coset $U$ of that group: this is much faster.
The result is illustrated on the following  
example : With $k$, the order of the centralizer $C(D_n, x)$ of $x$ in $U_0$ (the subset of the permutations of genus $0$ belonging to the left coset $U$),  and using the notation  $k^{\mu(k) = \# \text{orbits of length} |D_n|/k}$, one obtains, for $n=5$, the following  sizes and numbers of $D_n$ orbits:  $1^{1640} \ 2 ^{150} \ 5^{4}\ 10^{2}$, with $\vert U_0 \vert = 1796$, the number of long (open) spherical curves. 
%In order to obtain t
The number and sizes of orbits of $Y^\prime$ under the adjoint action of the group $\CC_\rho$
%, one has just to multiply each $k$ by the correcting factor $\mu(k)/2n$,
is given by a similar formula with the ``exponent" multiplied by the correcting factor $k/2n$,
 so that we get, instead, $1^{164}\ 2^{30}\ 5^{2}\ 10^{2}$, for a total of $198$ orbits (UOc bicoloured spherical immersions). A similar analysis can be done if we replace the dihedral subgroup $D_n$ by its cyclic subgroup $Z_n$, the correcting factor being this time equal to $k/n$: we have $1^{1790} 5^{6}$ $Z_n$-orbits  in $U_0$ and $1^{358} 5^{6}$ $Z_n$-orbits in $Y^\prime$, for a total of $364$ orbits (OOc  bicoloured spherical immersions).
\\ [-10pt]
\paragraph{Typical CPU time ($T$) and memory ($M$)} for calculations done on a MacBookPro 2.8 GHz Intel Core i7, leading to the results given in Table \ref{TableXYZ1} (bi{\colour}able and/or bi{\colour}ed immersions) are as follows:
$n  \leq 4: T < 0.4 \, s, \; M < 32 \, \text{MB}$; 
$n = 5: T = 0.63 \,  s, \; M < 32 \, \text{MB}$;
$n = 6: T = 4.37  \, s, \; M < 32 \, \text{MB}$;
$n = 7: T = 70.64 \,  s, \; M = 116.88 \, \text{MB}$;
$n = 8: T =  3285  \, s, \; M = 1316.81 \, \text{MB}$. 
For $n=9$, calculations were done genus by genus on a faster machine, with a large amount of available random access memory, but the results for each genus nevertheless required several hours of computer time.
For $n=10$, the enumerative algorithm was traded for a sampling method (see above), implemented in Mathematica, and required several weeks of CPU.
 With the exception of the total number of immersions (summing over genera) of all types, obtained (up to $n=20$) by a fast algorithm using double cosets, 
 calculations leading to Table \ref{TableXYZ2} (general immersions) are significantly slower and use more memory than the previous ones because we use the whole class of cyclic permutation (growing like $(2n-1)!$ for $n$ crossings).
 They could nevertheless be performed with enumerative methods up to $n=9$. Typical values are as follows: 
 $n = 6: T = 12  \, s, \; M < 32 \, \text{MB}$; $n = 7: T = 340 \,  s, \; M = 258 \, \text{MB}$;
 $n = 8: T =  6293  \, s, \; M = 4934 \, \text{MB}$; $n = 9: T = 106893  \, s, \; M
 = 124.5 \, \text{GB}$. 
% \\  %[-5pt]

%%%%%%%%%%%%%%%%%%%%%%%%%%%%%%%%%%%%%%%%%%
\vfill \eject
\section*{Appendix B}
%\vglue -0.1cm
%\vfill \eject
%\subsection*{B.1 The five parameters $x_{sm},\,y_{sm},\,z_{sm},\,v_{sm},\, w_{sm}$ for the $\CC_\rho$-orbits of $Y'$ (or for the $D_n$-orbits of $U$)}
{\bf B.1 The five parameters $x_{sm},\,y_{sm},\,z_{sm},\,v_{sm},\, w_{sm}$ for the $\CC_\rho$-orbits of $Y'$ (or for the $D_n$-orbits of $U$)}
{\qquad} \\[-56pt]
\begin{center}
\[
\begin{array}{cc}
\begin{array}{cc}
x_{sm},y_{sm},z_{sm},v_{sm},w_{sm} & {} \\[-10pt]
 \\[\smallskipamount]
 \smat{1&0&0&0&0}   & n=1 \\[-5pt]
 \\[\smallskipamount]
 \smat{1&0&0&1&0}   & n=2 \\[-5pt]
  \\[\smallskipamount]
   \smat{ 0& 0& 0& 6& 0 \cr 0& 0& 0& 1& 0\cr} & n=3 \\[-5pt]
  \\[\smallskipamount]
\smat{ 5 & 0 & 0 & 12 & 2 \cr
 0 & 0 & 0 & 4 & 2 \cr
 1 & 0 & 0 & 0 & 0 \cr} & n=4 \\
  \\[\smallskipamount]
 \smat{0 & 0 & 0 & 53 & 23 \\
 0 & 0 & 0 & 33 & 30 \\
 0 & 0 & 0 & 8 & 5 }  & n=5 \\
 \ommit{ \\[\smallskipamount]
  \smat{9 & 12 & 3 & 152 & 200 \\
 0 & 0 & 0 & 133 & 406 \\
 7 & 10 & 6 & 50 & 169 \\
 0 & 0 & 0 & 3 & 8} & n=6}
\end{array}
&
\qquad\qquad
\begin{array}{cc}
{} & {} \\
\smat{9 & 12 & 3 & 152 & 200 \\
 0 & 0 & 0 & 133 & 406 \\
 7 & 10 & 6 & 50 & 169 \\
 0 & 0 & 0 & 3 & 8} & n=6\\
 \\[\smallskipamount]
\smat{0 & 0 & 0 & 559 & 1635 \\
 0 & 0 & 0 & 758 & 4750 \\
 0 & 0 & 0 & 460 & 3723 \\
 0 & 0 & 0 & 84 & 579 } & n=7 \\
  \\[\smallskipamount]
\smat{39 & 105 & 29 & 1756 & 12685 \\
 0 & 0 & 0 & 3042 & 53025 \\
 47 & 228 & 104 & 2500 & 67239 \\
 0 & 0 & 0 & 725 & 23182 \\
 10 & 39 & 21 & 29 & 937 } & n=8 \\
  \\[\smallskipamount]
\smat{ 0& 0& 0& 6299& 100122\\0& 0& 0& 14861& 577596\\0& 0& 0& 16601& 
  1102579 \\ 0& 0& 0& 8004& 684745\\0& 0& 0& 1180& 93539} & n=9 
\end{array}
 \end{array}
 \]
\end{center}
where, for each $n$, successive rows correspond to increasing genus \\
 \phantom{x}\hspace{5ex} $g=0,1\cdots, \lfloor\frac{n}{2}\rfloor$ (for $n>2 $).

  For $n=10$ we have only the genus 0 data : \\
\phantom{x}\hspace{5ex}  $(x_{sm},y_{sm},z_{sm},v_{sm},w_{sm})_{g=0}=\smat{98, 969, 247, 20681, 801837}\qquad n=10\,.$

\vskip 0.8cm

%\subsection*{B.2 The five parameters $x_{sr},\,y_{sr},\,z_{sr},\,v_{sr},\, w_{sr}$ for the $Z_n$-orbits of~$U$}
{\bf B.2 The five parameters $x_{sr},\,y_{sr},\,z_{sr},\,v_{sr},\, w_{sr}$ for the $Z_n$-orbits of~$U$}
{\qquad} \\[-15mm]
 \begin{center}
\[
\begin{array}{cc}
\begin{array}{cc}
x_{sr},\,y_{sr},\,z_{sr},\,v_{sr},\, w_{sr} & {} \\
  \\[\smallskipamount]
 \smat {0&0&1&0&0}  & n=1 \\
   \\[\smallskipamount]
 \smat {0&0&0&1&1}  & n=2 \\
     \\[\smallskipamount]
 \smat{ 0& 0& 3& 0& 3\cr  0& 0& 1& 0& 0\cr}  & n=3 \\
   \\[\smallskipamount]
 \smat{0& 0& 0& 5& 16 \cr 0& 0& 0& 0& 8 \cr 0& 0& 0& 1& 0 \cr} & n=4 \\
   \\[\smallskipamount]
\smat{0& 0& 16& 0& 83\cr 0& 0& 16& 0& 77\cr 0& 0& 0& 0& 18 \cr}  & n=5 \\
\end{array}
&
\qquad\qquad \begin{array}{cc}
{} & {} \\
\smat{0& 0& 0& 33& 555\cr 0& 0& 0& 0& 945\cr 0& 0& 0& 27& 394\cr 0& 0& 0& 0& 19\cr} & n=6 \\
  \\[\smallskipamount]
\smat {0& 0& 105& 0& 3724\cr 0& 0& 210& 0& 10048\cr 0& 0& 57& 0& 7849\cr 0& 0& 12& 0& 1230}  & n=7 \\
  \\[\smallskipamount]
\smat{
0 & 0 & 0 & 249 & 27155 \cr
 0 & 0 & 0 & 0 & 109092 \cr
 0 & 0 & 0 & 503 & 137082 \cr
 0 & 0 & 0 & 0 & 47089 \cr
 0 & 0 & 0 & 88 & 1924
} & n=8 \\
  \\[\smallskipamount]
 \smat{
0 & 0 & 780 & 0 & 205763 \cr
 0 & 0 & 2600 & 0 & 1167453 \cr
 0 & 0 & 1708  & 0 & 2220051 \cr
 0 & 0 & 928 & 0 & 1376566 \cr
 0 & 0 & 128 & 0 & 188130
 } & n=9
 \end{array}
  \end{array}
 \]
 \end{center}
 
% \vfill \eject
  \vskip 3.cm
%\subsection*{B.3 The five parameters\footnote{They should not be confused with those of Appendix B.1}  $x_{sm},\,y_{sm},\,z_{sm},\,v_{sm},\, w_{sm}$ for the $Z_n$-orbits of $U$}
{\bf B.3 The five parameters\footnote{They should not be confused with those of Appendix B.1}  $x_{sm},\,y_{sm},\,z_{sm},\,v_{sm},\, w_{sm}$ for the $Z_n$-orbits of $U$}

\vskip 0.3cm
{\qquad} \\[-21mm]
 \begin{center}
\[
\begin{array}{cc}
\begin{array}{cc}
x_{sm},\,y_{sm},\,z_{sm},\,v_{sm},\, w_{sm} & {}  \\
  \\[\smallskipamount]
\smat{0&0&0&1&0}  & n=1 \\
  \\[\smallskipamount]
 \smat{0&0&1&0&1}  & n=2 \\
   \\[\smallskipamount]
 \smat{0 & 0 & 0 & 3 & 3 \cr  0 & 0 & 0 & 1 & 0 \cr}  & n=3 \\
   \\[\smallskipamount]
 \smat{ 0 & 0 & 5 & 0 & 16 \cr  0 & 0 & 0 & 0 & 8 \cr  0 & 0 & 1 & 0 & 0 \cr}  &  n=4 \\
   \\[\smallskipamount]
\smat{0 & 0 & 0 & 12 & 85 \\
 0 & 0 & 0 & 16 & 77 \\
 0 & 0 & 0 & 4 & 16}  &  n=5 
 \end{array}
 &
\qquad\qquad \begin{array}{cc}
{} & {} \\
 \smat{0 & 0 & 15 & 0 & 564 \\
 0 & 0 & 0 & 0 & 945 \\
 0 & 0 & 19 & 0 & 398 \\
 0 & 0 & 0 & 0 & 19 } &  n=6 \\
   \\[\smallskipamount]
\smat{0 & 0 & 0 & 71 & 3741 \\
 0 & 0 & 0 & 206 & 10050 \\
 0 & 0 & 0 & 141 & 7807 \\
 0 & 0 & 0 & 48 & 1212 }  &  n=7 \\
   \\[\smallskipamount]
 \smat{
 0 & 0 & 97 & 0 & 27231 \\
 0 & 0 & 0  & 0 & 109092 \\
 0 & 0 & 255 & 0 & 137206 \\
 0 & 0 & 0 & 0 & 47089 \\
 0 & 0 & 52 & 0 & 1942}  &  n=8 \\
   \\[\smallskipamount]
  \smat{
 0 & 0 & 0 & 514 & 205896 \\
 0 & 0 & 0  & 2498 & 1167504 \\
 0 & 0 & 0 & 3314 & 2219248 \\
 0 & 0 & 0 & 2452 &  1375804 \\
 0 & 0 & 0 & 616 & 187886  } &  n=9 
 \end{array}
 \end{array}
 \]
 \end{center}
where, for each $n\!\!\!>\!\!\!2$, successive rows correspond to increasing genus %\\
 %\phantom{x}\hspace{5ex} 
  $g=0,1\cdots, \lfloor\frac{n}{2}\rfloor$. %(for $n>2 $).
 
 \vskip 1.cm

%\subsection*{B.4 The five parameters $x_{rm},\,y_{rm},\,z_{rm},\,v_{rm},\, w_{rm}$ for the $S_n$-orbits of $Z^\prime$}
{\bf B.4 The five parameters $x_{rm},\,y_{rm},\,z_{rm},\,v_{rm},\, w_{rm}$ for the $S_n$-orbits of $Z^\prime$}

\vskip 0.3cm
 {\qquad} \\[-16mm]
  \begin{center}
\[
\begin{array}{cc}
\begin{array}{cc}
 x_{rm},\,y_{rm},\,z_{rm},\,v_{rm},\, w_{rm} & {} \\[-15pt]
    \\[\smallskipamount]
   \smat{ 1 & 0 & 0 & 0 & 0 } & n=1 \\[-8pt]
     \\[\smallskipamount]
   \smat{ 1 & 0 & 0 & 1 & 0 \\
 1 & 0 & 0 & 0 & 0} & n=2 \\[-8pt]
      \\[\smallskipamount]
    \smat{3 & 0 & 0 & 3 & 0 \\
 1 & 0 & 0 & 3 & 1 \\
 0 & 0 & 1 & 0 & 0 } & n=3 \\[-8pt]
       \\[\smallskipamount]
\smat{5 & 0 & 0 & 12 & 2 \\
 7 & 4 & 2 & 17 & 15 \\
 2 & 1 & 2 & 4 & 13 } & n=4 \\[-8pt]
       \\[\smallskipamount]
     \smat{10 & 3 & 1 & 42 & 20 \\
 10 & 3 & 3 & 98 & 221 \\
 4 & 6 & 22 & 56 & 339 \\
 0 & 0 & 4 & 0 & 52 } & n=5
  \end{array}
\qquad\qquad \begin{array}{cc}
  {} & {}  \\
   \smat{9 & 12 & 3 & 152 & 200 \\
 34 & 82 & 40 & 472 & 2473 \\
 25 & 58 & 72 & 473 & 6929 \\
 12 & 48 & 49 & 79 & 3493 } & n=6 \\
      \\[\smallskipamount]
    \smat{ 35& 35& 18& 506& 1600\cr 60& 75& 73& 2169& 27038\cr 53& 182& 421& 3272& 120991\cr 12& 60& 353& 1397& 137616\cr 0& 48& 116& 0& 17234}  & n=7 \\
       \\[\smallskipamount]
     \smat{39& 105& 29& 1756& 12685\cr 160& 1165& 514& 9533& 284487\cr
199& 1529& 1571& 20024& 1937923\cr 194& 2921& 2456& 15177& 4078227\cr  36& 686& 1242& 2092& 1764648} & n=8 
  \end{array}
 \end{array}
 \]
 \end{center}
 
and in genus 0, 
$$ n=9 \qquad (x_{rm},y_{rm},z_{rm},v_{rm},w_{rm})=\smat{124& 328& 195& 5980& 99794}\,, $$
$$n=10\qquad (x_{rm},y_{rm},z_{rm},v_{rm},w_{rm})=\smat{98& 969& 247& 20681& 801837}\,.$$ 
\def\mapover#1{\buildrel #1 \over{\mapsto}}

%%%%%%%%%%%%%%%%%%%%%%%%%%%%%%%%%%%%%%%%%%%%%%%%%%%%%%

\section*{Appendix C}
The following appendix recalls how certain integrals over real, complex or matrix variables enable one,
through their Feynman diagram interpretation, to construct generating functions of maps and in some cases, 
to compute the cardinals of some classes of maps.

%%%%%%%%%
\subsection*{C.1 The diagrammatic expansion of matrix integrals}

Let us consider the integral over a set of $f$ $N\times N$ Hermitian matrices $M_a$, $a=1,\cdots, f$
\be\label{mat-int} 
Z_X=\int\Big(\prod_{a=1}^f  DM_a\Big)\, \exp-N  [ \inv{2} \sum_{a=1}^f  \tr(M_a)^2 -\frac{\Gg}{4}\sum_{a,b=1}^f \tr(M_a M_b)^2]
\ee
(initially defined for $\Re \Gg \le 0$ and implicitly normalized by dividing by the Gaussian integral at $\Gg=0$). 
The measure $DM$ is the natural integration measure over Hermitian matrices, $DM=\prod_{i=1}^N {\rm d} M_{ii} 
\prod_{i<j}  {\rm d} \Re M_{ij}\,  {\rm d} \Im M_{ij}$.  
The integrand and the measure are clearly invariant under orthogonal 
transformations of the $M$'s, $M_a\mapsto M'_a=\sum_{a'=1}^fO_{aa'} M_{a'}$,
 $O\in {\rm O}(f)$.

We are mostly interested in the series expansion in powers of $\Gg$ of 
$F=
\log Z_X$\,, 
the ``free energy" in physicists'  parlance,
\be\label{free-en}  F=\sum_{n=1}^\infty \gamma^n F_n \,.\ee
 This expansion may be obtained by diagrammatic rules, in terms of 4-valent connected maps. 
As is well known since 't Hooft \cite{tH}, it is fruitful to represent Feynman diagrams arising from the expansion 
of $F$ with double lines, associated with the matrix indices of the  $M$'s;  
for reviews, see\cite{BIZ,Zv,Bouttier11}. The resulting diagrams, sometimes
called ``fat graphs", are in fact maps in the combinatorial sense; moreover, here,
each line across a vertex is decorated with an index $a$ (or $b$) running over $f$ values, referred to as 
``flavor". This flavor will enable us to identify the number of components when we regard the map 
as a multi-component curve or an alternating  link or knot diagram.

 \begin{figure}[htbp]
 \centering
\raisebox{-2ex}{\includegraphics[width=6pc]{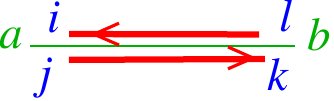}}$\ =\ \inv{N}\delta_{ab}\delta_{il}\delta_{jk}$\qquad 
\raisebox{-4ex}{\includegraphics[width=4pc,  height = 4pc]{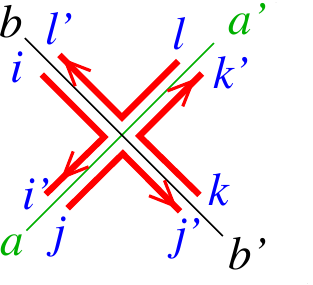}}$\ = \gamma N\delta_{aa'} \delta_{bb'} \delta_{ii'} \delta_{jj'} \delta_{kk'} \delta_{ll'} $. 
\caption{\label{vertex} Feynman rules}
\end{figure}

The diagrammatic rules are the following:  for a given map, to each vertex, assign a weight $\frac{\Gg}{4}N$; to 
each ``component", assign a weight $f$ (arising  from the summation over the running index $a$); 
to each ``index loop", \ie each face of the map,
assign a weight $N$ (arising from the summation over  matrix indices $i,j=1,\cdots N$);
and to each edge, a factor $N^{-1}$.
Each map then carries a power of $N$ equal to the Euler characteristics of the closed compact orientable
Riemann surface spanned by its faces, namely $N^{2-2g}$.

If $F$ in (\ref{free-en}) is written as 
\be F=\sum_{\g\ge 0}\sum_{c\ge 1}  N^{2-2\g} f^c\, \sum_{n\ge 1}  \Gg^n F_n^{(\g,c)}\ee
then 
$ $ is by the previous rules
the product of $\inv{n!} \(\frac{1}{4}\)^n$ times the number of {\it labelled maps} with genus $g$, $n$ vertices and $c$ components.
In other words  $F^{(\g,c)}(\Gg):= \sum_n \Gg^n F_n^{(\g,c)}$ is  the exponential generating function of labelled maps of given genus $g$ and number $c$ of
components, and with $n$ vertices. 
In the present paper, we are focusing on one-component diagrams, 
whose generating function is
\be\label{F1c} F^{[{\rm 1c}]}= f \sum_{n\ge 1}\Gg^n \sum_{\g\ge 0}   N^{2-2\g}\,  F_n^{(\g,1)}\,.\ee
In the formalism of Sect.\ \ref{Xpmethod}, 
\be F^{(g,1)}(\Gg)= \sum_{n\ge 1} \inv{n!} \(\frac{\Gg}{4}\)^n 
\#\{\tau \ {\rm satisfying\ (I) \ and\ (II){}_g}\}\ee
hence
\be\label{Fng1} F_n^{(g,1)}= \inv{4^n n! } |X'_{g {n}} | =\sum_{\CC_\sigma{\rm -orbits}\ o\atop {\rm of}\ X'_g}  \frac{\ell_o}{4^n n! }
 = \sum_{{\rm orbits}\ o\atop {\rm of}\ X'_g}   \inv{d_o}\ee
 with now a sum over {\it $\CC_\sigma$-orbits} $o$, \ie {\it unlabelled} maps, of length $\ell_o$. 
Thus $d_o=\frac{4^n n!}{\ell_0}=\frac{|\CC_\sigma|}{\ell_o}$, the ``symmetry factor" in Feynman rules, is the order of the stabilizer group of the orbit $o$. 
As an independent argument shows, (see for example Sect.\ \ref{gaugefixing}.c),  
$d_o$ turns out to be a divisor of $2n$.

For genus $g=0$  (planar maps), the first terms of the series expansion  (\ref{F1c}) read
\begin{equation}
\begin{split}
{ \inv{f N^2} F^{[{\rm pl,1c}]}}:=\sum_n \Gg^n F_n^{(0,1)}
     =& \inv{4}{2} \Gg + \inv{4^2 2!} {32} \Gg^2 + \inv{4^3 3!} {1344}  \Gg^3 + \inv{4^4 4!} {99\,840} \Gg^4\\
     {}&+\inv{4^5 5!}{11\,034\,624}\Gg^5+\cdots
        \end{split}
     \end{equation}
      and more terms appear in Tables \ref{TableX} and \ref{TableY}.
    
 Unfortunately, there exists  no closed formula for this series, in contrast with the cases $f=1$ for which we have Tutte's result (\cite{Tutte},
 see also  \cite{BIPZ} or equ.\ (3.9) of \cite{ZJZ}).
    \be 
  \lim_{N\to \infty} \inv{N^2}   F^{[{\rm pl}]}=\sum_{n\ge 1}\Gg^n  \sum_{c\ge 1}   \, F_n^{(0,c)}   = \sum_{n=1}^\infty (3\Gg)^n \frac{(2n-1)!}{n! (n+2)!}
      \ee
     or with the sum over all genera of one-component maps (\ie the case $N=1$ of (\ref{F1c}))
    \be   F^{[{\rm1c}]}\Big|_{N=1}   =f\sum_{n=1}^\infty 
   \(\frac{\Gg}{4}\)^n  \frac{(4n-2)!!}{n!}\qquad \hbox{see below equ.\ (\ref{F1cb} )}   
   \ee

\paragraph{Two other matrix integrals.}
The reader will convince him/her-self that the case of bi{\colour}able curves or of alternating knots and links is related 
in the same way to another matrix integral,
\be\label{matrix-int-2} 
Z_Y=\int\Big(\prod_{a=1}^f  D(M_a, M_a^\dagger)\Big)\, \exp-N  [  \sum_{a=1}^f  \tr(M_a M_a^\dagger) -
\frac{\Gg}{4}\sum_{a,b=1}^f \tr(M_a M_b^\dagger)^2]
\ee
with now an integration over complex $N\times N$ matrices. 
{}From the fact that any {\it planar} map may be bi{\colour}ed in two different ways, it follows that 
the free energy $\log Z_Y$   coincides up to a factor 2  with $F=\log Z_X$ considered above.
Thus, using the formalism of Sect. \ref{Ypmethod}, 
\bea F_n^{(0,1)}&=&{\frac{1}{2 (2n)!}} \#\{\sigma,\tau\in S_{2n} \,|\, \rho\in[2^n] \cap {\rm(I')} \,\cap\, {\rm (II')}_0 \} \\
&=&(2n-1)!! \frac{1}{2 (2n)!} \#\{\sigma\in S_{2n} \,|\,  {\rm(I')} \,\cap\, {\rm (II')}_0\  {\rm with} \ \rho=\rho_0\,,\ 
\tau=\sigma \rho \}\,,
\eea
where in the first line the factor $2 (2n)!$ comes from the 
two possible bi{\colour}ations along with a general relabelling of the $2n$ edges, and in the second, the factor $(2n-1)!!$ comes
 from the possible choices of $\rho$, (pairings at vertices).
    The bottom line of Table \ref{TableY} is that $F_n^{(0,1)}=\inv{2^{n+1}n!} \#\{\sigma\cdots\}$. (Fortunately 
    the results coincide with those of Table \ref{TableX} !)

Finally the counting of general oriented curves in Sect.\ \ref{Zpmethod} is related to the following integral
\be\label{matrix-int-3} 
Z_Z=\int\Big(\prod_{a=1}^f  D(M_a, M_a^\dagger)\Big)\, \exp-N  [  \sum_{a=1}^f  \tr(M_a M_a^\dagger) -
\frac{\Gg}{4}\sum_{a,b=1}^f \tr(M_a M_b  M_a^\dagger M_b^\dagger)]\,.
\ee
For {\it one-component} maps, there are two ways of orienting the corresponding curve, hence the 
free energy $F_Z^{[{\rm 1c}]}=\log Z_Z\Big|_{{\rm term}\ f^1}$   coincides up to a factor 2  with $F^{}$ considered above, 
for any genus.

%%%%%%%%%%%%%%%%%%%%%%%%%%%%%%%%%%%%%%%%%%%%%%%%%%%
\subsection*{C.2 The cardinal of $X'$ through a simple integral}
\label{simple-int}
In this section we compute the cardinal of the set $X'$ of Sect.\ \ref{Ypmethod} through a simple integral. Recall
that $X'$ gathers  maps of all genera. 
Since we are not concerned by the genus of the graph/map, we may use an integration over real vectors
$\phi$ of $\R^f$ rather than matrices, \ie the case $N=1$ of the integral (\ref{mat-int}). Let
\be Z= (2\pi)^{-f/2} \int d^f\!\phi\, \exp[-\oh \phi^2+\frac{\gamma}{4}(\phi^2)^2] \ee
in which the terms linear in $f$ yield the contribution of one-component graphs.
As above, we assume that $\Re \gamma <0$ and we have explicitly normalized $Z$ to be 1 for 
$\gamma=0$. Following a standard trick, we rewrite $Z$ as
\be Z = 
\int_\R \frac{d\alpha}{\sqrt{\pi}} e^{-\alpha^2} \, \int_{\R^f}\,\frac{d^f \!\phi}{(2\pi)^{f/2}} \, \exp[-\oh \phi^2 (1+2 i \sqrt{-\Gg} \alpha)]\,.\ee
 Integrating over the $f$-dimensional $\phi$ gives
\be Z =\int_\R \frac{d\alpha}{\sqrt{\pi}} e^{-\alpha^2}  (1+2i \sqrt{-\Gg}\alpha)^{-f/2}\,.\ee 
In the series expansion of the term  $(1+2i \sqrt{-\Gg}\alpha)^{-f/2}$, 
we keep only the term of order $f^1$, hence
\be  Z\Big|_{f\ {\rm term}}  = 
\frac{f}{2} \sum_{n\ge 1}  \frac{(2i\sqrt{-\Gg})^n \langle  \alpha^n\rangle}{n}\ee
where $\langle \alpha^m \rangle$ denote  the moments of the Gaussian measure  
$\frac{d\alpha}{\sqrt{\pi}} e^{-\alpha^2}$.
Only even moments are non vanishing 
and we find
\be  Z\Big|_{f\ {\rm term}} = f \sum_{n=1} \frac{ (2\gamma)^n}{4n} (2n-1)!! \ee
which may be recast as 
\be\label{F1cb}   Z\Big|_{f\ {\rm term}} =F^{[{\rm1c}]}\Big|_{N=1}  
=f \sum_{n=1} \inv{n!} \(\frac{\gamma}{4}\)^n  (4n-2)!!\,. \ee
By comparing this calculation with formula (\ref{Fng1}) one sees that the coefficient $(4n-2)!!$ is nothing else 
than the number of points in the set $X^\prime$. See also next Appendix for a direct combinatorial argument. 

%%%%%%%%%%%%%%%

\subsection*{C.3 The set $X'\subset [2^{2n}]$}\label{setXp}
In this section, we reproduce the previous result on $\card{X'}$ by a purely combinatorial argument. 
The set $X^\prime$ of permutations $\tau\in [2^{2n}]$  that satisfy
$\sigma^2\,\tau\in [(2n)^2]$  may be constructed explicitly. We choose $\sigma=(1324)\cdots (4n-3,4n-1,4n-2,4n)$ so that 
$\sigma^2=(12)(34)\cdots (4n-1,4n)$. We note that $i_1:=\tau(1)$ is different from 1 (since $\tau\in[2^{2n}]$), and
from $2=\sigma^2(1)$, otherwise $\sigma^2\,\tau$ would have a 1-cycle. We thus have $4n-2$ possible choices for $i_1$.
\\
By recursion, suppose that after $r< 2n$ iterations, we  choose 
$ i_r :=\tau(\sigma^2(i_{r-1}))$ different from 
$ i_0:=2, i_1,\cdots, i_{r-1}:=\tau(\sigma^2(i_{r-2}))$ and from their images by $\sigma^2$, 
with these $2r$ numbers assumed to be all different: we thus have $4n-2r$ choices for $i_r$. 
\\
Then for any $0\le s\le r-1$,
$\sigma^2(i_r) \ne \sigma^2(i_s)$  since $i_r\ne i_s$; and $\sigma^2(i_r) \ne i_s \Leftrightarrow i_r\ne \sigma^2(i_s)$ by
assumption. \\
Moreover $i_{r+1}:=\tau\,\sigma^2(i_r)=(\tau\,\sigma^2)^{r+1}(2)$ must be different from the $2(r+1)$ numbers 
$ i_0:=2, i_1,\cdots, i_{r}$ and their images by $\sigma^2$:
\bee 
\item   
for $0\le s \le r-1$, $i_{r+1}=\tau\,\sigma^2(i_r) \ne \sigma^2(i_s) \Longleftrightarrow \sigma^2(i_r) \ne \tau\sigma^2(i_s)=
i_{s+1} \Longleftrightarrow i_r\ne \sigma^2(i_{s+1})$ by the assumption on $i_r$ for $s<r-1$, and the fact that $\sigma^2$ has no
fixed point for $s=r-1$;
\item $i_{r+1}=\tau\,\sigma^2(i_r) \ne \sigma^2(i_r)$ since $\tau$ has no  fixed point;
\item for $1\le s \le r$, $i_{r+1}=\tau\,\sigma^2(i_r) \ne i_s=\tau(\sigma^2(i_{s-1}))$ since $i_r\ne i_{s-1}$;
\item finally,  $i_{r+1}=\tau\,\sigma^2(i_r)=(\tau\,\sigma^2)^{r+1}(2)$ may be equal to $i_0=2$  iff 
2 is a fixed point of $(\tau\,\sigma^2)^{r+1}$, which occurs iff $r+1=2n$ (remember that $\tau\,\sigma^2\in [(2n)^2]$). 
\eee
Hence, for $r<2n-1$, the recursion assumption is verified, and there are $4n-2(r+1)$ choices for $i_{r+1}$. 
At the end of this iterateve procedure we have constructed a 
$\tau=((1,i_1),(\sigma^2(i_1),i_2),\cdots, (\sigma^2(i_{2n-1}),i_{2n})$, and all $\tau\in X'$ are obtained that way. 
This completes the construction of the set $X^\prime$ and the proof that $\card{ X'}=\prod_{r=0}^{2n-2} (4n-2(r+1))= (4n-2)!!$.  

%%%%%%%%%%%%%%%

\subsection*{C.4 The set $Y'\subset S_{2n}$}
\label{thesetYp}
By lack of a direct combinatorial construction of the set $Y'$ (as we had for $X'$, see previous appendix), we resort 
again to a simple integral to compute the cardinality of $Y'$.
\def\bz{\bar z}\def\ba{\bar \alpha}
In the same spirit as in App.C.2, let us consider the integral over  vectors of $\C^f$
\be Z=\inv{\pi^f}\int d^f(z,\bz) \exp[-z\cdot \bz +\frac{\gamma}{4} z^2 \bz^2]\ee
where $z\cdot \bz=\sum_{a=1}^f z_a \bz_a$, $z^2:=\sum_{a=1}^f (z_a)^2$ and likewise for  $\bz^2$.
Note that this may be regarded as the $N=1$ version of integral (\ref{matrix-int-2}). 
We take $\gamma<0$ to ensure convergence.
Using again the same trick, we rewrite $Z$, up to a factor, as
\be Z= \int \frac{d(\alpha,\ba)}{\pi} e^{-\alpha\ba}\,
\int \frac{d^f(z,\bz)}{\pi^f} \exp[-z\cdot\bz -
-i\sqrt{\frac{\Gg}{4}} z^2 \ba -i \sqrt{\frac{\Gg}{4}}\bz^2 \alpha]\ee
which upon integration over $z,\bz$ gives
\be Z= \int \frac{d(\alpha,\ba)}{\pi} e^{-\alpha\ba} (1-\Gg \alpha\ba)^{-f/2}\ee
Keeping again the term of order $f^1$ in the expansion of $(1-\Gg \alpha\ba)^{-f/2}$ gives
\be Z\Big|_{f\ {\rm term}}=
 \frac{f}{2} \sum_{n\ge 1} \frac{\langle(\alpha\ba)^n\rangle}{n} \ee
 where  $\langle(\alpha\ba)^n\rangle=n!$  are the moments of
  the measure   $\frac{d(\alpha,\ba)}{\pi} e^{-\alpha\ba}$, 
hence 
\be Z\Big|_{f\ {\rm term}}=f\sum_{n=1} \inv{n!} \Big(\frac{\gamma}{4}\Big)^n  2^{2n-1} n! (n-1)! \ee
which (in view of the diagrammatic interpretation \`a la Feynman of this computation) shows 
that the number of points in the set $Y'$ is indeed $2^{2n-1} n! (n-1)! =(2n)!! (2n-2)!!$. 
%\vskip1cm

As a little exercise left to the reader, one may check that the same reasoning applied to integral
(\ref{matrix-int-3}), \ie consideration of the integral
\be Z=\int d^f(z,\bz) \exp[-z\cdot \bz +\frac{\gamma}{4} (z\cdot\bz)^2]\ee
and computation of its  $f^1$ term 
will reproduce the counting of points in the set $Z'$ of Sect. \ref{Zpmethod}, namely $(2n-1)!$.

%%%%%%%%%%%%%%%%%%%%%%%%%%%%%%%%%%%%%%%%%
\section*{Acknowledgements} 
We acknowledge useful discussions with Alina Vdovina and Paul Zinn-Justin 
and a very stimulating correspondence with Guy Valette.  
Partial support from the EPLANET network and from IMPA (Rio de Janeiro) is also acknowledged.
%\vfill \eject
%%%%%%%%%%%%%%%%%%%%%%%%%%%%%%%%%%%%%%%%%%%%%%%%%%%%%%

   %%%%%%%%%%%%%%%%%%%%%

\vfill \eject
   
 \end{document}